\documentclass[10pt]{article}
\usepackage{amssymb,amsmath}% font used for R in Real numbers
\usepackage{mathrsfs}
\usepackage{latexsym}
\usepackage{amssymb}
\usepackage{amsmath}
\usepackage{amsthm}
\allowdisplaybreaks

\catcode`!=11
\let\!int\int \def\int{\displaystyle\!int}
\let\!lim\lim \def\lim{\displaystyle\!lim}
\let\!sum\sum \def\sum{\displaystyle\!sum}
\let\!sup\sup \def\sup{\displaystyle\!sup}
\let\!inf\inf \def\inf{\displaystyle\!inf}
\let\!cap\cap \def\cap{\displaystyle\!cap}
\let\!max\max \def\max{\displaystyle\!max}
\let\!min\min \def\min{\displaystyle\!min}
\catcode`!=12

\newtheorem{theorem}{\bf Theorem}[section]
\newtheorem{lemma}{\bf Lemma}[section]
\newtheorem{definition}{\bf Definition}[section]
\newtheorem{proposition}{\bf Proposition}[section]
\newtheorem{corollary}{\bf Corollary}[section]
\newtheorem{remark}{\bf Remark}[section]

\catcode`@=11
\@addtoreset{equation}{section}
\catcode`@=12

\begin{document}

\title{Carleman estimates for forward and backward
stochastic fourth order Schr\"{o}dinger equations and their applications }
\author{Peng Gao
\\[2mm]
\\{ School of Mathematics and Statistics, and Center for Mathematics}
\\{and Interdisciplinary Sciences, Northeast Normal University, }
\\{Changchun 130024,  P. R. China}
\\[2mm]
\small Email: gaopengjilindaxue@126.com }
\date{}
\maketitle

\vbox to -13truemm{}

\begin{abstract}
In this paper, we establish the Carleman estimates for forward and backward
stochastic fourth order Schr\"{o}dinger equations, on basis of which, we can obtain the observability, unique continuation property and the exact controllability for the
forward and backward
stochastic fourth order Schr\"{o}dinger equations.
\\[6pt]
\end{abstract}
\section{Introduction}
\par
The fourth order Schr\"{o}dinger equation reads as
$$iy_{t}+y_{xxxx}=0.$$
This equation
arises in quantum mechanics, nonlinear
optics, and plasma physics, and its general nonlinear form
$$iy_{t}+\frac{1}{2}y_{xx}+y_{xxxx}+|y|^{2p}y=0$$
has been introduced in
\cite{K1,K2} to take into account the role of small fourth order dispersion terms in the
propagation of intense laser beams in a bulk medium with Kerr nonlinearity, where $p\geq 1$ is an integer number. The
existence and uniqueness of the solution have been studied intensively from the mathematical
perspective; see \cite{H1,H2,P1,P2} and the references therein.
Similar to its deterministic counterpart, the stochastic fourth order Schr\"{o}dinger equation plays
an important role in quantum mechanics.
\par
The main purpose of this paper is to estiablish
Carleman estimates for forward and backward
stochastic fourth order Schr\"{o}dinger equations.
Carleman estimate is an $L^{2}$-weighted estimate with large
parameter for a solution to a partial differential equation (PDE). Carleman estimate was first
established by Carleman \cite{C3} for a two-dimensional elliptic
equation. It is an important tool for the study of unique continuation
property, stabilization, controllability and inverse problems for PDEs. Although there are numerous results
for the Carleman estimate for deterministic PDEs, very little
is known about the corresponding stochastic situation. The Carleman estimates for
stochastic heat equation, stochastic wave equation, stochastic 
Korteweg-de Vries equation, stochastic Kuramoto-Sivashinsky equation, stochastic Kawahara equation and stochastic second order Schr\"{o}dinger equation
were completed (see, for instance, \cite{2,3,4,5,6,66,666,7}). But nothing is known for stochastic fourth order Schr\"{o}dinger equation. To the knowledge
of the authors, the Carleman estimates in this paper are new, it is the first
attempt for forward and backward stochastic fourth order Schr\"{o}dinger equation.
The Carleman estimate for deterministic fourth order Schr\"{o}dinger equation has been estiablished in \cite{Z2}.
\par
Through this paper, we make the following assumptions:
\par
(H1) Let $T>0, I=(0,1)$ and $I_{0}$ be a nonempty open subset of $I.$ Set $Q=I\times (0,T)$ and $Q^{I_{0}}=I_{0}\times (0,T).$

\par
(H2)
Let $(\Omega,\mathcal {F},\{\mathcal {F}_{t}\}_{t\geq0},P)$ be a
complete filtered probability space on which a one-dimensional
standard Brownian motion $\{w(t)\}_{t\geq 0}$ is defined such that
$\{\mathcal {F}_{t}\}_{t\geq0}$ is the natural filtration generated
by $w(\cdot),$ augmented by all the $P$-null sets in $\mathcal
{F}.$ Let $H$ be a Banach space, and let $C([0,T];H)$ be the Banach
space of all $H-$valued strongly continuous functions defined on
$[0,T].$ We denote by $L_{\mathcal {F}}^{p}(0,T;H)(1\leq p<+\infty)$ the Banach space
consisting of all $H-$valued $\{\mathcal {F}_{t}\}_{t\geq0}-$adapted
processes $X(\cdot)$ such that
$E(\|X(\cdot)\|^{p}_{L^{p}(0,T;H)})<\infty;$ by
$L_{\mathcal {F}}^{\infty}(0,T;H)$ the Banach space consisting of
all $H-$valued ${\mathcal \{\mathcal {F}_{t}}\}_{t\geq0}-$adapted bounded
processes; by $L_{\mathcal {F}}^{2}(\Omega;C([0,T];H))$ the
Banach space consisting of all $H-$valued $\{\mathcal
{F}_{t}\}_{t\geq0}-$adapted continuous processes $X(\cdot)$ such
that $E(\|X(\cdot)\|^{2}_{C([0,T];H)})<\infty;$ and by $C_{\mathcal {F}}([0,T];L^{2}(\Omega;H))$ the
Banach space consisting of all $H-$valued $\{\mathcal
{F}_{t}\}_{t\geq0}-$adapted processes $X(\cdot)$ such
that $\|X(\cdot)\|_{L^{2}(\Omega;H)} \in C([0,T]).$ All the
above spaces are endowed with the canonical norm.
\par
(H3) We denote by $L^{2}(I)$ the space of all Lebesgue square integrable
complex-valued functions on $I$. The inner product on $L^{2}(I)$ is
\begin{eqnarray*}
\langle u,v\rangle=\int_{I}u\overline{v}dx,
\end{eqnarray*}
for any $u,v\in L^{2}(I),$ where $\overline{\bullet}$ denotes the conjugate of $\bullet.$ The norm on $L^{2}(I)$ is
\begin{eqnarray*}
\|u\|_{L^{2}(I)}=\langle u,u\rangle^{\frac{1}{2}},
\end{eqnarray*}
for any $u\in L^{2}(I).$
\par
$H^{s}(I)(s\geq 0)$ are the classical Sobolev
spaces of complex-valued functions on $I$. The definition of $H^{s}(I)$ can be found in \cite{L1}.
\par For $0\leq s\leq4,$ the $s$-compatibility conditions are following:
\begin{eqnarray*}
\begin{array}{l}
\left\{
\begin{array}{llll}
 \varphi(0)=\varphi(1)=0 &\rm{when}& \frac{1}{2}<s\leq\frac{3}{2},\\
 \varphi(0)=\varphi(1)=\varphi^{\prime}(0)=\varphi^{\prime}(1)=0 &\rm{when}& \frac{3}{2}<s\leq 4.\\
\end{array}
\right.
\end{array}
\end{eqnarray*}
Set
\begin{eqnarray*}
X_{s}&=&\{\varphi\in H^{s}(I)~|~\varphi ~~ \rm{satisfies~~ the
~~}\emph{s}\rm{-compatibility ~~conditions}\},
\end{eqnarray*}
the norms on $X_{s}$ is defined by
$\|\varphi\|_{X_{s}}\triangleq \|\varphi\|_{H^{s}(I)},$ for any
$\varphi \in X_{s}.$
\par
The space $X_{s}^{\prime}$ denotes the dual space of $X_{s}$ with respect to the space $L^{2}(I),$ the $(\cdot,\cdot)_{X_{s}^{\prime},X_{s}}$ denotes the duality pairing between $X_{s}^{\prime}$ and $X_{s}.$
\par
(H4) Let $\widehat{\psi} \in
C^{\infty}(\overline{I})$ satisfy that $\widehat{\psi} >0$ in $I,$
$\widehat{\psi}(0)=\widehat{\psi}(1)=0,$ $\|\widehat{\psi}\|_{C(\overline{I})}=1,$ $|\widehat{\psi}_{x}|>0$
in $\overline{I}\backslash I_{0},$ $\widehat{\psi}_{x}(0)>0$ and
$\widehat{\psi}_{x}(1)<0.$ For any given positive constants $\lambda$ and
$\mu,$ we set $\widehat{a}(x,t)=\frac{e^{\mu(\widehat{\psi}(x)+3)}-e^{5\mu}}{t(T-t)},$
$\widehat{l}=\lambda \widehat{a},$ $\widehat{\theta}=e^{\widehat{l}}$ and
$\widehat{\varphi}(x,t)=\frac{e^{\mu(\widehat{\psi}(x)+3)}}{t(T-t)},\forall (x,t)\in Q.$
\par
Let $\widetilde{\psi}(x)=(x-x_{0})^{2}+\delta_{0},$ where $\delta_{0}$ is a positive constant such that $\widetilde{\psi}\geq \frac{3}{4}\|\widetilde{\psi}\|_{L^{\infty}(I)}$ and $x_{0}>1$. For any given positive constants $\lambda$ and
$\mu,$ we set $\widetilde{a}(x,t)=\frac{e^{\mu\widetilde{\psi}(x)}-e^{\frac{3}{2}\|\widetilde{\psi}\|_{L^{\infty}(I)}\mu}}{t(T-t)},$
$\widetilde{l}=\lambda \widetilde{a},$ $\widetilde{\theta}=e^{\widetilde{l}}$ and
$\widetilde{\varphi}(x,t)=\frac{e^{\mu\widetilde{\psi}(x)}}{t(T-t)},\forall (x,t)\in Q.$

\par
(H5) Unless otherwise stated, $C$ stands for a generic positive
constant whose value can change from line to line. Whenever necessary, the dependence of a constant $C$ on some parameters, say $``\cdot"$ ,
will be written by $C(\cdot)$.
\par
(H6) $a,b\in L^{\infty}_{\mathcal {F}}(0,T;W^{4,\infty}(I)).$
\subsection{Carleman estimates for a forward
stochastic fourth order Schr\"{o}dinger equation and their applications}
\par
In this section, we consider the following system
\begin{eqnarray}\label{1}
\begin{array}{l}
\left\{
\begin{array}{lll}idy+y_{xxxx}dt=fdt+gdw\\
y(0,t)=0=y(1,t)\\
y_{x}(0,t)=0=y_{x}(1,t)\\
y(x,0)=y_{0}(x)
\end{array}
\right.
\end{array}
\begin{array}{lll}\textrm{in}~Q,
\\\textrm{in}~(0,T),\\\textrm{in}~(0,T),\\\textrm{in}~I.
\end{array}
\end{eqnarray}
\par
First, we establish the following global Carleman estimate.
\begin{theorem}\label{T1}
Let $y_{0}\in X_{3}$ and $f,g\in L^{2}_{\mathcal{F}}(0,T;X_{3})$ be given. There exist $\lambda_{0},\mu_{0}$ and $C$ such that for any $\lambda\geq \lambda_{0},\mu\geq \mu_{0}$ and any solution $y$ of (\ref{1}),
it holds that
\begin{eqnarray}\label{7}
\begin{array}{l}
\begin{array}{llll}
~~~\displaystyle E\int_{Q}(\lambda\widehat{\varphi}
\widehat{\theta}^{2}|y_{xxx}|^{2}+\lambda^{3}\widehat{\varphi}^{3}\widehat{\theta}^{2}|y_{xx}|^{2}+\lambda^{5}\widehat{\varphi}^{5}\widehat{\theta}^{2}|y_{x}|^{2}+\lambda^{7}\widehat{\varphi}^{7}\widehat{\theta}^{2}|y|^{2})dxdt
\\\leq\displaystyle
C\Big[E\int_{Q^{I_{0}}}(\lambda\widehat{\varphi}
\widehat{\theta}^{2}|y_{xxx}|^{2}+\lambda^{7}\widehat{\varphi}^{7}\widehat{\theta}^{2}|y|^{2})dxdt
\\~~~~~~~~~~+\displaystyle E\int_{Q}(\lambda^{4}\widehat{\varphi}^{4}\widehat{\theta}^{2}|g|^{2}
+\lambda^{2}\widehat{\varphi}^{2}\widehat{\theta}^{2}|g_{x}|^{2}+\widehat{\theta}^{2}|g_{xx}|^{2}+\widehat{\theta}^{2}|f|^{2})dxdt\Big].
\end{array}
\end{array}
\end{eqnarray}
\end{theorem}
\par
We give some applications of Theorem \ref{T1}.
First, we can obtain the following observability inequality.
\begin{corollary}\label{C4}
Let $y_{0}\in X_{3},f,g\in L_{\mathcal{F}}^{2}(0,T;X_{3})$ and $y$ be the solution of
\begin{eqnarray}\label{1.49}
\begin{array}{l}
\left\{
\begin{array}{lll}idy+y_{xxxx}dt=(ay+f)dt+(by+g)dw\\
y(0,t)=0= y(1,t)\\
y_{x}(0,t)=0=y_{x}(1,t)\\
y(x,0)=y_{0}(x)
\end{array}
\right.
\end{array}
\begin{array}{lll}\textrm{in}~Q,
\\\textrm{in}~(0,T),\\\textrm{in}~(0,T),\\\textrm{in}~I,
\end{array}
\end{eqnarray} we have
\begin{eqnarray}\label{1.40}
\begin{array}{l}
\begin{array}{llll}
\displaystyle\|y_{0}\|_{X_{3}}
\leq C\Big[\|y\|_{L_{\mathcal{F}}^{2}(0,T;L^{2}(I_{0}))}+\|y_{xxx}\|_{L_{\mathcal{F}}^{2}(0,T;L^{2}(I_{0}))} \\~~~~~~~~~~~~+\|f\|_{L_{\mathcal{F}}^{2}(0,T;X_{3})}+\|g\|_{L_{\mathcal{F}}^{2}(0,T;X_{3})}\Big],
\end{array}
\end{array}
\end{eqnarray}
where $C=C(a,b,T).$
\end{corollary}
\begin{remark} Note that the observability inequality (\ref{1.40}) can be applied to the state observation problems for semilinear
stochastic fourth order Schr\"{o}dinger equation. This is similar to Section 6 in \cite{0}.
\end{remark}
Also, we can obtain the following unique continuation property.
\begin{corollary}\label{C1}
Let $y_{0}\in X_{3},$ $y$ is the solution of
\begin{eqnarray}\label{1.7}
\begin{array}{l}
\left\{
\begin{array}{lll}idy+y_{xxxx}dt=aydt+bydw\\
y(0,t)=0=y(1,t)\\
y_{x}(0,t)=0=y_{x}(1,t)\\
y(x,0)=y_{0}(x)
\end{array}
\right.
\end{array}
\begin{array}{lll}\textrm{in}~Q,
\\\textrm{in}~(0,T),\\\textrm{in}~(0,T),\\\textrm{in}~I.
\end{array}
\end{eqnarray}
If $$y\equiv0~~ {\rm{in}}~~ Q^{I_{0}},~P-{\rm{a.s.}},$$
we have $y\equiv0 ~{\rm{in}}~ Q,$ $P$-a.s.
\end{corollary}
\begin{remark} The classical Holmgren Uniqueness Theorem does not work for stochastic PDEs.
\end{remark}

Next, we establish another type of global Carleman estimate.
\begin{theorem}\label{T3}
Let $y_{0}\in X_{3}$ and $f,g\in L^{2}_{\mathcal{F}}(0,T;X_{3})$ be given. There exist $\lambda_{0},\mu_{0}$ and $C$ such that for any $\lambda\geq \lambda_{0},\mu\geq \mu_{0}$ and any solution $y$ of (\ref{1}),
it holds that
\begin{eqnarray}\label{1.21}
\begin{array}{l}
\begin{array}{llll}
~~~\displaystyle E\int_{Q}(\lambda\widetilde{\varphi}
\widetilde{\theta}^{2}|y_{xxx}|^{2}+\lambda^{3}\widetilde{\varphi}^{3}\widetilde{\theta}^{2}|y_{xx}|^{2}+\lambda^{5}\widetilde{\varphi}^{5}\widetilde{\theta}^{2}|y_{x}|^{2}+\lambda^{7}\widetilde{\varphi}^{7}\widetilde{\theta}^{2}|y|^{2})dxdt
\\\leq\displaystyle
C\Big[E\int_{0}^{T}(\lambda\widetilde{\varphi}(0,t)
\widetilde{\theta}^{2}(0,t)|y_{xxx}(0,t)|^{2}+\lambda^{3}\widetilde{\varphi}^{3}(0,t)\widetilde{\theta}^{2}(0,t)|y_{xx}(0,t)|^{2}
)dt
\\~~~~~~~~~~+\displaystyle E\int_{Q}(\lambda^{4}\widetilde{\varphi}^{4}\widetilde{\theta}^{2}|g|^{2}
+\lambda^{2}\widetilde{\varphi}^{2}\widetilde{\theta}^{2}|g_{x}|^{2}+\widetilde{\theta}^{2}|g_{xx}|^{2}+\widetilde{\theta}^{2}|f|^{2})dxdt\Big].
\end{array}
\end{array}
\end{eqnarray}
\end{theorem}
\begin{remark}
It follows from hidden regularity property (Proposition \ref{P3}) that $y_{xx}(0,\cdot),y_{xx}(1,\cdot),y_{xxx}(0,\cdot),$ $y_{xxx}(1,\cdot) \in L^{2}_{\mathcal{F}}(\Omega,L^{2}(0,T)),$ thus the right-hand side of (\ref{1.21}) makes sense.
\end{remark}
\par
Now, we give two applications of Theorem \ref{T3}.
\begin{corollary}\label{C5}
Let $y_{0}\in X_{3},f,g\in L_{\mathcal{F}}^{2}(0,T;X_{3})$ and $y$ be the solution of (\ref{1.49}), we have
\begin{eqnarray}\label{1.43}
\begin{array}{l}
\begin{array}{llll}
\displaystyle\|y_{0}\|_{X_{3}}
\leq C\Big[\|y_{xx}(0,\cdot)\|_{L^{2}_{\mathcal{F}}(\Omega,L^{2}(0,T))}+\|y_{xxx}(0,\cdot)\|_{L^{2}_{\mathcal{F}}(\Omega,L^{2}(0,T))} \\~~~~~~~~~~~~+\|f\|_{L_{\mathcal{F}}^{2}(0,T;X_{3})}+\|g\|_{L_{\mathcal{F}}^{2}(0,T;X_{3})}\Big],
\end{array}
\end{array}
\end{eqnarray}
where $C=C(a,b,T).$
\end{corollary}
\begin{corollary}\label{C2}
Let $y_{0}\in X_{3}$ and $y$ be the solution of (\ref{1.7}). If $$y_{xx}(0,t)=y_{xxx}(0,t)\equiv0~~ {\rm{in}}~~ (0,T),~P-{\rm{a.s.}},$$
we have $y\equiv0 ~{\rm{in}}~ Q,~P-{\rm{a.s.}}$
\end{corollary}
\subsection{Carleman estimate for a backward
stochastic fourth order Schr\"{o}dinger equation and its applications}
\par
In this section, we first consider the backward stochastic fourth order Schr\"{o}dinger equation
\begin{eqnarray}\label{15}
\begin{array}{l}
\left\{
\begin{array}{lll}
idz+z_{xxxx}dt=hdt+Zdw\\
z(0,t)=0=z(1,t)\\
z_{x}(0,t)=0=z_{x}(1,t)\\
z(x,T)=z_{T}(x)
\end{array}
\right.
\end{array}
\begin{array}{lll}\textrm{in}~Q,
\\\textrm{in}~(0,T),\\\textrm{in}~(0,T),\\\textrm{in}~I.
\end{array}
\end{eqnarray}
By the same method in Proof of Theorem \ref{T3}, we can obtain
\begin{theorem}\label{T6}
Let $z_{T}\in L^{2}\mathcal (\Omega,\mathcal
{F}_{T},P; X_{3})$ and $h \in L^{2}_{\mathcal{F}}(0,T;X_{3})$ be given. There exist $\lambda_{0},\mu_{0}$ and $C$ such that for any $\lambda\geq \lambda_{0},\mu\geq \mu_{0}$ and any solution $(z,Z)$ of (\ref{15}),
it holds that
\begin{eqnarray*}
\begin{array}{l}
\begin{array}{llll}
~~~\displaystyle E\int_{Q}(\lambda\widetilde{\varphi}
\widetilde{\theta}^{2}|z_{xxx}|^{2}+\lambda^{3}\widetilde{\varphi}^{3}\widetilde{\theta}^{2}|z_{xx}|^{2}+\lambda^{5}\widetilde{\varphi}^{5}\widetilde{\theta}^{2}|z_{x}|^{2}+\lambda^{7}\widetilde{\varphi}^{7}\widetilde{\theta}^{2}|z|^{2})dxdt
\\\leq\displaystyle
C\Big[E\int_{0}^{T}(\lambda\widetilde{\varphi}(0,t)
\widetilde{\theta}^{2}(0,t)|z_{xxx}(0,t)|^{2}+\lambda^{3}\widetilde{\varphi}^{3}(0,t)\widetilde{\theta}^{2}(0,t)|z_{xx}(0,t)|^{2}
)dt
\\~~~~~~~~~~+\displaystyle E\int_{Q}(\lambda^{4}\widetilde{\varphi}^{4}\widetilde{\theta}^{2}|Z|^{2}
+\lambda^{2}\widetilde{\varphi}^{2}\widetilde{\theta}^{2}|Z_{x}|^{2}+\widetilde{\theta}^{2}|Z_{xx}|^{2}+\widetilde{\theta}^{2}|h|^{2})dxdt\Big].
\end{array}
\end{array}
\end{eqnarray*}
\end{theorem}
Now we consider the exact controllability of the following system:
\begin{eqnarray}\label{1.3}
\begin{array}{l}
\left\{
\begin{array}{lll}idy+y_{xxxx}dt=(ay+f)dt+(by+g)dw\\
y(0,t)=u_{1}(t),y(1,t)=0\\
y_{x}(0,t)=u_{2}(t),y_{x}(1,t)=0\\
y(x,0)=y_{0}(x)
\end{array}
\right.
\end{array}
\begin{array}{lll}\textrm{in}~Q,
\\\textrm{in}~(0,T),\\\textrm{in}~(0,T),\\\textrm{in}~I.\end{array}
\end{eqnarray}
\begin{definition}
System (\ref{1.3}) is said to be exactly controllable at time $T$ if for every initial state $y_{0}\in X_{3}^{\prime}$ and every $y_{1}\in L^{2}\mathcal (\Omega,\mathcal
{F}_{T},P; X_{3}^{\prime}),$ one can find controls
\begin{eqnarray*}
(u_{1},u_{2},g)\in L^{2}_{\mathcal{F}}(\Omega,L^{2}(0,T))\times L^{2}_{\mathcal{F}}(\Omega,L^{2}(0,T))\times L_{\mathcal {F}}^{2}(0,T;X_{3}^{\prime})
\end{eqnarray*}
 such that the solution of the system (\ref{1.3}) satisfies that $y(T)=y_{1}$ in $L^{2}\mathcal (\Omega,\mathcal
{F}_{T},P; X_{3}^{\prime}).$
\end{definition}
\par
In order to establish the exactly controllability of (\ref{1.3}), we introduce the dual system of (\ref{1.3})
\begin{eqnarray}\label{1.48}
\begin{array}{l}
\left\{
\begin{array}{lll}
idz+z_{xxxx}dt=(\overline{a}z-i\overline{b}Z)dt+Zdw\\
z(0,t)=0=z(1,t)\\
z_{x}(0,t)=0=z_{x}(1,t)\\
z(x,T)=z_{T}(x)
\end{array}
\right.
\end{array}
\begin{array}{lll}\textrm{in}~Q,
\\\textrm{in}~(0,T),\\\textrm{in}~(0,T),\\\textrm{in}~I.
\end{array}
\end{eqnarray}
By the same method in Proof of Corollary \ref{C5}, we can obtain
\begin{corollary}\label{C6}
Let $(z,Z)$ solve (\ref{1.48}) with the terminal state $z_{T}\in L^{2}\mathcal (\Omega,\mathcal
{F}_{T},P; X_{3}).$ Then we have
\begin{eqnarray}\label{1.29}
\begin{array}{l}
\begin{array}{llll}
\displaystyle\|z_{T}\|_{L^{2}\mathcal (\Omega,\mathcal{F}_{T},P; X_{3})}
\leq C\Big[\|z_{xx}(0,\cdot)\|_{L^{2}_{\mathcal{F}}(\Omega,L^{2}(0,T))}+\|z_{xxx}(0,\cdot)\|_{L^{2}_{\mathcal{F}}(\Omega,L^{2}(0,T))} \\~~~~~~~~~~~~~~~~~~~~~~~~~~~~~~~~~~~~~+\|Z\|_{L_{\mathcal{F}}^{2}(0,T;X_{3})}\Big],
\end{array}
\end{array}
\end{eqnarray}
where $C=C(a,b,T).$
\end{corollary}
\par
By means of Corollary \ref{C6} and the duality argument, we can obtain the following exact controllability result for the system (\ref{1.3}).
\begin{theorem}\label{T5}
System (\ref{1.3}) is exactly controllable at any time $T>0.$
\end{theorem}
\par
The controllability problems for linear and nonlinear deterministic fourth order Schr\"{o}dinger equations
are well studied in the literature (see \cite{W1,Z1} and the rich references
cited therein). In contrast, to the author¡¯s knowledge there is no published paper that
addresses the controllability of stochastic fourth order Schr\"{o}dinger equations.
\par
This paper is organized as follows. Section 2 is devoted to the well-posedness results. Section 3 establishes a crucial identity for a stochastic  fourth order Schr\"{o}dinger operator. In
Section 4, we give the proofs of Theorem \ref{T1}, Corollary \ref{C4} and Corollary \ref{C1}. Section 5 is devoted
to proving Theorem \ref{T3}, Corollary \ref{C5} and  Corollary \ref{C2}. In
Section 4, we establish the exact controllability of (\ref{1.3}).
\section{Well-posedness}
\par
In this section we prove the well-posedness results we need along
this paper.
\subsection{Well-posedness of forward and backward stochastic fourth order Schr\"{o}dinger equations with homogeneous boundary value contidion}
\begin{definition} A stochastic process $y$ is said to be a solution
of (\ref{1.49}) if
\begin{eqnarray*}
&&y~{\rm{is}}~ L^{2}(I){\rm{-valued~ and~}} \mathcal
{F}_{t}-{\rm{measurable~for~each~}}t\in [0,T],
\\
&&y\in L^{2}_{\mathcal {F}}(\Omega;C([0,T];L^{2}(I))),
\\
&&y(0)=y_{0}~{\rm{in~}} I, P-a.s.
\end{eqnarray*}
and
\begin{eqnarray}\label{1.5}
\begin{array}{l}
\begin{array}{llll}
\displaystyle
\int_{I}iy(t)vdx=\int_{I}iy_{0}vdx-\displaystyle\int_{0}^{t}\int_{I}y(s)v_{xxxx}
dxds
\\~~~~~~~~~~~~~~~~~~+\displaystyle\int_{0}^{t}\int_{I}(ay+f)vdxds+\int_{0}^{t}\int_{I}(by+g)vdxdw
\end{array}
\end{array}
\end{eqnarray}
holds for all $t\in[0,T]$ and all $v\in C^{\infty}_{0}(\overline{I}),$ for almost
all $\omega\in \Omega.$
\end{definition}

\begin{definition}\label{D3} A pair of stochastic processes $(z,Z)$ is said to be a solution
of
\begin{eqnarray}\label{1.50}
\begin{array}{l}
\left\{
\begin{array}{lll}
idz+z_{xxxx}dt=(az+bZ+h)dt+Zdw\\
z(0,t)=0=z(1,t)\\
z_{x}(0,t)=0=z_{x}(1,t)\\
z(x,T)=z_{T}(x)
\end{array}
\right.
\end{array}
\begin{array}{lll}\textrm{in}~Q,
\\\textrm{in}~(0,T),\\\textrm{in}~(0,T),\\\textrm{in}~I
\end{array}
\end{eqnarray}
if
\begin{eqnarray*}
&&(z,Z)~{\rm{is}}~ L^{2}(I)\times L^{2}(I)-{\rm{valued~ and~}}
\mathcal {F}_{t}-{\rm{measurable~for~each~}}t\in [0,T],
\\
&&(z,Z) \in L^{2}_{\mathcal {F}}(\Omega;C([0,T];L^{2}(I)))\times L^{2}_{\mathcal
{F}}(0,T;L^{2}(I)),
\\
&&z(T)=z_{T}~{\rm{in}}~ I, P-a.s.
\end{eqnarray*}
and
\begin{eqnarray*}
\begin{array}{l}
\begin{array}{llll}
\displaystyle\int_{I}iz_{T}vdx=
\int_{I}iz(t)vdx-\int_{t}^{T}\int_{I}z(s)v_{xxxx}dxds
\\~~~~~~~~~~~~~~~~~~\displaystyle+\int_{t}^{T}\int_{I}(az+bZ+h)vdxds+\int_{t}^{T}\int_{I}Z(s)vdxdw
\end{array}
\end{array}
\end{eqnarray*}
holds for all $t\in[0,T]$ and all $v\in C^{\infty}_{0}(\overline{I}),$ for almost
all $\omega\in \Omega.$
\end{definition}
\par
Consider the one-dimensional fourth order elliptic operator
$\Lambda$ on $L^{2}(I)$ as follows
\begin{eqnarray*}
\begin{array}{l}
\left\{
\begin{array}{llll}
\mathcal{D}(\Lambda)=H^{2}_{0}(I)\cap H^{4}(I),
\\ \Lambda y=y_{xxxx}~~~~\forall y\in \mathcal{D}(\Lambda).
\end{array}
\right.
\end{array}
\end{eqnarray*}
\par Let $\{\varphi_{k}\}_{k=1}^{\infty}$ be the corresponding eigenfunctions of $\Lambda$ such that $\|\varphi_{k}\|_{L^{2}(I)}=1~(k=1,2,3,\cdots),$ which serves as an orthonormal basis of $L^{2}(I)$ (See \cite[Theorem 8.94]{32}).
\par
According to \cite[Theorem 3.7]{F1}, we have
\begin{lemma}\label{L2}
For $0\leq\alpha<\beta$ and $0<s<1.$ We have the following results:
\begin{eqnarray*}
\begin{array}{l}
\begin{array}{llll}
\left[L_{\mathcal {F}}^{2}(\Omega;C([0,T];H^{\alpha}(I))),L_{\mathcal {F}}^{2}(\Omega;C([0,T];H^{\beta}(I)))\right]_{s}=L_{\mathcal {F}}^{2}(\Omega;C([0,T];H^{(1-s)\alpha+s\beta}(I))),
\\
\left [L^{2}_{\mathcal{F}}(0,T;H^{\alpha}(I)),L^{2}_{\mathcal{F}}(0,T;H^{\beta}(I))\right]_{s}=L^{2}_{\mathcal{F}}(0,T;H^{(1-s)\alpha+s\beta}(I)).
\end{array}
\end{array}
\end{eqnarray*}

\end{lemma}
\begin{proposition}\label{P1}
The well-posedness of (\ref{1.49}) is given in
the following:
\par
i) Let $y_{0}\in X_{0}$ and $f,g\in L^{2}_{\mathcal{F}}(0,T;X_{0})$ be given. Then (\ref{1.49}) admits a unique solution
$y\in L^{2}_{\mathcal {F}}(\Omega;C([0,T];X_{0}))$ such that
\begin{eqnarray}\label{1.4}
\begin{array}{l}
\begin{array}{llll}
\|y\|_{L^{2}_{\mathcal {F}}(\Omega;C([0,T];X_{0}))}
\leq C(\|y_{0}\|_{X_{0}}+\|f\|_{L_{\mathcal
{F}}^{2}(0,T;X_{0})}+\|g\|_{L_{\mathcal {F}}^{2}(0,T;X_{0})}).
\end{array}
\end{array}
\end{eqnarray}
 Moreover, it holds that
\begin{eqnarray}\label{1.6}
\begin{array}{l}
\begin{array}{llll}
\displaystyle E\|y(t)\|_{X_{0}}^{2}
\leq C [E\|y(\tau)\|_{X_{0}}^{2}+E\int_{t}^{\tau}(\|f(\eta)\|_{X_{0}}^{2}+\|g(\eta)\|_{X_{0}}^{2})d\eta]
\end{array}
\end{array}
\end{eqnarray}
for $0\leq t\leq \tau\leq T,$ where $C=C(a,b,T).$
\par
ii) Let $y_{0}\in X_{4}$ and $f,g\in L^{2}_{\mathcal{F}}(0,T;X_{4})$ be given. Then (\ref{1.49}) admits a unique solution $y\in L^{2}_{\mathcal {F}}(\Omega;C([0,T];X_{4}))$ such that
\begin{eqnarray}\label{1.14}
\begin{array}{l}
\begin{array}{llll}
\|y\|_{L^{2}_{\mathcal {F}}(\Omega;C([0,T];X_{4}))}
\leq C(\|y_{0}\|_{L^{2}\mathcal (\Omega,\mathcal
{F}_{0},P; X_{4})}+\|f\|_{L_{\mathcal
{F}}^{2}(0,T;X_{4})}+\|g\|_{L_{\mathcal {F}}^{2}(0,T;X_{4})}).
\end{array}
\end{array}
\end{eqnarray}  Moreover, it holds that
\begin{eqnarray}\label{1.41}
\begin{array}{l}
\begin{array}{llll}
\displaystyle E\|y(t)\|_{X_{4}}^{2}
\leq C [E\|y(\tau)\|_{X_{4}}^{2}+E\int_{t}^{\tau}(\|f(\eta)\|_{X_{4}}^{2}+\|g(\eta)\|_{X_{4}}^{2})d\eta]
\end{array}
\end{array}
\end{eqnarray}
for $0\leq t\leq \tau\leq T,$ where $C=C(a,b,T).$
\par
iii) For $0\leq s\leq 4.$ Let $y_{0}\in X_{s}$ and $f,g\in L^{2}_{\mathcal{F}}(0,T;X_{s})$ be given. Then (\ref{1.49}) admits a unique solution $y\in L^{2}_{\mathcal {F}}(\Omega;C([0,T];X_{s}))$ such that
\begin{eqnarray}\label{1.22}
\begin{array}{l}
\begin{array}{llll}
\|y\|_{L^{2}_{\mathcal {F}}(\Omega;C([0,T];X_{s}))}
\leq C(\|y_{0}\|_{L^{2}\mathcal (\Omega,\mathcal
{F}_{0},P; X_{s})}+\|f\|_{L_{\mathcal
{F}}^{2}(0,T;X_{s})}+\|g\|_{L_{\mathcal {F}}^{2}(0,T;X_{s})}).
\end{array}
\end{array}
\end{eqnarray}
 Moreover, it holds that
\begin{eqnarray}\label{1.39}
\begin{array}{l}
\begin{array}{llll}
\displaystyle E\|y(t)\|_{X_{s}}^{2}
\leq C [E\|y(\tau)\|_{X_{s}}^{2}+E\int_{t}^{\tau}(\|f(\eta)\|_{X_{s}}^{2}+\|g(\eta)\|_{X_{s}}^{2})d\eta]
\end{array}
\end{array}
\end{eqnarray}
for $0\leq t\leq \tau\leq T,$ where $C=C(a,b,T).$
\end{proposition}
\begin{proof}
\par
Let $(p,q)=\int_{I}pqdx,$ for any $p,q\in L^{2}(I).$
\par
i) Inspired by \cite{Kim}, we use the Galerkin method.
\par It follows from the classical theory of stochastic differential equations (adapted for the complex case) that the following system
\begin{eqnarray}\label{1.28}
\begin{array}{l}
\left\{
\begin{array}{llll}
dc^{m}_{k}=(\sum\limits_{j=1}^{m}a_{kj}c^{m}_{j}+u_{k})dt+(\sum\limits_{j=1}^{m}b_{kj}c^{m}_{j}+v_{k})dw,
\\c^{m}_{k}(0)=(y_{0},\varphi_{k})
\end{array}
\right.
\end{array}
\end{eqnarray}
admits a unique solution $c^{m}_{k}(t),$
where
\begin{eqnarray*}
\begin{array}{l}
\begin{array}{llll}
a_{kj}=(-i)[(a\varphi_{j},\varphi_{k})-(\varphi_{jxxxx},\varphi_{k})],
\\b_{kj}=(-i)(b\varphi_{j},\varphi_{k}),
\\u_{k}=(-i)(f,\varphi_{k}),
\\v_{k}=(-i)(g,\varphi_{k})
\end{array}
\end{array}
\end{eqnarray*}
for $k,j=1,\ldots,m.$
\par
Let us write
$$y^{m}=\sum\limits_{k=1}^{m} c^{m}_{k}\varphi_{k},$$
then $y^{m} \in C([0,T];H^{2}_{0}(I))$ for almost all $\omega\in
\Omega.$ It follows from (\ref{1.28}) that $y^{m}$ satisfies the following equations
\begin{eqnarray}\label{1.47}
\begin{array}{l}
\begin{array}{lll}
id(y^{m},\varphi_{k})+(y^{m}_{xxxx},\varphi_{k})dt=[(ay^{m},\varphi_{k})+(f,\varphi_{k})]dt
+[(by^{m},\varphi_{k})+(g,\varphi_{k})]dw,
\end{array}
\end{array}
\end{eqnarray}
$k=1,\ldots,m.$
\par Direct computation yields
\begin{eqnarray}\label{1.9}
\begin{array}{l}
\begin{array}{llll}
d|c_{k}^{m}|^{2}=[i(\overline{a}\overline{y}^{m},c_{k}\varphi_{k})-i(ay^{m},\overline{c}_{k}\varphi_{k})+i\overline{f}_{k}c_{k}-if_{k}\overline{c}_{k}+|(by^{m},\varphi_{k})+g_{k}|^{2}]dt
\\~~~~~~~~~~~~~~~~~~~~~~~~~~~~~~~~~~~~~+i[(\overline{b}\overline{y}^{m},c_{k}\varphi_{k})-(by^{m},\overline{c}_{k}\varphi_{k})+\overline{g}_{k}c_{k}-g_{k}\overline{c}_{k}]dw
\end{array}
\end{array}
\end{eqnarray}
where $f_{k}=(f,\varphi_{k}),g_{k}=(g,\varphi_{k}).$
\par
We take sums from $1$ to $m$ about $k$ in (\ref{1.9}) to obtain
\begin{eqnarray}\label{1.11}
\begin{array}{l}
\begin{array}{llll}
d\|y^{m}\|_{L^{2}(I)}^{2}=[i(\overline{a}\overline{y}^{m},y^{m})-i(ay^{m},\overline{y}^{m})+i(\overline{f}^{m},y^{m})-i(f^{m},\overline{y}^{m})+\sum\limits_{k=1}^{m}|(by^{m},\varphi_{k})+g_{k}|^{2}]dt
\\~~~~~~~~~~~~~~~~~~~~~~~~~~~~~~~~~~~~~+i[(\overline{b}\overline{y}^{m},y^{m})-(by^{m},\overline{y}^{m})+(\overline{g}^{m},y^{m})-(g^{m},\overline{y}^{m})]dw
\end{array}
\end{array}
\end{eqnarray}
where $f^{m}=\sum\limits_{k=1}^{m}f_{k}\varphi_{k}$ and $g^{m}=\sum\limits_{k=1}^{m}g_{k}\varphi_{k}.$
Namely,
\begin{eqnarray*}
\begin{array}{l}
\begin{array}{llll}
\|y^{m}(t)\|_{L^{2}(I)}^{2}=\|y^{m}(0)\|_{L^{2}(I)}^{2}
\\\displaystyle+\int_{0}^{t}[i(\overline{a}\overline{y}^{m},y^{m})-i(ay^{m},\overline{y}^{m})+i(\overline{f}^{m},y^{m})-i(f^{m},\overline{y}^{m})+\sum\limits_{k=1}^{m}|(by^{m},\varphi_{k})+g_{k}|^{2}]d\eta
\\\displaystyle+\int_{0}^{t}i[(\overline{b}\overline{y}^{m},y^{m})-(by^{m},\overline{y}^{m})+(\overline{g}^{m},y^{m})-(g^{m},\overline{y}^{m})]dw.
\end{array}
\end{array}
\end{eqnarray*}
\par
Next, we fix $m\geq 1$ and any positive integer $L$, and define a
   stopping time
\[
\tau_{L}=\left\{\begin{array}{cc}
    0,&{\rm if } |c_{k}(0)|\geq L, \\
 \inf\{t\in [0,T]:|c_{k}(t)|\geq L \},&{\rm if } |c_{k}(0)|\leq
    L\rm{~and~} \\&\{t\in [0,T]:|c_{k}(t)|\geq L \} \\& \rm{~is~not~empty,}\\
   T,&{\rm if } |c_{k}(0)|\leq L \rm{~and~} \\&\{t\in [0,T]:|c_{k}(t)|\geq L \} \\& \rm{~is~} empty.\end{array} \right.
 \]
\par By the Burkholder-Davis-Gundy inequality and Cauchy
   inequality, we have
 \begin{eqnarray*}
 \begin{array}{lll}
&&E\Big(\sup\limits_{0\leq \tau \leq t\wedge
\tau_{L}}|\displaystyle{\int_{0}^{\tau}i((\overline{b}\overline{y}^{m},y^{m})-(by^{m},\overline{y}^{m}))dw}|\Big)\\
&\leq&\varepsilon E\Big(\sup\limits_{0\leq \tau\leq t\wedge
\tau_{L}}\|y^{m}(\tau)\|_{L^{2}(I)}^{2}\Big)+C(b,\varepsilon)E\Big(\displaystyle \int_{0}^{t\wedge \tau_{L}} \|y^{m}(\tau)\|_{L^{2}(I)}^{2}d\tau\Big),\\
\\
&&E\Big(\sup\limits_{0\leq \tau \leq t\wedge
\tau_{L}}|\displaystyle{\int_{0}^{\tau}i((\overline{g}^{m},y^{m})-(g^{m},\overline{y}^{m}))dw}|\Big)\\
&\leq&\varepsilon E\Big(\sup\limits_{0\leq \tau\leq t\wedge
\tau_{L}}\|y^{m}(\tau)\|_{L^{2}(I)}^{2}\Big)+C(\varepsilon)E\Big(\displaystyle \int_{0}^{t\wedge \tau_{L}} \|g^{m}(\tau)\|_{L^{2}(I)}^{2}d\tau\Big),\\
\\
&&E\Big(\sup\limits_{0\leq \tau \leq t\wedge
\tau_{L}}|\displaystyle{\int_{0}^{\tau}i((\overline{a}\overline{y}^{m},y^{m})-(ay^{m},\overline{y}^{m}))d\eta}|\Big)
\leq C(a)E\Big(\displaystyle \int_{0}^{t\wedge \tau_{L}} \|y^{m}(\tau)\|_{L^{2}(I)}^{2}d\tau\Big),\\
\\
&&E\Big(\sup\limits_{0\leq \tau \leq t\wedge
\tau_{L}}|\displaystyle{\int_{0}^{\tau}i((\overline{f}^{m},y^{m})-(f^{m},\overline{y}^{m}))d\eta}|\Big)\\
&\leq&E\Big(\displaystyle \int_{0}^{t\wedge \tau_{L}} \|f^{m}(\tau)\|_{L^{2}(I)}^{2}d\tau\Big)+E\Big(\displaystyle \int_{0}^{t\wedge \tau_{L}} \|y^{m}(\tau)\|_{L^{2}(I)}^{2}d\tau\Big),
\end{array}
\end{eqnarray*}
\begin{eqnarray*}
 \begin{array}{lll}
&&E\Big(\sup\limits_{0\leq \tau \leq t\wedge
\tau_{L}}|\displaystyle{\int_{0}^{\tau}\sum\limits_{k=1}^{m}|(by^{m},\varphi_{k})+g_{k}|^{2}d\eta}|\Big)\\
&\leq&E\Big(\displaystyle \int_{0}^{t\wedge \tau_{L}}\sum\limits_{k=1}^{m}|(by^{m},\varphi_{k})+g_{k}|^{2} d\tau\Big)\\
&\leq&E\Big(\displaystyle \int_{0}^{t\wedge \tau_{L}}[\|by^{m}(\tau)\|_{L^{2}(I)}^{2}+\| g^{m}(\tau)\|_{L^{2}(I)}^{2}] d\tau\Big)\\
&\leq& C(b)\Big(\displaystyle E\int_{0}^{t\wedge \tau_{L}} \|y^{m}(\tau)\|_{L^{2}(I)}^{2}d\tau+E\int_{0}^{t\wedge \tau_{L}} \|g^{m}(\tau)\|_{L^{2}(I)}^{2}d\tau\Big)
\end{array}
\end{eqnarray*}
for $\forall t\in [0,T]$ and $\forall \varepsilon >0,$ here and below $C(\varepsilon)$
denote positive constants independent of $m.$ Thus, it holds that
\begin{eqnarray*}
\begin{array}{l}
\begin{array}{llll}
\displaystyle E\sup\limits_{0\leq \tau\leq t\wedge \tau_{L}}\|y^{m}(\tau)\|_{L^{2}(I)}^{2}
\\\leq \displaystyle C(a,b)E\Big[\|y^{m}(0)\|_{L^{2}(I)}^{2}
+\int_{0}^{t\wedge \tau_{L}}(\|f^{m}(\tau)\|_{L^{2}(I)}^{2}+\|g^{m}(\tau)\|_{L^{2}(I)}^{2}+ \|y^{m}(\tau)\|_{L^{2}(I)}^{2})d\tau\Big].
\end{array}
\end{array}
\end{eqnarray*}
By passing $L\rightarrow\infty$ in the above equation, we arrive at
\begin{eqnarray*}
\begin{array}{l}
\begin{array}{llll}
\displaystyle E\sup\limits_{0\leq \tau\leq t}\|y^{m}(\tau)\|_{L^{2}(I)}^{2}
\\\leq \displaystyle C(a,b)E\Big[\|y^{m}(0)\|_{L^{2}(I)}^{2}
+\int_{0}^{t}(\|f^{m}(\tau)\|_{L^{2}(I)}^{2}+\|g^{m}(\tau)\|_{L^{2}(I)}^{2}+ \|y^{m}(\tau)\|_{L^{2}(I)}^{2})d\tau\Big].
\end{array}
\end{array}
\end{eqnarray*}
Applying the Gronwall inequality, we can obtain
\begin{eqnarray}\label{3}
\begin{array}{l}
\begin{array}{llll}
\displaystyle E\sup\limits_{0\leq \tau\leq t}\|y^{m}(\tau)\|_{L^{2}(I)}^{2}
\leq C(a,b,T)E\Big[\|y^{m}(0)\|_{L^{2}(I)}^{2}
+\int_{0}^{t}(\|f^{m}(\tau)\|_{L^{2}(I)}^{2}+\|g^{m}(\tau)\|_{L^{2}(I)}^{2})d\tau\Big]
\end{array}
\end{array}
\end{eqnarray}
for all $t\in [0,T].$
\par By the same argument, we also have, for $m\geq
n\geq 1,$
\begin{eqnarray}\label{1.8}
\begin{array}{l}
\begin{array}{llll}
&&\displaystyle E\sup\limits_{0\leq \tau\leq t}\|y^{m}(\tau)-y^{n}(\tau)\|_{X_{0}}^{2}
\\&\leq& \displaystyle CE\Big[\|y^{m}(0)-y^{n}(0)\|_{X_{0}}^{2}
\\&&~~~~~~~~~~\displaystyle+\int_{0}^{t}(\| f^{m}(\tau)- f^{n}(\tau)\|_{X_{0}}^{2}+\|g^{m}(\tau)-g^{n}(\tau)\|_{X_{0}}^{2})d\tau\Big]
\end{array}
\end{array}
\end{eqnarray}
where $C$ denotes a positive contant independent of $m,n.$
Next we observe that the right-hand side of (\ref{1.8}) converges to
zero as $n,m\rightarrow \infty.$ Hence, it follows that
$\{y^{m}\}_{m=1}^{\infty}$ is a Cauchy sequence that converges
strongly in $L_{\mathcal {F}}^{2}(\Omega;C([0,T];X_{0})).$ Let $y$ be the limit. It
is apparent that $y$ satisfies the initial in
(\ref{1}), and $y(t)$ is $\mathcal {F}_{t}-$adapted for each
$t\in[0,T].$ Also, it follows from (\ref{1.47}) that (\ref{1.5}) holds. Furthermore, by passing $m\rightarrow\infty$ in (\ref{3}), we arrive at (\ref{1.4}).
\par For the uniqueness of the solution, we
suppose that $y_{1}$ and $y_{2}$ are two solutions of (\ref{1.49}).
Let $y=y_{1}-y_{2}.$ Then
\begin{eqnarray*}
E\Big(\sup\limits_{0\leq \tau\leq T}\| y(\tau)\|_{X_{0}}^{2}\Big)\leq 0,
\end{eqnarray*}thus
$y\equiv 0$ for any $t\in [0,T],$ for almost all $\omega \in \Omega.$
\par
It follows from (\ref{1.11}) that
\begin{eqnarray*}
\begin{array}{l}
\begin{array}{llll}
\|y^{m}(\tau)\|_{L^{2}(I)}^{2}=\|y^{m}(t)\|_{L^{2}(I)}^{2}
\\\displaystyle+\int_{t}^{\tau}[i(\overline{a}\overline{y}^{m},y^{m})-i(ay^{m},\overline{y}^{m})+i(\overline{f}^{m},y^{m})-i(f^{m},\overline{y}^{m})+\sum\limits_{k=1}^{m}|(by^{m},\varphi_{k})+g_{k}|^{2}]d\eta
\\\displaystyle+\int_{t}^{\tau}i[(\overline{b}\overline{y}^{m},y^{m})-(by^{m},\overline{y}^{m})+(\overline{g}^{m},y^{m})-(g^{m},\overline{y}^{m})]dw
\end{array}
\end{array}
\end{eqnarray*}
for $0\leq t\leq \tau\leq T.$ By taking expectation in above equality, we can obtain
\begin{eqnarray}\label{1.17}
\begin{array}{l}
\begin{array}{llll}
E\|y^{m}(t)\|_{L^{2}(I)}^{2}= E\|y^{m}(\tau)\|_{L^{2}(I)}^{2}
\displaystyle-E\int_{t}^{\tau}[i(\overline{a}\overline{y}^{m},y^{m})-i(ay^{m},\overline{y}^{m})
\\\displaystyle+i(\overline{f}^{m},y^{m})-i(f^{m},\overline{y}^{m})+\sum\limits_{k=1}^{m}|(by^{m},\varphi_{k})+g_{k}|^{2}]d\eta
\\\leq \displaystyle C(a,b)E\Big[\|y^{m}(\tau)\|_{L^{2}(I)}^{2}
+\int_{t}^{\tau}(\|f^{m}(\eta)\|_{L^{2}(I)}^{2}+\|g^{m}(\eta)\|_{L^{2}(I)}^{2}+ \|y^{m}(\eta)\|_{L^{2}(I)}^{2})d\eta\Big].
\end{array}
\end{array}
\end{eqnarray}
Applying the Gronwall inequality, we can obtain
\begin{eqnarray}\label{1.10}
\begin{array}{l}
\begin{array}{llll}
E\|y^{m}(t)\|_{L^{2}(I)}^{2}\leq \displaystyle C(a,b,T)E\Big[\|y^{m}(\tau)\|_{L^{2}(I)}^{2}
+\int_{t}^{\tau}(\|f^{m}(\eta)\|_{L^{2}(I)}^{2}+\|g^{m}(\eta)\|_{L^{2}(I)}^{2})d\eta\Big].
\end{array}
\end{array}
\end{eqnarray}
Taking now the limit $m\rightarrow \infty$ in (\ref{1.10}), we can obtain (\ref{1.6}).
\par
ii)
Inspired by \cite{G2}, we multiply (\ref{1.9}) by $\lambda_{k}^{2}$ and take sums from $1$ to $m$ about $k$ to obtain
\begin{eqnarray*}
\begin{array}{l}
\begin{array}{llll}
d\|\Lambda y^{m}\|_{L^{2}(I)}^{2}=[i(\overline{a}\overline{y}^{m},\Lambda^{2}y^{m})-i(ay^{m},\Lambda^{2}\overline{y}^{m})+i(\Lambda\overline{f}^{m},\Lambda y^{m})-i(\Lambda f^{m},\Lambda\overline{y}^{m})
\\+\sum\limits_{k=1}^{m}\lambda_{k}^{2}|(by^{m},\varphi_{k})+g_{k}|^{2}]dt
+i[(\overline{b}\overline{y}^{m},\Lambda^{2}y^{m})-(by^{m},\Lambda^{2}\overline{y}^{m})+(\Lambda\overline{g}^{m},\Lambda y^{m})-(\Lambda g^{m},\Lambda\overline{y}^{m})]dw
\end{array}
\end{array}
\end{eqnarray*}
for all $t\in [0,T].$ Using integration by parts, we get that
\begin{eqnarray}\label{1.15}
\begin{array}{l}
\begin{array}{llll}
d\|\Lambda y^{m}\|_{L^{2}(I)}^{2}=[i(\Lambda(\overline{a}\overline{y}^{m}),\Lambda y^{m})
\\-i(\Lambda(ay^{m}),\Lambda\overline{y}^{m})+i(\Lambda\overline{f}^{m},\Lambda y^{m})-i(\Lambda f^{m},\Lambda\overline{y}^{m})
+\sum\limits_{k=1}^{m}\lambda_{k}^{2}|(by^{m},\varphi_{k})+g_{k}|^{2}]dt
\\+i[(\Lambda(\overline{b}\overline{y}^{m}),\Lambda y^{m})-(\Lambda(by^{m}),\Lambda \overline{y}^{m})+(\Lambda\overline{g}^{m},\Lambda y^{m})-(\Lambda g^{m},\Lambda\overline{y}^{m})]dw
\end{array}
\end{array}
\end{eqnarray}
Namely,
\begin{eqnarray*}
\begin{array}{l}
\begin{array}{llll}
\|\Lambda y^{m}(t)\|_{L^{2}(I)}^{2}=\|\Lambda y^{m}(0)\|_{L^{2}(I)}^{2}
\\\displaystyle+\int_{0}^{t}[i(\Lambda(\overline{a}\overline{y}^{m}),\Lambda y^{m})-i(\Lambda(ay^{m}),\Lambda\overline{y}^{m})+i(\Lambda\overline{f}^{m},\Lambda y^{m})-i(\Lambda f^{m},\Lambda\overline{y}^{m})
\\+\sum\limits_{k=1}^{m}\lambda_{k}^{2}|(by^{m},\varphi_{k})+g_{k}|^{2}]d\eta
\\\displaystyle+\int_{0}^{t}i[(\Lambda(\overline{b}\overline{y}^{m}),\Lambda y^{m})-(\Lambda(by^{m}),\Lambda\overline{y}^{m})+(\Lambda\overline{g}^{m},\Lambda y^{m})-(\Lambda g^{m},\Lambda\overline{y}^{m})]dw
\end{array}
\end{array}
\end{eqnarray*}
\par By the Burkholder-Davis-Gundy inequality and Cauchy
   inequality, we have
 \begin{eqnarray*}
 \begin{array}{lll}
&&E\Big(\sup\limits_{0\leq \tau \leq t\wedge
\tau_{L}}|\displaystyle{\int_{0}^{\tau}i((\Lambda(\overline{b}\overline{y}^{m}),\Lambda y^{m})-(\Lambda(by^{m}),\Lambda\overline{y}^{m}))dw}|\Big)\\
&\leq&\varepsilon E\Big(\sup\limits_{0\leq \tau\leq t\wedge
\tau_{L}}\|\Lambda y^{m}(\tau)\|_{L^{2}(I)}^{2}\Big)+C(b,\varepsilon)E\Big(\displaystyle \int_{0}^{t\wedge \tau_{L}} \|y^{m}(\tau)\|_{L^{2}(I)}^{2}+ \|\Lambda y^{m}(\tau)\|_{L^{2}(I)}^{2}d\tau\Big),\\
\\
&&E\Big(\sup\limits_{0\leq \tau \leq t\wedge
\tau_{L}}|\displaystyle{\int_{0}^{\tau}i((\Lambda\overline{g}^{m},\Lambda y^{m})-(\Lambda g^{m},\Lambda\overline{y}^{m}))dw}|\Big)\\
&\leq&\varepsilon E\Big(\sup\limits_{0\leq \tau\leq t\wedge
\tau_{L}}\|\Lambda y^{m}(\tau)\|_{L^{2}(I)}^{2}\Big)+C(\varepsilon)E\Big(\displaystyle \int_{0}^{t\wedge \tau_{L}} \|\Lambda g^{m}(\tau)\|_{L^{2}(I)}^{2}d\tau\Big),\\
\\
&&E\Big(\sup\limits_{0\leq \tau \leq t\wedge
\tau_{L}}|\displaystyle{\int_{0}^{\tau}i((\Lambda(\overline{a}\overline{y}^{m}),\Lambda y^{m})-(\Lambda(ay^{m}),\Lambda\overline{y}^{m}))d\eta}|\Big)\\
&\leq& C(a)\Big(\displaystyle E\int_{0}^{t\wedge \tau_{L}} \|y^{m}(\tau)\|_{L^{2}(I)}^{2}d\tau+E\int_{0}^{t\wedge \tau_{L}} \|\Lambda y^{m}(\tau)\|_{L^{2}(I)}^{2}d\tau\Big),\\
\\
&&E\Big(\sup\limits_{0\leq \tau \leq t\wedge
\tau_{L}}|\displaystyle{\int_{0}^{\tau}i((\Lambda\overline{f}^{m},\Lambda y^{m})-(\Lambda f^{m},\Lambda\overline{y}^{m}))d\eta}|\Big)\\
&\leq&E\Big(\displaystyle \int_{0}^{t\wedge \tau_{L}} \|\Lambda f^{m}(\tau)\|_{L^{2}(I)}^{2}d\tau\Big)+E\Big(\displaystyle \int_{0}^{t\wedge \tau_{L}} \|\Lambda y^{m}(\tau)\|_{L^{2}(I)}^{2}d\tau\Big),\\
\\
&&E\Big(\sup\limits_{0\leq \tau \leq t\wedge
\tau_{L}}|\displaystyle{\int_{0}^{\tau}\sum\limits_{k=1}^{m}\lambda_{k}^{2}|(by^{m},\varphi_{k})+g_{k}|^{2}d\eta}|\Big)\\
&\leq&E\Big(\displaystyle \int_{0}^{t\wedge \tau_{L}}\sum\limits_{k=1}^{m}\lambda_{k}^{2}|(by^{m},\varphi_{k})+g_{k}|^{2} d\tau\Big)\\
&\leq&E\Big(\displaystyle \int_{0}^{t\wedge \tau_{L}}[\|\Lambda(by^{m})(\tau)\|_{L^{2}(I)}^{2}+\|\Lambda g^{m}(\tau)\|_{L^{2}(I)}^{2}] d\tau\Big)\\
&\leq& C(b)\Big(\displaystyle E\int_{0}^{t\wedge \tau_{L}} \|y^{m}(\tau)\|_{L^{2}(I)}^{2}d\tau+E\int_{0}^{t\wedge \tau_{L}} \|\Lambda y^{m}(\tau)\|_{L^{2}(I)}^{2}d\tau+E\int_{0}^{t\wedge \tau_{L}} \|\Lambda g^{m}(\tau)\|_{L^{2}(I)}^{2}d\tau\Big)
\end{array}
\end{eqnarray*}
for $\forall t\in [0,T]$ and $\forall \varepsilon >0,$ here and below $C(\varepsilon)$
denote positive constants independent of $m.$ Thus, if we take $\varepsilon$ small enough, we have
\begin{eqnarray*}
\begin{array}{l}
\begin{array}{llll}
\displaystyle E\sup\limits_{0\leq \tau\leq t\wedge \tau_{L}}\|\Lambda y^{m}(\tau)\|_{L^{2}(I)}^{2}
\\\leq \displaystyle C(a,b)E\Big[\|\Lambda y^{m}(0)\|_{L^{2}(I)}^{2}
\\\displaystyle+\int_{0}^{t\wedge \tau_{L}}(\|\Lambda f^{m}(\tau)\|_{L^{2}(I)}^{2}+\|\Lambda g^{m}(\tau)\|_{L^{2}(I)}^{2}+\|y^{m}(\tau)\|_{L^{2}(I)}^{2} +\|\Lambda y^{m}(\tau)\|_{L^{2}(I)}^{2})d\tau\Big].
\end{array}
\end{array}
\end{eqnarray*}
By passing $L\rightarrow\infty,$ we arrive at
\begin{eqnarray*}
\begin{array}{l}
\begin{array}{llll}
\displaystyle E\sup\limits_{0\leq \tau\leq t}\|\Lambda y^{m}(\tau)\|_{L^{2}(I)}^{2}
\\\leq \displaystyle C(a,b)E\Big[\|\Lambda y^{m}(0)\|_{L^{2}(I)}^{2}
\\\displaystyle+\int_{0}^{t}(\|\Lambda f^{m}(\tau)\|_{L^{2}(I)}^{2}+\|\Lambda g^{m}(\tau)\|_{L^{2}(I)}^{2}+\|y^{m}(\tau)\|_{L^{2}(I)}^{2}+ \|\Lambda y^{m}(\tau)\|_{L^{2}(I)}^{2})d\tau\Big].
\end{array}
\end{array}
\end{eqnarray*}
Applying the Gronwall inequality, we can obtain
\begin{eqnarray*}
\begin{array}{l}
\begin{array}{llll}
\displaystyle E\sup\limits_{0\leq \tau\leq t}\|\Lambda y^{m}(\tau)\|_{L^{2}(I)}^{2}
\\\leq \displaystyle C(a,b,T)E\Big[\|\Lambda y^{m}(0)\|_{L^{2}(I)}^{2}
+\int_{0}^{t}(\|\Lambda f^{m}(\tau)\|_{L^{2}(I)}^{2}+\|\Lambda g^{m}(\tau)\|_{L^{2}(I)}^{2}+\|y^{m}(\tau)\|_{L^{2}(I)}^{2})d\tau\Big].
\end{array}
\end{array}
\end{eqnarray*}
It follows from (\ref{3}) that
\begin{eqnarray}\label{4}
\begin{array}{l}
\begin{array}{llll}
\displaystyle E\sup\limits_{0\leq \tau\leq t}\|\Lambda y^{m}(\tau)\|_{L^{2}(I)}^{2}
\\\leq \displaystyle C(a,b,T)E\Big[\|\Lambda y^{m}(0)\|_{L^{2}(I)}^{2}+\|y^{m}(0)\|_{L^{2}(I)}^{2}+\int_{0}^{t}(\|f^{m}(\tau)\|_{L^{2}(I)}^{2}+\|g^{m}(\tau)\|_{L^{2}(I)}^{2}
\\~~~~~~~~~~~~~~~~~~~~~~~~~~+\|\Lambda f^{m}(\tau)\|_{L^{2}(I)}^{2}+\|\Lambda g^{m}(\tau)\|_{L^{2}(I)}^{2})d\tau\Big]
\end{array}
\end{array}
\end{eqnarray}
for all $t\in [0,T].$
\par
Now summing up (\ref{3}) and (\ref{4}) yields
\begin{eqnarray}\label{1.13}
\begin{array}{l}
\begin{array}{llll}
\displaystyle E\sup\limits_{0\leq \tau\leq t}\| y^{m}(\tau)\|_{X_{4}}^{2}\leq CE\Big[\|y^{m}(0)\|_{X_{4}}^{2}+\int_{0}^{t}(\| f^{m}(\tau)\|_{X_{4}}^{2}+\|g^{m}(\tau)\|_{X_{4}}^{2})d\tau\Big].
\end{array}
\end{array}
\end{eqnarray}
\par By the same argument, we also have, for $m\geq
n\geq 1,$
\begin{eqnarray}\label{1.12}
\begin{array}{l}
\begin{array}{llll}
&&\displaystyle E\sup\limits_{0\leq \tau\leq t}\|y^{m}(\tau)-y^{n}(\tau)\|_{X_{4}}^{2}
\\&\leq& \displaystyle CE\Big[\|y^{m}(0)-y^{n}(0)\|_{X_{4}}^{2}
\\&&~~~~~~~~~~\displaystyle+\int_{0}^{t}(\| f^{m}(\tau)- f^{n}(\tau)\|_{X_{4}}^{2}+\|g^{m}(\tau)-g^{n}(\tau)\|_{X_{4}}^{2})d\tau\Big]
\end{array}
\end{array}
\end{eqnarray}
where $C$ denotes a positive contant independent of $m,n.$
Next we observe that the right-hand side of (\ref{1.12}) converges to
zero as $n,m\rightarrow \infty.$ Hence, it follows that
$\{y^{m}\}_{m=1}^{\infty}$ is a Cauchy sequence that converges
strongly in $L_{\mathcal {F}}^{2}(\Omega;C([0,T];X_{4})).$ Let $y_{1}$ be the limit. It
is easy to know that $y_{1}=y,$ namely, we have
 \begin{eqnarray*}
 y^{m}\rightarrow y ~~{\rm{in}}~~L_{\mathcal {F}}^{2}(\Omega;C([0,T];X_{4})).
\end{eqnarray*}
By passing $m\rightarrow\infty$ in (\ref{1.13}), we arrive at (\ref{1.14}).
\par
It follows from (\ref{1.15})
\begin{eqnarray*}
\begin{array}{l}
\begin{array}{llll}
\|\Lambda y^{m}(\tau)\|_{L^{2}(I)}^{2}=\|\Lambda y^{m}(t)\|_{L^{2}(I)}^{2}
\\\displaystyle+\int_{t}^{\tau}[i(\Lambda (\overline{a}\overline{y}^{m}),\Lambda y^{m})-i(\Lambda (ay^{m}),\Lambda \overline{y}^{m})+i(\Lambda \overline{f}^{m},\Lambda y^{m})-i(\Lambda f^{m},\Lambda \overline{y}^{m})
\\+\sum\limits_{k=1}^{m}\lambda_{k}^{2}|(by^{m},\varphi_{k})+g_{k}|^{2}]d\eta
\\\displaystyle+\int_{t}^{\tau}i[(\Lambda (\overline{b}\overline{y}^{m}),\Lambda y^{m})-(\Lambda (by^{m}),\Lambda \overline{y}^{m})+(\Lambda \overline{g}^{m},\Lambda y^{m})-(\Lambda g^{m},\Lambda \overline{y}^{m})]dw
\end{array}
\end{array}
\end{eqnarray*}
for $0\leq t\leq \tau\leq T.$ By taking expectation in above equality, we can obtain
\begin{eqnarray}\label{1.16}
\begin{array}{l}
\begin{array}{llll}
E\|\Lambda y^{m}(t)\|_{L^{2}(I)}^{2}= E\|\Lambda y^{m}(\tau)\|_{L^{2}(I)}^{2}
\displaystyle-E\int_{t}^{\tau}[i(\Lambda (\overline{a}\overline{y}^{m}),\Lambda y^{m})-i(\Lambda (ay^{m}),\Lambda \overline{y}^{m})
\\\displaystyle+i(\Lambda \overline{f}^{m},\Lambda y^{m})-i(\Lambda f^{m},\Lambda \overline{y}^{m})+\sum\limits_{k=1}^{m}\lambda_{k}^{2}|(by^{m},\varphi_{k})+g_{k}|^{2}]d\eta
\\\leq \displaystyle C(a,b)E\Big[\|\Lambda y^{m}(\tau)\|_{L^{2}(I)}^{2}
\\\displaystyle+\int_{t}^{\tau}(\|\Lambda f^{m}(\eta)\|_{L^{2}(I)}^{2}+\|\Lambda g^{m}(\eta)\|_{L^{2}(I)}^{2}+ \|\Lambda y^{m}(\eta)\|_{L^{2}(I)}^{2}+ \|y^{m}(\eta)\|_{L^{2}(I)}^{2})d\eta\Big].
\end{array}
\end{array}
\end{eqnarray}
\par
Now summing up (\ref{1.17}) and (\ref{1.16}) yields
\begin{eqnarray*}
\begin{array}{l}
\begin{array}{llll}
E\|y^{m}(t)\|_{X_{4}}^{2}\leq \displaystyle C(a,b)E\Big[\|y^{m}(\tau)\|_{X_{4}}^{2}+\int_{t}^{\tau}(\|f^{m}(\eta)\|_{X_{4}}^{2}+\|g^{m}(\eta)\|_{X_{4}}^{2}+\|y^{m}(\eta)\|_{X_{4}}^{2})d\eta\Big].
\end{array}
\end{array}
\end{eqnarray*}

Applying the Gronwall inequality, we can obtain
\begin{eqnarray}\label{1.20}
\begin{array}{l}
\begin{array}{llll}
E\|y^{m}(t)\|_{X_{4}}^{2}\leq \displaystyle C(a,b,T)E\Big[\|y^{m}(\tau)\|_{X_{4}}^{2}
+\int_{t}^{\tau}(\|f^{m}(\eta)\|_{X_{4}}^{2}+\|g^{m}(\eta)\|_{X_{4}}^{2})d\eta\Big].
\end{array}
\end{array}
\end{eqnarray}
Taking now the limit $m\rightarrow \infty$ in (\ref{1.20}), we can obtain (\ref{1.41}).

\par
iii) The main idea here comes from \cite[Lemma 3.3]{B3} and \cite[Theorem 2.9]{R1}. The cases $s=0$ and $s=4$ have been proved in i) and ii). The
cases of $0<s<4$ follows by the interpolation theory in \cite{F1,B1} and Lemma \ref{L1}, thus we can obtain iii).
\par The proof of Proposition \ref{P1} is completed.
\end{proof}
\par
By the same argument as in Proposition \ref{P1}, we have
\begin{proposition}\label{P2}The well-posedness of (\ref{1.50}) is given in
the following:
\par
i) Let $z_{T}\in L^{2}\mathcal (\Omega,\mathcal
{F}_{T},P; X_{0})$ and $h\in L^{2}_{\mathcal{F}}(0,T;X_{0})$ be given. Then (\ref{1.50}) admits a unique solution $(z,Z)\in L^{2}_{\mathcal {F}}(\Omega;C([0,T];X_{0}))\times L^{2}_{\mathcal {F}}(0,T;X_{0})$ such that
\begin{eqnarray*}
\begin{array}{l}
\begin{array}{llll}
\|z\|_{L^{2}_{\mathcal {F}}(\Omega;C([0,T];X_{0}))}+\|Z\|_{L^{2}_{\mathcal {F}}(0,T;X_{0})}
\leq C(\|z_{T}\|_{L^{2}\mathcal (\Omega,\mathcal
{F}_{T},P; X_{0})}+\|h\|_{L_{\mathcal{F}}^{2}(0,T;X_{0})}).
\end{array}
\end{array}
\end{eqnarray*}
 Moreover, it holds that
\begin{eqnarray*}
\begin{array}{l}
\begin{array}{llll}
\displaystyle E\|z(t)\|_{X_{0}}^{2}
\leq C E\|z(\tau)\|_{X_{0}}^{2}+C E\int_{\tau}^{t}(\|h(\eta)\|_{X_{0}}^{2}+\|Z(\eta)\|_{X_{0}}^{2})d\eta,
\end{array}
\end{array}
\end{eqnarray*}
for $0\leq \tau\leq t\leq T,$ where $C=C(a,b,T).$
\par
ii) Let $z_{T}\in L^{2}\mathcal (\Omega,\mathcal
{F}_{T},P; X_{4})$ and $h\in L^{2}_{\mathcal{F}}(0,T;X_{4})$ be given. Then (\ref{1.50}) admits a unique solution $(z,Z)\in L^{2}_{\mathcal {F}}(\Omega;C([0,T];X_{4}))\times L^{2}_{\mathcal {F}}(0,T;X_{4})$ such that
\begin{eqnarray*}
\begin{array}{l}
\begin{array}{llll}
\|z\|_{L^{2}_{\mathcal {F}}(\Omega;C([0,T];X_{4}))}+\|Z\|_{L^{2}_{\mathcal {F}}(0,T;X_{4})}
\leq C(\|z_{T}\|_{L^{2}\mathcal (\Omega,\mathcal
{F}_{T},P; X_{4})}+\|h\|_{L_{\mathcal{F}}^{2}(0,T;X_{4})}).
\end{array}
\end{array}
\end{eqnarray*} Moreover, it holds that
\begin{eqnarray*}
\begin{array}{l}
\begin{array}{llll}
\displaystyle E\|z(t)\|_{X_{4}}^{2}
\leq C E\|z(\tau)\|_{X_{4}}^{2}+C E\int_{\tau}^{t}(\|h(\eta)\|_{X_{4}}^{2}+\|Z(\eta)\|_{X_{4}}^{2})d\eta,
\end{array}
\end{array}
\end{eqnarray*}
for $0\leq \tau\leq t\leq T,$ where $C=C(a,b,T).$
\par
iii) For $0\leq s\leq 4.$ Let $z_{T}\in L^{2}\mathcal (\Omega,\mathcal
{F}_{T},P; X_{s})$ and $h\in L^{2}_{\mathcal{F}}(0,T;X_{s})$ be given. Then (\ref{1.50}) admits a unique solution $(z,Z)\in L^{2}_{\mathcal {F}}(\Omega;C([0,T];X_{s}))\times L^{2}_{\mathcal {F}}(0,T;X_{s})$ such that
\begin{eqnarray}\label{1.27}
\begin{array}{l}
\begin{array}{llll}
\|z\|_{L^{2}_{\mathcal {F}}(\Omega;C([0,T];X_{s}))}+\|Z\|_{L^{2}_{\mathcal {F}}(0,T;X_{s})}
\leq C(\|z_{T}\|_{L^{2}\mathcal (\Omega,\mathcal
{F}_{T},P; X_{s})}+\|h\|_{L_{\mathcal{F}}^{2}(0,T;X_{s})}).
\end{array}
\end{array}
\end{eqnarray}
 Moreover, it holds that
\begin{eqnarray*}
\begin{array}{l}
\begin{array}{llll}
\displaystyle E\|z(t)\|_{X_{s}}^{2}
\leq C E\|z(\tau)\|_{X_{s}}^{2}+C E\int_{\tau}^{t}(\|h(\eta)\|_{X_{s}}^{2}+\|Z(\eta)\|_{X_{s}}^{2})d\eta,
\end{array}
\end{array}
\end{eqnarray*}
for $0\leq \tau\leq t\leq T,$ where $C=C(a,b,T).$
\end{proposition}
\par
Next, we establish the regularity of the solutions to (\ref{1.49}) and (\ref{1.50}).
\begin{proposition}\label{P3}
Let $y_{0}\in X_{3},f,g\in L^{2}_{\mathcal{F}}(0,T;X_{3})$ and $y$ be the solution to (\ref{1.49}),
then $y_{xx}(0,t),y_{xx}(1,t),$ $y_{xxx}(0,t),y_{xxx}(1,t)\in L^{2}_{\mathcal{F}}(\Omega,L^{2}(0,T)).$ Further, it holds that
\begin{eqnarray}\label{1.19}
\begin{array}{l}
\begin{array}{llll}
\|y_{xx}(0,\cdot)\|_{L^{2}_{\mathcal{F}}(\Omega,L^{2}(0,T))}+\|y_{xx}(1,\cdot)\|_{L^{2}_{\mathcal{F}}(\Omega,L^{2}(0,T))} \\~~+\|y_{xxx}(0,\cdot)\|_{L^{2}_{\mathcal{F}}(\Omega,L^{2}(0,T))} +\|y_{xxx}(1,\cdot)\|_{L^{2}_{\mathcal{F}}(\Omega,L^{2}(0,T))}
\\\leq \displaystyle C\Big[\|y_{0}\|_{L^{2}_{\mathcal{F}} (\Omega,\mathcal
{F}_{0},P; X_{3})}
+\| f\|_{L^{2}_{\mathcal{F}}(0,T;X_{3})}+\|g\|_{L^{2}_{\mathcal{F}}(0,T;X_{3})}\Big],

\end{array}
\end{array}
\end{eqnarray}
where $C=C(a,b,T)$ is a constant independent of $y_{0},f$ and $g.$
\end{proposition}
\par In order to prove Proposition \ref{P3}, we first establish the following pointwise identity by some direct computations.
\begin{lemma}\label{L1}
Let $A \in C^{3}(\overline{I},\mathbb{R})$ and $y$ be an $H^{4}(\mathbb{R})$-valued $\{\mathcal
{F}_{t}\}_{t\geq0}$-adapted semi-martingale. Then for a.e. $x\in \mathbb{R}$ and P-a.s. $\omega\in \Omega$, it holds that
\begin{eqnarray*}
\begin{array}{l}
\begin{array}{llll}
A\overline{y}_{x}(idy+y_{xxxx}dt)+Ay_{x}(-id\overline{y}+\overline{y}_{xxxx}dt)
\\=((\frac{-iA}{2})(yd\overline{y}-\overline{y}dy))_{x}-\frac{1}{2}d((-iA)y\overline{y}_{x}-(-iA)y_{x}\overline{y})
\\~~~+\frac{1}{2}(-iA)(dyd\overline{y}_{x}-dy_{x}d\overline{y})+\frac{1}{2}(-iA_{x})(\overline{y}dy-y d\overline{y})+\frac{1}{2}(-iA_{t})(y\overline{y}_{x}-y_{x}\overline{y})dt
\\~~~+(A|y_{x}|^{2})_{xxx}dt-3(A_{x}|y_{x}|^{2})_{xx}dt+(3A_{xx}|y_{x}|^{2}-3A|y_{xx}|^{2})_{x}dt
\\~~~-A_{xxx}|y_{x}|^{2}dt+3A_{x}|y_{xx}|^{2}dt

\end{array}
\end{array}
\end{eqnarray*}
and
\begin{eqnarray*}
\begin{array}{l}
\begin{array}{llll}
A\overline{y}_{xxx}(idy+y_{xxxx}dt)+Ay_{xxx}(-id\overline{y}+\overline{y}_{xxxx}dt)
\\=[(-iA)(y_{xx}d\overline{y}-\overline{y}_{xx}dy)-(-iA_{x})(y_{x}d\overline{y}-\overline{y}_{x}dy)
\\~~~+\frac{1}{2}(-iA_{xx})(yd\overline{y}-\overline{y}dy)-\frac{1}{2}(-iA)(y_{x}d\overline{y}_{x}-\overline{y}_{x}dy_{x})+A|y_{xxx}|^{2}dt]_{x}
\\~~~+\frac{1}{2}d[(-iA)(y_{x}\overline{y}_{xx}-\overline{y}_{x}y_{xx})-(-iA)_{xx}(y\overline{y}_{x}-\overline{y}y_{x})]-\frac{1}{2}(-iA)(dy_{x}d\overline{y}_{xx}-dy_{xx}d\overline{y}_{x})
\\~~~-\frac{3}{2}(-iA)_{x}(\overline{y}_{x}dy_{x}-y_{x}d\overline{y}_{x})+\frac{1}{2}(-iA)_{xx}(dyd\overline{y}_{x}-dy_{x}d\overline{y})
+\frac{1}{2}(-iA)_{xxx}(\overline{y}dy-yd\overline{y})
\\~~~+\frac{1}{2}(-iA)_{xxt}(y\overline{y}_{x}-y_{x}\overline{y})dt-\frac{1}{2}(-iA)_{t}(y_{x}\overline{y}_{xx}-y_{xx}\overline{y}_{x})dt
-A_{x}|y_{xxx}|^{2}dt.
\end{array}
\end{array}
\end{eqnarray*}
\end{lemma}
\begin{proof}[Proof of Proposition \ref{P3}]
According to Lemma \ref{L1} and the boundary value conditions of (\ref{1.49}), we can obtain
\begin{eqnarray}\label{1.30}
\begin{array}{l}
\begin{array}{llll}
\displaystyle E\int_{Q}(A|y_{xx}|^{2})_{x}dxdt=-E\int_{I}\int_{0}^{T}A\overline{y}_{x}(idy+y_{xxxx}dt)dx-E\int_{I}\int_{0}^{T}Ay_{x}(-id\overline{y}+\overline{y}_{xxxx}dt)dx
\\-\displaystyle\frac{1}{2}E\int_{Q}d[(-iA)(y\overline{y_{x}}-y_{x}\overline{y})]dx- E\int_{Q}A_{xxx}|y_{x}|^{2}dxdt+3E\int_{Q}A_{x}|y_{xx}|^{2}dxdt
\\+\displaystyle E\int_{Q} \frac{1}{2}(-iA)(dyd\overline{y_{x}}-dy_{x}d\overline{y})dx+E\int_{Q} \frac{1}{2}(-iA)_{x}(\overline{y}dy-yd\overline{y})dx
\\+\displaystyle E\int_{Q} \frac{1}{2}(-iA)_{t}(y\overline{y}_{x}-y_{x}\overline{y})dxdt
\end{array}
\end{array}
\end{eqnarray}
and
\begin{eqnarray}\label{1.31}
\begin{array}{l}
\begin{array}{llll}
\displaystyle E\int_{Q}(A|y_{xxx}|^{2})_{x}dxdt=E\int_{I}\int_{0}^{T}A\overline{y}_{xxx}(idy+y_{xxxx}dt)dx+E\int_{I}\int_{0}^{T}Ay_{xxx}(-id\overline{y}+\overline{y}_{xxxx}dt)dx
\\-\displaystyle\frac{1}{2}E\int_{Q}d[(-iA)(y_{x}\overline{y_{xx}}-y_{xx}\overline{y_{x}})-(-iA)_{xx}(y\overline{y_{x}}-y_{x}\overline{y})]dx
\\+ E\int_{Q} \frac{1}{2}(-iA)(dy_{x}d\overline{y_{xx}}-dy_{xx}d\overline{y_{x}})dx+ E\int_{Q} \frac{3}{2}(-iA)_{x}(\overline{y_{x}}d y_{x}-y_{x}d\overline{y_{x}})dx
\\\displaystyle-E\int_{Q} \frac{1}{2}(-iA)_{xx}(dyd\overline{y_{x}}-dy_{x}d\overline{y})dx-E\int_{Q} \frac{1}{2}(-iA)_{xxx}(\overline{y}dy-yd\overline{y})dx
\\-\displaystyle E\int_{Q} \frac{1}{2}(-iA)_{xxt}(y\overline{y}_{x}-y_{x}\overline{y})dxdt+E\int_{Q} \frac{1}{2}(-iA)_{t}(y_{x}\overline{y}_{xx}-y_{xx}\overline{y_{x}})dxdt
\\+\displaystyle E\int_{Q}A_{x}|y_{xxx}|^{2}dxdt.
\end{array}
\end{array}
\end{eqnarray}
Summing up (\ref{1.30}) and (\ref{1.31}), taking $A(x)=-4x^{3}+6x^{2}-1$ and using the Cauchy inequality, we can obtain
\begin{eqnarray*}
\begin{array}{l}
\begin{array}{llll}
E\displaystyle\int_{Q}(|y_{xx}|^{2}(1,t)+|y_{xx}|^{2}(0,t)+|y_{xxx}|^{2}(1,t)+|y_{xxx}|^{2}(0,t))dt
\\\leq \displaystyle CE\Big[\|y\|^{2}_{C([0,T];X_{3})}
+\int_{0}^{T}\|ay+f\|^{2}_{X_{3}}dt+\int_{0}^{T}\|by+g\|^{2}_{X_{3}}dt\Big].
\end{array}
\end{array}
\end{eqnarray*}
According to (\ref{1.22}) with $s=3,$ this implies (\ref{1.19}).
\end{proof}
\par
By the same method as in Proposition \ref{P3}, we have
\begin{proposition}\label{P4}
Let $z_{T}\in L^{2}_{\mathcal{F}} (\Omega,\mathcal
{F}_{T},P;X_{3}),h\in L^{2}_{\mathcal{F}}(0,T;X_{3})$ and $(z,Z)$ be the solution to (\ref{1.50}),
then $z_{xx}(0,t),z_{xx}(1,t),z_{xxx}(0,t),z_{xxx}(1,t)\in L^{2}_{\mathcal{F}}(\Omega,L^{2}(0,T)).$ Further, it holds that
\begin{eqnarray*}
\begin{array}{l}
\begin{array}{llll}
\|z_{xx}(0,\cdot)\|_{L^{2}_{\mathcal{F}}(\Omega,L^{2}(0,T))}+\|z_{xx}(1,\cdot)\|_{L^{2}_{\mathcal{F}}(\Omega,L^{2}(0,T))} \\~~+\|z_{xxx}(0,\cdot)\|_{L^{2}_{\mathcal{F}}(\Omega,L^{2}(0,T))} +\|z_{xxx}(1,\cdot)\|_{L^{2}_{\mathcal{F}}(\Omega,L^{2}(0,T))}
\\\leq \displaystyle C\Big[\|z_{T}\|_{L^{2}_{\mathcal{F}} (\Omega,\mathcal
{F}_{T},P;X_{3})}+\|h\|_{L^{2}_{\mathcal{F}}(0,T;X_{3})}\Big],

\end{array}
\end{array}
\end{eqnarray*}
where $C=C(a,b,T)$ is a constant independent of $z_{T}$ and $h.$
\end{proposition}

\subsection{Well-posedness of forward stochastic fourth order Schr\"{o}dinger equation with nonhomogeneous boundary value contidion}
\par
Now, refering to \cite{0,Z1,B2,L1,L2} for the transposition method, we can give a meaning to (\ref{1.3}).
\begin{definition} A stochastic process $y\in C_{\mathcal {F}}([0,T];L^{2}(\Omega;X_{3}^{\prime}))$ is said to be a solution
of (\ref{1.3}) if for every $\tau\in [0,T]$ and every $z_{\tau}\in L^{2}\mathcal (\Omega,\mathcal{F}_{\tau},P; X_{3})$
it holds that
\begin{eqnarray*}
\begin{array}{l}
\begin{array}{llll}
\displaystyle
E(y(\tau),\overline{z}_{\tau})_{X_{3}^{\prime},X_{3}}-E(y_{0},\overline{z}(\cdot,0))_{X_{3}^{\prime},X_{3}}
\\=\displaystyle E\int_{0}^{\tau}i(u_{1}(t)\overline{z}_{xxx}(0,t)-u_{2}(t)\overline{z}_{xx}(0,t))dt
\\~~~~~~~~~~~~~~~~~~~+\displaystyle E\int_{0}^{\tau}[-i(f,\overline{z})_{X_{3}^{\prime},X_{3}}+(g,\overline{Z})_{X_{3}^{\prime},X_{3}}]dt,
\end{array}
\end{array}
\end{eqnarray*}
where $(z,Z)$ is the solution to (\ref{1.48}) with terminal state $z_{\tau}$.
\end{definition}

\begin{proposition}\label{P}
Let $y_{0}\in X_{3}^{\prime}, f\in L^{1}_{\mathcal{F}}(0,T;X_{3}^{\prime})$ and $g\in L^{2}_{\mathcal{F}}(0,T;X_{3}^{\prime})$ be given. Then (\ref{1.3}) admits a unique solution $y\in C_{\mathcal {F}}([0,T];L^{2}(\Omega;X_{3}^{\prime}))$ such that
\begin{eqnarray}\label{1.32}
\begin{array}{l}
\begin{array}{llll}
\|y\|_{C_{\mathcal {F}}([0,T];L^{2}(\Omega;X_{3}^{\prime}))}\leq C(\|y_{0}\|_{X_{3}^{\prime}}+\|f\|_{L_{\mathcal
{F}}^{1}(0,T;X_{3}^{\prime})}+\|g\|_{L_{\mathcal {F}}^{2}(0,T;X_{3}^{\prime})}
\\~~~~~~~~~~~~~~~~~~~~~~~~~~~~~~~~~~~~~~~~+\|u_{1}\|_{L^{2}_{\mathcal{F}}(\Omega,L^{2}(0,T))}+\|u_{2}\|_{L^{2}_{\mathcal{F}}(\Omega,L^{2}(0,T))}).
\end{array}
\end{array}
\end{eqnarray}
where $C=C(a,b,T).$
\end{proposition}
\begin{proof}
\par The main idea in this part comes from \cite{0,Z1,B2}.
\par Let us define a linear functional $F$ on $L^{2}\mathcal (\Omega,\mathcal{F}_{\tau},P; X_{3})$ as
\begin{eqnarray*}
\begin{array}{l}
\begin{array}{llll}
F(z_{\tau})=\displaystyle E(y_{0},\overline{z}(\cdot,0))_{X_{3}^{\prime},X_{3}} +E\int_{0}^{\tau}i(u_{1}(t)\overline{z}_{xxx}(0,t)-u_{2}(t)\overline{z}_{xx}(0,t))dt
\\~~~~~~~~~~~+\displaystyle E\int_{0}^{\tau}[-i(f,\overline{z})_{X_{3}^{\prime},X_{3}}+(g,\overline{Z})_{X_{3}^{\prime},X_{3}}]dt.
\end{array}
\end{array}
\end{eqnarray*}
Applying Proposition \ref{P2} iii) with $s=3,h=0$ and Proposition \ref{P4} with $h=0$ to (\ref{1.48}), we can obtain that the solution $(z,Z)$ for (\ref{1.48})
satisfies
\begin{eqnarray*}
\begin{array}{l}
\begin{array}{llll}
\|z\|_{L^{2}_{\mathcal {F}}(\Omega;C([0,\tau];X_{3}))}+\|Z\|_{L^{2}_{\mathcal {F}}(0,\tau;X_{3})}\leq C\|z_{\tau}\|_{L^{2}\mathcal (\Omega,\mathcal{F}_{\tau},P; X_{3})},
\\\|z_{xx}(0,\cdot)\|_{L^{2}_{\mathcal{F}}(\Omega,L^{2}(0,\tau))} +\|z_{xxx}(0,\cdot)\|_{L^{2}_{\mathcal{F}}(\Omega,L^{2}(0,\tau))} \leq \displaystyle C\|z_{\tau}\|_{L^{2}_{\mathcal{F}} (\Omega,\mathcal
{F}_{\tau},P;X_{3})}.
\end{array}
\end{array}
\end{eqnarray*}
Thus
\begin{eqnarray*}
\begin{array}{l}
\begin{array}{llll}
|F(z_{\tau})|\leq C(\|y_{0}\|_{X_{3}^{\prime}}+\|f\|_{L_{\mathcal
{F}}^{1}(0,T;X_{3}^{\prime})}+\|g\|_{L_{\mathcal {F}}^{2}(0,T;X_{3}^{\prime})}
\\~~~~~~~~~~~~~~~~+\|u_{1}\|_{L^{2}_{\mathcal{F}}(\Omega,L^{2}(0,T))}+\|u_{2}\|_{L^{2}_{\mathcal{F}}(\Omega,L^{2}(0,T))})\|z_{\tau}\|_{L^{2}\mathcal (\Omega,\mathcal
{F}_{\tau},P;X_{3})}.
\end{array}
\end{array}
\end{eqnarray*}
Hence, we get that $F$ is bounded linear functional on $L^{2}\mathcal (\Omega,\mathcal{F}_{\tau},P; X_{3}).$ By the Riesz Representation Theorem, we know that there exists a unique $y_{\tau}\in L^{2}\mathcal (\Omega,\mathcal{F}_{\tau},P;X_{3}^{\prime} )$ such that
\begin{eqnarray*}
\begin{array}{l}
\begin{array}{llll}
F(z_{\tau})=E(y_{\tau},z_{\tau})_{X_{3}^{\prime},X_{3}}
\end{array}
\end{array}
\end{eqnarray*}
for any $z_{\tau} \in L^{2}\mathcal (\Omega,\mathcal{F}_{\tau},P; X_{3})$
and
\begin{eqnarray*}
\begin{array}{l}
\begin{array}{llll}
\|y_{\tau}\|_{L^{2}\mathcal (\Omega,\mathcal{F}_{\tau},P; X_{3}^{\prime} )}
\leq C(\|y_{0}\|_{X_{3}^{\prime}}+\|f\|_{L_{\mathcal
{F}}^{1}(0,T;X_{3}^{\prime})}+\|g\|_{L_{\mathcal {F}}^{2}(0,T;X_{3}^{\prime})}
\\~~~~~~~~~~~~~~~~~~~~~~~~~~~~~~~~~~~~+\|u_{1}\|_{L^{2}_{\mathcal{F}}(\Omega,L^{2}(0,T))}+\|u_{2}\|_{L^{2}_{\mathcal{F}}(\Omega,L^{2}(0,T))}).
\end{array}
\end{array}
\end{eqnarray*}
for any $\tau \in (0,T).$
\par
Define a process $y(\cdot)$ by
\begin{eqnarray*}
y(\tau)=y_{\tau}
\end{eqnarray*}
for any $\tau \in (0,T).$
\par
Now we prove that $y\in C_{\mathcal {F}}([0,T];L^{2}(\Omega;X_{3}^{\prime})).$
\par
Indeed, let $\tau\in [0,T)$ and $\xi\in L^{2}\mathcal (\Omega,\mathcal{F}_{T},P;X_{3}).$ Consider the following forward random
Schr\"{o}dinger equation
\begin{eqnarray}\label{1.18}
\begin{array}{l}
\left\{
\begin{array}{llll}
i\widetilde{z}_{t}+\widetilde{z}_{xxxx}=\overline{a}\widetilde{z}
\\\widetilde{z}(0,t)=0=\widetilde{z}(1,t)
\\\widetilde{z}_{x}(0,t)=0=\widetilde{z}_{x}(1,t)
\\\widetilde{z}(\tau)=E(\xi\mid \mathcal{F}_{\tau})
\end{array}
\right.
\end{array}
\begin{array}{lll}
{\rm{in}}~I\times(\tau,\tau+\delta),\\
{\rm{in}}~(\tau,\tau+\delta),\\
{\rm{in}}~(\tau,\tau+\delta),\\
{\rm{in}}~I,
\end{array}
\end{eqnarray}
with $\delta>0$ satisfying that $\tau+\delta<T.$
\par
It is easy to see that
\begin{eqnarray*}
\begin{array}{l}
\begin{array}{llll}
\lim\limits_{\delta\rightarrow 0^{+}}E\|\widetilde{z}(\tau+\delta)-\widetilde{z}(\tau)\|^{2}_{X_{3}}=0.
\end{array}
\end{array}
\end{eqnarray*}
Further, since $\{\mathcal{F}_{t}\}_{t\geq0}$ is the natural filtration of $\{w(t)\}_{t\geq0},$ we have
\begin{eqnarray*}
\begin{array}{l}
\begin{array}{llll}
\lim\limits_{\delta\rightarrow 0^{+}}E\|E(\xi\mid \mathcal{F}_{\tau+\delta})-E(\xi\mid \mathcal{F}_{\tau})\|^{2}_{X_{3}}=0.
\end{array}
\end{array}
\end{eqnarray*}
Thus we have
\begin{eqnarray}\label{1.38}
\begin{array}{l}
\begin{array}{llll}
\lim\limits_{\delta\rightarrow 0^{+}}E\|\widetilde{z}(\tau+\delta)-E(\xi\mid \mathcal{F}_{\tau+\delta})\|^{2}_{X_{3}}=0.
\end{array}
\end{array}
\end{eqnarray}
Let $(z_{1},Z_{1}),(z_{2},Z_{2})$ and $(z_{3},Z_{3})$ satisfy
\begin{eqnarray*}
\begin{array}{l}
\left\{
\begin{array}{lll}
idz_{1}+z_{1xxxx}dt=(\overline{a}z_{1}-i\overline{b}Z_{1})dt+Z_{1}dw\\
z_{1}(0,t)=0=z_{1}(1,t)\\
z_{1x}(0,t)=0=z_{1x}(1,t)\\
z_{1}(x,\tau+\delta)=E(\xi\mid \mathcal{F}_{\tau+\delta})
\end{array}
\right.
\end{array}
\begin{array}{lll}\textrm{in}~I\times(0,\tau+\delta),
\\\textrm{in}~(0,\tau+\delta),\\\textrm{in}~(0,\tau+\delta),\\\textrm{in}~I,
\end{array}
\end{eqnarray*}
\begin{eqnarray*}
\begin{array}{l}
\left\{
\begin{array}{lll}
idz_{2}+z_{2xxxx}dt=(\overline{a}z_{2}-i\overline{b}Z_{2})dt+Z_{2}dw\\
z_{2}(0,t)=0=z_{2}(1,t)\\
z_{2x}(0,t)=0=z_{2x}(1,t)\\
z_{2}(x,\tau+\delta)=\widetilde{z}(\tau+\delta)
\end{array}
\right.
\end{array}
\begin{array}{lll}\textrm{in}~I\times(0,\tau+\delta),
\\\textrm{in}~(0,\tau+\delta),\\\textrm{in}~(0,\tau+\delta),\\\textrm{in}~I
\end{array}
\end{eqnarray*}
and
\begin{eqnarray*}
\begin{array}{l}
\left\{
\begin{array}{lll}
idz_{3}+z_{3xxxx}dt=(\overline{a}z_{3}-i\overline{b}Z_{3})dt+Z_{3}dw\\
z_{3}(0,t)=0=z_{3}(1,t)\\
z_{3x}(0,t)=0=z_{3x}(1,t)\\
z_{3}(x,\tau)=E(\xi\mid \mathcal{F}_{\tau})
\end{array}
\right.
\end{array}
\begin{array}{lll}\textrm{in}~I\times(0,\tau),
\\\textrm{in}~(0,\tau),\\\textrm{in}~(0,\tau),\\\textrm{in}~I.
\end{array}
\end{eqnarray*}
It follows from (\ref{1.38}), Proposition \ref{P2} and Proposition \ref{P4} that
\begin{eqnarray}\label{1.34}
\begin{array}{l}
\begin{array}{llll}
\lim\limits_{\delta\rightarrow 0^{+}}\|z_{1}-z_{2}\|_{L_{\mathcal {F}}^{2}(0,\tau;X_{3})}=0,
\\\lim\limits_{\delta\rightarrow 0^{+}}\|Z_{1}-Z_{2}\|_{L_{\mathcal {F}}^{2}(0,\tau;X_{3})}=0,
\\\lim\limits_{\delta\rightarrow 0^{+}}\|z_{1xx}(0,t)-z_{2xx}(0,t)\|_{L^{2}_{\mathcal{F}}(\Omega,L^{2}(0,\tau))}=0,
\\\lim\limits_{\delta\rightarrow 0^{+}}\|z_{1xxx}(0,t)-z_{2xxx}(0,t)\|_{L^{2}_{\mathcal{F}}(\Omega,L^{2}(0,\tau))}=0,
\\\lim\limits_{\delta\rightarrow 0^{+}}\|z_{1}(0)-z_{2}(0)\|_{X_{3}}=0.
\end{array}
\end{array}
\end{eqnarray}
By the uniqueness of the solution to (\ref{1.18}) and (\ref{1.48}), we have
\begin{eqnarray}\label{1.35}
\begin{array}{l}
\begin{array}{llll}
z_{3}=z_{2}
\\Z_{3}=Z_{2}
\end{array}
\end{array}
\begin{array}{lll}
{\rm{in}}~I\times(0,\tau),\\
{\rm{in}}~I\times(0,\tau).
\end{array}
\end{eqnarray}
From the definition of $y(\tau),$ we have
\begin{eqnarray*}
\begin{array}{l}
\begin{array}{llll}
E(y(\tau+\delta)-y(\tau),\xi)_{X_{3}^{\prime},X_{3}}
\\=E(y(\tau+\delta),\xi)_{X_{3}^{\prime},X_{3}}-E(y(\tau),\xi)_{X_{3}^{\prime},X_{3}}
\\=E(y(\tau+\delta),E(\xi\mid \mathcal{F}_{\tau+\delta}))_{X_{3}^{\prime},X_{3}}-E(y(\tau),E(\xi\mid \mathcal{F}_{\tau}))_{X_{3}^{\prime},X_{3}}
\\=\displaystyle E(y_{0},\overline{z}_{1}(0)-\overline{z}_{3}(0))_{X_{3}^{\prime},X_{3}}+ E\int_{0}^{\tau}[-i(f,\overline{z}_{1}-\overline{z}_{3})_{X_{3}^{\prime},X_{3}}+(g,\overline{Z}_{1}-\overline{Z}_{3})_{X_{3}^{\prime},X_{3}}]dt
\\\displaystyle+E\int_{0}^{\tau}i[u_{1}(t)(\overline{z}_{1xxx}(0,t)-\overline{z}_{3xxx}(0,t))-u_{2}(t)(\overline{z}_{1xx}(0,t)-\overline{z}_{3xx}(0,t))]dt
\\\displaystyle+E\int_{\tau}^{\tau+\delta}i(u_{1}(t)\overline{z}_{1xxx}(0,t)-u_{2}(t)\overline{z}_{1xx}(0,t))dt+ E\int_{\tau}^{\tau+\delta}[-i(f,\overline{z}_{1})_{X_{3}^{\prime},X_{3}}+(g,\overline{Z}_{1})_{X_{3}^{\prime},X_{3}}]dt
\end{array}
\end{array}
\end{eqnarray*}
This, together with (\ref{1.34}) and (\ref{1.35}), implies that
\begin{eqnarray*}
\lim\limits_{\delta\rightarrow 0^{+}}E(y(\tau+\delta)-y(\tau),\xi)_{X_{3}^{\prime},X_{3}}=0
\end{eqnarray*}
for any $\xi\in L^{2}\mathcal (\Omega,\mathcal{F}_{T},P;X_{3}).$
\par Similarly, we can show that for any $\tau\in (0,T],$
\begin{eqnarray*}
\lim\limits_{\delta\rightarrow 0^{-}}E(y(\tau+\delta)-y(\tau),\xi)_{X_{3}^{\prime},X_{3}}=0
\end{eqnarray*}
for any $\xi\in L^{2}\mathcal (\Omega,\mathcal{F}_{T},P;X_{3}).$
\par
Hence, we have $y\in C_{\mathcal {F}}([0,T];L^{2}(\Omega;X_{3}^{\prime})).$
\end{proof}

\section{An identity for a stochastic fourth order Schr\"{o}dinger operator}
\par
In this section, we obtain an identity for a stochastic fourth order Schr\"{o}dinger operator, which plays a key role in the
proof of Theorem \ref{T1}, Theorem \ref{T3} and Theorem \ref{T6}.
\begin{theorem}\label{T2}
Let $l\in C^{\infty}(\mathbb{R}\times\mathbb{R},\mathbb{R})$ and $\theta=e^{l}.$
Assume that $y$ is a continuous $H^{4}(\mathbb{R})$-valued $\{\mathcal
{F}_{t}\}_{t\geq0}$-adapted semi-martingale.
Put
\begin{eqnarray*}
\begin{array}{l}
\begin{array}{llll}
Ly=idy+y_{xxxx}dt,
\\u=\theta y,
\\\theta Ly=\theta(idy+y_{xxxx}dt)=I_{2}+I_{1}dt,
\\I_{1}=B_{0}u+C_{1}u_{x}+B_{2}u_{xx}+C_{3}u_{xxx},
\\I_{2}=idu+(u_{xxxx}+C_{2}u_{xx}+B_{1}u_{x}+C_{0}u+D_{0}u)dt,
\end{array}
\end{array}
\end{eqnarray*}
where the coefficients $B_{0},B_{1},B_{2},C_{0},C_{1},C_{2},C_{3}$ are real value functions and $D_{0}$ is a complex value function.
Then for a.e. $x\in \mathbb{R}$ and P-a.s. $\omega\in \Omega$, it holds that
\begin{eqnarray}\label{5}
\begin{array}{l}
\begin{array}{llll}
\theta(Ly\cdot \overline{I_{1}}+\overline{Ly}\cdot I_{1})
\\=2|I_{1}|^{2}dt+(|u|^{2}\{\cdot\}+|u_{x}|^{2}\{\cdot\}+|u_{xx}|^{2}\{\cdot\}+|u_{xxx}|^{2}\{\cdot\})dt
\\+(\{\cdot\}_{x}+\{\cdot\}_{xx}+\{\cdot\}_{xxx}+\{\cdot\}_{xxxx})dt+dM
\\+(\overline{u}du-ud\overline{u})\Big[i(B_{0}-\frac{1}{2}C_{1x}+\frac{1}{2}B_{2xx}-\frac{1}{2}C_{3xxx})\Big]+(\overline{u}_{x}du_{x}-u_{x}d\overline{u}_{x})\Big[i(\frac{3}{2}C_{3x}-B_{2})\Big]
\\+(dud\overline{u}_{x}-du_{x}d\overline{u})(-\frac{i}{2})(C_{1}-B_{2x}+C_{3xx})+(du_{x}d\overline{u}_{xx}-du_{xx}d\overline{u}_{x})(\frac{i}{2}C_{3})
\\+u\overline{u}_{x}\Big[(-\frac{i}{2})(C_{1t}-B_{2xt}+C_{3xxt})+C_{1}D_{0}\Big]dt+u_{x}\overline{u}\Big[\frac{i}{2}(C_{1t}-B_{2xt}+C_{3xxt})+C_{1}\overline{D_{0}}\Big]dt
\\+(u_{x}\overline{u}_{xx}-u_{xx}\overline{u}_{x})(\frac{i}{2}C_{3t})dt
\\+B_{2}D_{0}u\overline{u}_{xx}dt+B_{2}\overline{D_{0}}u_{xx}\overline{u}dt
+C_{3}D_{0}u\overline{u}_{xxx}dt+C_{3}\overline{D_{0}}u_{xxx}\overline{u}dt,
\end{array}
\end{array}
\end{eqnarray}
where
\begin{eqnarray*}
\begin{array}{l}
\begin{array}{llll}
|u|^{2}\{\cdot\}=|u|^{2}\{B_{0xxxx}+(B_{0}C_{2})_{xx}-(B_{0}B_{1})_{x}+2B_{0}C_{0}+B_{0}(D_{0}+\overline{D_{0}})-(C_{0}C_{1})_{x}
\\~~~~~~~~~~~~~~~~~~+(B_{2}C_{0})_{xx}-(C_{0}C_{3})_{xxx}\},
\\|u_{x}|^{2}\{\cdot\}=|u_{x}|^{2}\{-4B_{0xx}-2B_{0}C_{2}-C_{1xxx}-(C_{1}C_{2})_{x}+2B_{1}C_{1}-(B_{1}B_{2})_{x}-2B_{2}C_{0}
\\~~~~~~~~~~~~~~~~~~~~+(B_{1}C_{3})_{xx}+3(C_{0}C_{3})_{x}\},
\\|u_{xx}|^{2}\{\cdot\}=|u_{xx}|^{2}\{2B_{0}+3C_{1x}+B_{2xx}+2B_{2}C_{2}-(C_{2}C_{3})_{x}-2B_{1}C_{3}\},
\\|u_{xxx}|^{2}\{\cdot\}=|u_{xxx}|^{2}\{-2B_{2}-C_{3x}\},
\\\{\cdot\}_{x}=\{8B_{0x}|u_{x}|^{2}-4B_{0xxx}|u|^{2}-2(B_{0}C_{2})_{x}|u|^{2}+B_{0}B_{1}|u|^{2}+(\frac{-iC_{1}}{2})(ud\overline{u}-\overline{u}du)
\\~~~~~~~~~~+3C_{1xx}|u_{x}|^{2}-3C_{1}|u_{xx}|^{2}+C_{1}C_{2}|u_{x}|^{2}+C_{0}C_{1}|u|^{2}
\\~~~~~~~~~~+(-iB_{2})(u_{x}d\overline{u}-\overline{u}_{x}du)-\frac{1}{2}(-iB_{2x})(u d\overline{u}-\overline{u}du)-2B_{2x}|u_{xx}|^{2}+B_{1}B_{2}|u_{x}|^{2}
\\~~~~~~~~~~-2(B_{2}C_{0})_{x}|u|^{2}+(-iC_{3})(u_{xx}d\overline{u}-\overline{u}_{xx}du)-(-iC_{3x})(u_{x}d\overline{u}-\overline{u}_{x}du)
\\~~~~~~~~~~+\frac{1}{2}(-iC_{3xx})(ud\overline{u}-\overline{u}du)-\frac{1}{2}(-iC_{3})(u_{x}d\overline{u}_{x}-\overline{u}_{x}du_{x})+C_{3}|u_{xxx}|^{2}
\\~~~~~~~~~~+C_{2}C_{3}|u_{xx}|^{2}-2(B_{1}C_{3})_{x}|u_{x}|^{2}+3(C_{0}C_{3})_{xx}|u|^{2}-3C_{0}C_{3}|u_{x}|^{2}\}_{x},
\\
\{\cdot\}_{xx}=\{6B_{0xx}|u|^{2}-4B_{0}|u_{x}|^{2}+B_{0}C_{2}|u|^{2}-3C_{1x}|u_{x}|^{2}+B_{2}|u_{xx}|^{2}+B_{2}C_{0}|u|^{2}
\\~~~~~~~~~~+B_{1}C_{3}|u_{x}|^{2}-3(C_{0}C_{3})_{x}|u|^{2}\}_{xx},
\\
\{\cdot\}_{xxx}=\{-4B_{0x}|u|^{2}+C_{1}|u_{x}|^{2}+C_{0}C_{3}|u|^{2}\}_{xxx},
\\\{\cdot\}_{xxxx}=\{B_{0}|u|^{2}\}_{xxxx},
\\
M=\frac{i}{2}\{(C_{1}-B_{2x}+C_{3xx})(u\overline{u}_{x}-u_{x}\overline{u})-C_{3}(u_{x}\overline{u}_{xx}-u_{xx}\overline{u}_{x})\}.
\end{array}
\end{array}
\end{eqnarray*}

\end{theorem}
\begin{remark} The similar identity for stochastic second order Schr\"{o}dinger-like operator has been established in \cite{2}.
\end{remark}
\begin{proof}
From the definitions of $I_{1}$ and $I_{2},$ we know
\begin{eqnarray*}
\begin{array}{l}
\begin{array}{llll}
\overline{I}_{1}=\overline{B_{0}}\overline{u}+\overline{C_{1}}\overline{u}_{x}+\overline{B_{2}}\overline{u}_{xx}+\overline{C_{3}}\overline{u}_{xxx},
\\\overline{I}_{2}=-id\overline{u}+(\overline{u}_{xxxx}+\overline{C_{2}}\overline{u}_{xx}+\overline{B_{1}}\overline{u}_{x}+\overline{C_{0}}\overline{u}+\overline{D_{0}}\overline{u})dt.
\end{array}
\end{array}
\end{eqnarray*}
\par
According to
\begin{eqnarray*}
\begin{array}{l}
\begin{array}{llll}
\theta(Ly\cdot \overline{I_{1}}+\overline{Ly}\cdot I_{1})=2|I_{1}|^{2}dt+I_{1}\overline{I}_{2}+\overline{I}_{1}I_{2},
\end{array}
\end{array}
\end{eqnarray*}
we need to compute
\begin{eqnarray*}
I_{1}\overline{I}_{2}+\overline{I}_{1}I_{2}.
\end{eqnarray*}
First, we consider
\begin{eqnarray}\label{1.1}
B_{0}u\overline{I}_{2}+\overline{B_{0}u}I_{2}.
\end{eqnarray}
Each term in (\ref{1.1}) can be computed as follows:
\begin{eqnarray*}
\begin{array}{l}
\begin{array}{llll}
B_{0}u\cdot(-id\overline{u})+idu\cdot \overline{B_{0}}\overline{u}=iB_{0}(\overline{u}du-ud\overline{u}),
\\B_{0}u\cdot \overline{u}_{xxxx}+u_{xxxx}\cdot \overline{B_{0}}\overline{u}=(B_{0}|u|^{2})_{xxxx}-4(B_{0x}|u|^{2})_{xxx}+(6B_{0xx}|u|^{2}-4B_{0}|u_{x}|^{2})_{xx}
\\~~~~~~+(8B_{0x}|u_{x}|^{2}-4B_{0xxx}|u|^{2})_{x}+B_{0xxxx}|u|^{2}-4B_{0xx}|u_{x}|^{2}+2B_{0}|u_{xx}|^{2},
\\B_{0}u\cdot \overline{C_{2}}\overline{u}_{xx}+C_{2}u_{xx}\cdot \overline{B_{0}}\overline{u}=(B_{0}C_{2}|u|^{2})_{xx}-2((B_{0}C_{2})_{x}|u|^{2})_{x}
+(B_{0}C_{2})_{xx}|u|^{2}-2B_{0}C_{2}|u_{x}|^{2},
\\B_{0}u\cdot \overline{B_{1}}\overline{u}_{x}+B_{1}u_{x}\cdot \overline{B_{0}}\overline{u}=(B_{0}B_{1}|u|^{2})_{x}-(B_{0}B_{1})_{x}|u|^{2},
\\B_{0}u\cdot (\overline{C_{0}}\overline{u}+\overline{D_{0}}\overline{u})+(C_{0}u+D_{0}u)\cdot \overline{B_{0}}\overline{u}=[2B_{0}C_{0}+B_{0}(D_{0}+\overline{D_{0}})]|u|^{2}.
\end{array}
\end{array}
\end{eqnarray*}
Simarly, we consider
$C_{1}u_{x}\overline{I}_{2}+\overline{C_{1}u_{x}}I_{2},B_{2}u_{xx}\overline{I}_{2}+\overline{B_{2}u_{xx}}I_{2}$
and
$C_{3}u_{xxx}\overline{I}_{2}+\overline{C_{3}u_{xxx}}I_{2}.$
\par
By a similar argument, calculating each term in $C_{1}u_{x}\overline{I}_{2}+\overline{C_{1}u_{x}}I_{2},B_{2}u_{xx}\overline{I}_{2}+\overline{B_{2}u_{xx}}I_{2}$
and
$C_{3}u_{xxx}\overline{I}_{2}+\overline{C_{3}u_{xxx}}I_{2},$ we obtain
\begin{eqnarray*}
\begin{array}{l}
\begin{array}{llll}
C_{1}u_{x}\cdot (-id\overline{u})+idu\cdot \overline{C_{1}}\overline{u}_{x}=((\frac{-iC_{1}}{2})(ud\overline{u}-\overline{u}du))_{x}-\frac{1}{2}d((-iC_{1})u\overline{u}_{x}-(-iC_{1})u_{x}\overline{u})
\\~~~~~~+\frac{1}{2}(-iC_{1})(dud\overline{u}_{x}-du_{x}d\overline{u})+\frac{1}{2}(-iC_{1x})(\overline{u}du-u d\overline{u})+\frac{1}{2}(-iC_{1t})(u\overline{u}_{x}-u_{x}\overline{u})dt,
\\C_{1}u_{x}\cdot\overline{u}_{xxxx}+u_{xxxx}\cdot \overline{C_{1}}\overline{u}_{x}=(C_{1}|u_{x}|^{2})_{xxx}-3(C_{1x}|u_{x}|^{2})_{xx}+(3C_{1xx}|u_{x}|^{2}-3C_{1}|u_{xx}|^{2})_{x}
\\~~~~~~-C_{1xxx}|u_{x}|^{2}+3C_{1x}|u_{xx}|^{2},
\\C_{1}u_{x}\cdot \overline{C_{2}}\overline{u}_{xx}+C_{2}u_{xx}\cdot \overline{C_{1}}\overline{u}_{x}=(C_{1}C_{2}|u_{x}|^{2})_{x}-(C_{1}C_{2})_{x}|u_{x}|^{2},
\\C_{1}u_{x}\cdot \overline{B_{1}}\overline{u}_{x}+B_{1}u_{x}\cdot \overline{C_{1}}\overline{u}_{x}=2B_{1}C_{1}|u_{x}|^{2},
\\C_{1}u_{x}\cdot \overline{C_{0}}\overline{u}+C_{0}u\cdot \overline{C_{1}}\overline{u}_{x}=(C_{0}C_{1}|u|^{2})_{x}-(C_{0}C_{1})_{x}|u|^{2},
\\C_{1}u_{x}\cdot \overline{D_{0}}\overline{u}+D_{0}u\cdot \overline{C_{1}}\overline{u}_{x}=C_{1}(D_{0}u\overline{u}_{x}+\overline{D_{0}}u_{x}\overline{u}),
\\
B_{2}u_{xx}\cdot (-id\overline{u})+idu\cdot\overline{B_{2}}\overline{u}_{xx}=[(-iB_{2})(u_{x}d\overline{u}-\overline{u}_{x}du)-\frac{1}{2}(-iB_{2x})(ud\overline{u}-\overline{u}du)]_{x}
\\~~~~~~+\frac{1}{2}d[(-iB_{2x})(u\overline{u}_{x}-u_{x}\overline{u})]-\frac{1}{2}(-iB_{2x})(dud\overline{u}_{x}-du_{x}d\overline{u})-\frac{1}{2}(-iB_{2xx})(\overline{u}du-ud\overline{u})
\\~~~~~~-\frac{1}{2}(-iB_{2xt})(\overline{u}_{x}u-u_{x}\overline{u})dt-(-iB_{2})(u_{x}d\overline{u}_{x}-\overline{u}_{x}d u_{x}),
\\B_{2}u_{xx}\cdot\overline{u}_{xxxx}+u_{xxxx}\cdot\overline{B_{2}}\overline{u}_{xx}=(B_{2}|u_{xx}|^{2})_{xx}-2(B_{2x}|u_{xx}|^{2})_{x}+B_{2xx}|u_{xx}|^{2}-2B_{2}|u_{xxx}|^{2},
\\B_{2}u_{xx}\cdot \overline{C_{2}}\overline{u}_{xx}+C_{2}u_{xx}\cdot \overline{B_{2}}\overline{u}_{xx}=2B_{2}C_{2}|u_{xx}|^{2},
\\B_{2}u_{xx}\cdot \overline{B_{1}}\overline{u}_{x}+B_{1}u_{x}\cdot \overline{B_{2}}\overline{u}_{xx}=(B_{1}B_{2}|u_{x}|^{2})_{x}-(B_{1}B_{2})_{x}|u_{x}|^{2},
\\B_{2}u_{xx}\cdot \overline{C_{0}}\overline{u}+C_{0}u\cdot \overline{B_{2}}\overline{u}_{xx}=(B_{2}C_{0}|u|^{2})_{xx}-2((B_{2}C_{0})_{x}|u|^{2})_{x}+(B_{2}C_{0})_{xx}|u|^{2}-2B_{2}C_{0}|u_{x}|^{2},
\\B_{2}u_{xx}\cdot \overline{D_{0}}\overline{u}+D_{0}u\cdot \overline{B_{2}}\overline{u}_{xx}=B_{2}(D_{0}\overline{u}_{xx}u+\overline{D_{0}}\overline{u}u_{xx}).
\\
C_{3}u_{xxx}\cdot (-id\overline{u})+idu\cdot\overline{C_{3}}\overline{u}_{xxx}=[(-iC_{3})(u_{xx}d\overline{u}-\overline{u}_{xx}du)-(-iC_{3x})(u_{x}d\overline{u}-\overline{u}_{x}du)
\\~~~~~~+\frac{1}{2}(-iC_{3xx})(ud\overline{u}-\overline{u}du)-\frac{1}{2}(-iC_{3})(u_{x}d\overline{u}_{x}-\overline{u}_{x}du_{x})]_{x}
\\~~~~~~+\frac{1}{2}d[(-iC_{3})(u_{x}\overline{u}_{xx}-\overline{u}_{x}u_{xx})-(-iC_{3})_{xx}(u\overline{u}_{x}-\overline{u}u_{x})]-\frac{1}{2}(-iC_{3})(du_{x}d\overline{u}_{xx}-du_{xx}d\overline{u}_{x})
\\~~~~~~-\frac{3}{2}(-iC_{3})_{x}(\overline{u}_{x}du_{x}-u_{x}d\overline{u}_{x})+\frac{1}{2}(-iC_{3})_{xx}(dud\overline{u}_{x}-du_{x}d\overline{u})
+\frac{1}{2}(-iC_{3})_{xxx}(\overline{u}du-ud\overline{u})
\\~~~~~~+\frac{1}{2}(-iC_{3})_{xxt}(u\overline{u}_{x}-u_{x}\overline{u})dt-\frac{1}{2}(-iC_{3})_{t}(u_{x}\overline{u}_{xx}-u_{xx}\overline{u}_{x})dt,
\\C_{3}u_{xxx}\cdot\overline{u}_{xxxx}+u_{xxxx}\cdot\overline{C_{3}}\overline{u}_{xxx}=(C_{3}|u_{xxx}|^{2})_{x}-C_{3x}|u_{xxx}|^{2},
\\C_{3}u_{xxx}\cdot \overline{C_{2}}\overline{u}_{xx}+C_{2}u_{xx}\cdot \overline{C_{3}}\overline{u}_{xxx}=(C_{2}C_{3}|u_{xx}|^{2})_{x}-(C_{2}C_{3})_{x}|u_{xx}^{2}|,
\\C_{3}u_{xxx}\cdot \overline{B_{1}}\overline{u}_{x}+B_{1}u_{x}\cdot \overline{C_{3}}\overline{u}_{xxx}=(B_{1}C_{3}|u_{x}|^{2})_{xx}-2((B_{1}C_{3})_{x}|u_{x}|^{2})_{x}
+(B_{1}C_{3})_{xx}|u_{x}|^{2}-2B_{1}C_{3}|u_{xx}|^{2},
\\C_{3}u_{xxx}\cdot \overline{C_{0}}\overline{u}+C_{0}u\cdot\overline{C_{3}}\overline{u}_{xxx}=(C_{0}C_{3}|u|^{2})_{xxx}-3((C_{0}C_{3})_{x}|u|^{2})_{xx}+(3(C_{0}C_{3})_{xx}|u|^{2}-3C_{0}C_{3}|u_{x}|^{2})_{x}
\\~~~~~~-(C_{0}C_{3})_{xxx}|u|^{2}+3(C_{0}C_{3})_{x}|u_{x}|^{2},
\\C_{3}u_{xxx}\cdot \overline{D_{0}}\overline{u}+D_{0}u\cdot \overline{C_{3}}\overline{u}_{xxx}=C_{3}(D_{0}\overline{u}_{xxx}u+\overline{D_{0}}\overline{u}u_{xxx}).
\end{array}
\end{array}
\end{eqnarray*}
Taking into account the above equations, we obtain the following equation
\begin{eqnarray*}
\begin{array}{l}
\begin{array}{llll}
I_{1}\overline{I}_{2}+\overline{I}_{1}I_{2}
\\=|u|^{2}\{B_{0xxxx}+(B_{0}C_{2})_{xx}-(B_{0}B_{1})_{x}+2B_{0}C_{0}+B_{0}(D_{0}+\overline{D_{0}})-(C_{0}C_{1})_{x}
\\~~~~~~~~~~~+(B_{2}C_{0})_{xx}-(C_{0}C_{3})_{xxx}\}dt
\\+|u_{x}|^{2}\{-4B_{0xx}-2B_{0}C_{2}-C_{1xxx}-(C_{1}C_{2})_{x}+2B_{1}C_{1}-(B_{1}B_{2})_{x}-2B_{2}C_{0}+(B_{1}C_{3})_{xx}
\\~~~~~~~~~~~+3(C_{0}C_{3})_{x}\}dt
\\+|u_{xx}|^{2}\{2B_{0}+3C_{1x}+B_{2x}+2B_{2}C_{2}-(C_{2}C_{3})_{x}-2B_{1}C_{3})_{xx}\}dt
\\+|u_{xxx}|^{2}\{-2B_{2}-C_{3x}\}dt
\\+\{\cdot\}_{x}dt+\{\cdot\}_{xx}dt+\{\cdot\}_{xxx}dt+\{\cdot\}_{xxxx}dt
\\+\frac{i}{2}d\{(C_{1}-B_{2x}+C_{3xx})(u\overline{u}_{x}-u_{x}\overline{u})-C_{3}(u_{x}\overline{u}_{xx}-u_{xx}\overline{u}_{x})\}
\\+(\overline{u}du-ud\overline{u})(i(B_{0}-\frac{1}{2}C_{1x}+\frac{1}{2}B_{2xx}-\frac{1}{2}C_{3xxx}))+(\overline{u}_{x}du_{x}-u_{x}d\overline{u}_{x})(i(\frac{3}{2}C_{3x}-B_{2}))
\\+(dud\overline{u}_{x}-du_{x}d\overline{u})(-\frac{i}{2})(C_{1}-B_{2x}+C_{3xx})+(du_{x}d\overline{u}_{xx}-du_{xx}d\overline{u}_{x})(\frac{i}{2}C_{3})
\\+(u\overline{u}_{x}-u_{x}\overline{u})(-\frac{i}{2})(C_{1t}-B_{2xt}+C_{3xxt})dt+(u_{x}\overline{u}_{xx}-u_{xx}\overline{u}_{x})(\frac{i}{2}C_{3t})dt
\\+C_{1}D_{0}u\overline{u}_{x}dt+C_{1}\overline{D_{0}}u_{x}\overline{u}dt+B_{2}D_{0}u\overline{u}_{xx}dt+B_{2}\overline{D_{0}}u_{xx}\overline{u}dt
+C_{3}D_{0}u\overline{u}_{xxx}dt+C_{3}\overline{D_{0}}u_{xxx}\overline{u}dt,
\end{array}
\end{array}
\end{eqnarray*}
this implies (\ref{5}).
\end{proof}
Direct computation shows that
\begin{eqnarray}\label{19}
\theta(idy+y_{xxxx}dt)=idu+(A_{0}u+A_{1}u_{x}+A_{2}u_{xx}+A_{3}u_{xxx}+u_{xxxx})dt,
\end{eqnarray}
where
\begin{eqnarray*}
\begin{array}{l}
\begin{array}{llll}
A_{0}=l_{x}^{4}+4l_{x}l_{xxx}-l_{xxxx}-6l_{x}^{2}l_{xx}+3l_{xx}^{2}-il_{t},
\\A_{1}=-4l_{x}^{3}+12l_{x}l_{xx}-4l_{xxx},
\\A_{2}=6l_{x}^{2}-6l_{xx},
\\A_{3}=-4l_{x}.
\end{array}
\end{array}
\end{eqnarray*}
We have the following corollary.
\begin{corollary}\label{C3}
Let
\begin{eqnarray*}
\begin{array}{l}
\begin{array}{llll}
B_{1}=8l_{x}l_{xx}-4l_{xxx}, ~~B_{2}=-6l_{xx},~~C_{0}=l_{x}^{4}+2l_{x}l_{xxx}-2l_{xxxx}+l_{xx}^{2},
\\C_{1}=-4l_{x}^{3}+4l_{x}l_{xx},~~C_{2}=6l_{x}^{2},~~C_{3}=-4l_{x},~~D_{0}=-il_{t}.
\end{array}
\end{array}
\end{eqnarray*}
We have the following inequality
\begin{eqnarray}\label{10}
\begin{array}{l}
\begin{array}{llll}
\displaystyle E\int_{Q}\Big\{(|u|^{2}\{\cdot\}+|u_{x}|^{2}\{\cdot\}+|u_{xx}|^{2}\{\cdot\}+|u_{xxx}|^{2}\{\cdot\})dt \\~~~~~~~~~~~~~~~~~~~+(\{\cdot\}_{x}+\{\cdot\}_{xx}+\{\cdot\}_{xxx}+\{\cdot\}_{xxxx})dt+dM\Big\}dx
\\\leq-\displaystyle E\int_{Q}\Big\{u\overline{u}_{x}[(-\frac{i}{2})(C_{1t}-B_{2xt}+C_{3xxt})+C_{1}D_{0}]dt
\\~~~~~~~~~~~~+u_{x}\overline{u}[\frac{i}{2}(C_{1t}-B_{2xt}+C_{3xxt})+C_{1}\overline{D_{0}}]dt+(u_{x}\overline{u}_{xx}-u_{xx}\overline{u}_{x})(\frac{i}{2}C_{3t})dt
\\~~~~~~~~~~~~+B_{2}D_{0}u\overline{u}_{xx}dt+B_{2}\overline{D_{0}}u_{xx}\overline{u}dt
+C_{3}D_{0}u\overline{u}_{xxx}dt+C_{3}\overline{D_{0}}u_{xxx}\overline{u}dt
\\~~~~~~~~~~~~+(dud\overline{u}_{x}-du_{x}d\overline{u})(-\frac{i}{2})(C_{1}-B_{2x}+C_{3xx})+(du_{x}d\overline{u}_{xx}-du_{xx}d\overline{u}_{x})(\frac{i}{2}C_{3})\Big\}dx
\\+\displaystyle E\int_{Q}\theta^{2}|f|^{2}dxdt,
\end{array}
\end{array}
\end{eqnarray}
where $|u|^{2}\{\cdot\},|u_{x}|^{2}\{\cdot\},|u_{xx}|^{2}\{\cdot\},|u_{xxx}|^{2}\{\cdot\} ,\{\cdot\}_{x},\{\cdot\}_{xx},\{\cdot\}_{xxx},\{\cdot\}_{xxxx}$ and $M$ are the same as in Theorem \ref{T2}.
\end{corollary}
\begin{proof}
\par
It follows from (\ref{5}) and the fact
\begin{eqnarray*}
\begin{array}{l}
\begin{array}{llll}
B_{0}-\frac{1}{2}C_{1x}+\frac{1}{2}B_{2xx}-\frac{1}{2}C_{3xxx}=0
\\B_{2}-\frac{3}{2}C_{3x}=0
\end{array}
\end{array}
\end{eqnarray*}
that
\begin{eqnarray*}
\begin{array}{l}
\begin{array}{llll}
\displaystyle E\int_{Q}\theta(Ly\cdot \overline{I_{1}}+\overline{Ly}\cdot I_{1})dx
\\=\displaystyle E\int_{Q}\Big\{2|I_{1}|^{2}dt+(|u|^{2}\{\cdot\}+|u_{x}|^{2}\{\cdot\}+|u_{xx}|^{2}\{\cdot\}+|u_{xxx}|^{2}\{\cdot\})dt
\\+(\{\cdot\}_{x}+\{\cdot\}_{xx}+\{\cdot\}_{xxx}+\{\cdot\}_{xxxx})dt+dM
\\+(dud\overline{u}_{x}-du_{x}d\overline{u})(-\frac{i}{2})(C_{1}-B_{2x}+C_{3xx})+(du_{x}d\overline{u}_{xx}-du_{xx}d\overline{u}_{x})(\frac{i}{2}C_{3})
\\+u\overline{u}_{x}[(-\frac{i}{2})(C_{1t}-B_{2xt}+C_{3xxt})+C_{1}D_{0}]dt
\\+u_{x}\overline{u}[\frac{i}{2}(C_{1t}-B_{2xt}+C_{3xxt})+C_{1}\overline{D_{0}}]dt+(u_{x}\overline{u}_{xx}-u_{xx}\overline{u}_{x})(\frac{i}{2}C_{3t})dt
\\+B_{2}D_{0}u\overline{u}_{xx}dt+B_{2}\overline{D_{0}}u_{xx}\overline{u}dt
+C_{3}D_{0}u\overline{u}_{xxx}dt+C_{3}\overline{D_{0}}u_{xxx}\overline{u}dt\Big\}dx,
\end{array}
\end{array}
\end{eqnarray*}
thus
\begin{eqnarray*}
\begin{array}{l}
\begin{array}{llll}
\displaystyle E\int_{Q}\Big\{2|I_{1}|^{2}dt+(|u|^{2}\{\cdot\}+|u_{x}|^{2}\{\cdot\}+|u_{xx}|^{2}\{\cdot\}+|u_{xxx}|^{2}\{\cdot\})dt
\\~~~~~~~~~~~+(\{\cdot\}_{x}+\{\cdot\}_{xx}+\{\cdot\}_{xxx}+\{\cdot\}_{xxxx})dt+dM\Big\}dx
\\=-\displaystyle E\int_{Q}\Big\{
(dud\overline{u}_{x}-du_{x}d\overline{u})(-\frac{i}{2})(C_{1}-B_{2x}+C_{3xx})+(du_{x}d\overline{u}_{xx}-du_{xx}d\overline{u}_{x})(\frac{i}{2}C_{3})
\\~~~~~~~~~~~~+u\overline{u}_{x}[(-\frac{i}{2})(C_{1t}-B_{2xt}+C_{3xxt})+C_{1}D_{0}]dt
\\~~~~~~~~~~~~+u_{x}\overline{u}[\frac{i}{2}(C_{1t}-B_{2xt}+C_{3xxt})+C_{1}\overline{D_{0}}]dt+(u_{x}\overline{u}_{xx}-u_{xx}\overline{u}_{x})(\frac{i}{2}C_{3t})dt
\\~~~~~~~~~~~~+B_{2}D_{0}u\overline{u}_{xx}dt+B_{2}\overline{D_{0}}u_{xx}\overline{u}dt+C_{3}D_{0}u\overline{u}_{xxx}dt+C_{3}\overline{D_{0}}u_{xxx}\overline{u}dt\Big\}dx
\\~~~~+\displaystyle E\int_{Q}\theta(Ly\cdot \overline{I_{1}}+\overline{Ly}\cdot I_{1})dx.
\end{array}
\end{array}
\end{eqnarray*}
Noting
\begin{eqnarray*}
\begin{array}{l}
\begin{array}{llll}
\displaystyle E\int_{Q}\theta(Ly\cdot \overline{I_{1}}+\overline{Ly}\cdot I_{1})dx=\displaystyle E\int_{Q}\theta((fdt+gdw)\cdot \overline{I_{1}}+\overline{(fdt+gdw)}\cdot I_{1})dx
\\=\displaystyle E\int_{Q}\theta(f\overline{I_{1}}+\overline{f}I_{1})dxdt\leq 2E\int_{Q}|I_{1}|^{2}dxdt+\frac{1}{2}E\int_{Q}\theta^{2}|f|^{2}dxdt,
\end{array}
\end{array}
\end{eqnarray*}
we can obtain (\ref{10}).
\end{proof}
\section{Proof of Theorem \ref{T1}, Corollary \ref{C4} and Corollary \ref{C1}}
\subsection{Proof of Theorem \ref{T1}}
\par
\textbf{Step 1.}
 We shall prove the following estimate
\begin{eqnarray}\label{12}
\begin{array}{l}
\begin{array}{llll}
~~~\displaystyle E\int_{Q}(\lambda\widehat{\varphi}
\widehat{\theta}^{2}|y_{xxx}|^{2}+\lambda^{3}\widehat{\varphi}^{3}\widehat{\theta}^{2}|y_{xx}|^{2}+\lambda^{5}\widehat{\varphi}^{5}\widehat{\theta}^{2}|y_{x}|^{2}+\lambda^{7}\widehat{\varphi}^{7}\widehat{\theta}^{2}|y|^{2})dxdt
\\\leq\displaystyle
C[E\int_{Q^{I_{0}}}(\lambda\widehat{\varphi}
\widehat{\theta}^{2}|y_{xxx}|^{2}+\lambda^{3}\widehat{\varphi}^{3}\widehat{\theta}^{2}|y_{xx}|^{2}+\lambda^{5}\widehat{\varphi}^{5}\widehat{\theta}^{2}|y_{x}|^{2}+\lambda^{7}\widehat{\varphi}^{7}\widehat{\theta}^{2}|y|^{2}
)dxdt
\\~~~~~~~~~~+\displaystyle E\int_{Q}(\lambda^{4}\widehat{\varphi}^{4}\widehat{\theta}^{2}|g|^{2}
+\lambda^{2}\widehat{\varphi}^{2}\widehat{\theta}^{2}|g_{x}|^{2}+\widehat{\theta}^{2}|g_{xx}|^{2}+\widehat{\theta}^{2}|f|^{2})dxdt].
\end{array}
\end{array}
\end{eqnarray}
\par
Indeed,
applying Corollary \ref{C3} with $l=\widehat{l},$ then we have $\theta=\widehat{\theta},u=\widehat{u}=\widehat{\theta}y$
and
\begin{eqnarray}\label{1.2}
\begin{array}{l}
\begin{array}{llll}
\displaystyle E\int_{Q}\Big\{(|\widehat{u}|^{2}\{\cdot\}+|\widehat{u}_{x}|^{2}\{\cdot\}+|\widehat{u}_{xx}|^{2}\{\cdot\}+|\widehat{u}_{xxx}|^{2}\{\cdot\})dt \\~~~~~~~~~~~~~~~~~~~+(\{\cdot\}_{x}+\{\cdot\}_{xx}+\{\cdot\}_{xxx}+\{\cdot\}_{xxxx})dt+d\widehat{M}\Big\}dx
\\\leq-\displaystyle E\int_{Q}\{\widehat{u}\overline{\widehat{u}}_{x}[(-\frac{i}{2})(C_{1t}-B_{2xt}+C_{3xxt})+C_{1}D_{0}]dt
\\~~~~~~~~~~~~+\widehat{u}_{x}\overline{\widehat{u}}[\frac{i}{2}(C_{1t}-B_{2xt}+C_{3xxt})+C_{1}\overline{D_{0}}]dt+(\widehat{u}_{x}\overline{\widehat{u}}_{xx}-\widehat{u}_{xx}\overline{\widehat{u}}_{x})(\frac{i}{2}C_{3t})dt
\\~~~~~~~~~~~~+B_{2}D_{0}\widehat{u}\overline{\widehat{u}}_{xx}dt+B_{2}\overline{D_{0}}\widehat{u}_{xx}\overline{\widehat{u}}dt
+C_{3}D_{0}\widehat{u}\overline{\widehat{u}}_{xxx}dt+C_{3}\overline{D_{0}}\widehat{u}_{xxx}\overline{\widehat{u}}dt
\\~~~~~~~~~~~~+(d\widehat{u}d\overline{\widehat{u}}_{x}-d\widehat{u}_{x}d\overline{\widehat{u}})(-\frac{i}{2})(C_{1}-B_{2x}+C_{3xx})+(d\widehat{u}_{x}d\overline{\widehat{u}}_{xx}-d\widehat{u}_{xx}d\overline{\widehat{u}}_{x})(\frac{i}{2}C_{3})\}dx
\\+\displaystyle E\int_{Q}\widehat{\theta}^{2}|f|^{2}dxdt,
\end{array}
\end{array}
\end{eqnarray}
where $\{\cdot\}_{x},\{\cdot\}_{xx},\{\cdot\}_{xxx},\{\cdot\}_{xxxx}$ and $\widehat{M}$ are the same as in Theorem \ref{T2} with $u=\widehat{u}.$
\par
By the definitions of $\widehat{a},\widehat{\varphi},\widehat{\psi},$ it is
obvious that for $n\in \mathbb{N}$
\begin{equation*}
\begin{array}{l}
\begin{array}{llll}
|\partial_{x}^{n}\widehat{a}|\leq C(\widehat{\psi})\mu^{n}\widehat{\varphi},
~~|\partial_{x}^{n}\widehat{a}_{t}|\leq C(\widehat{\psi})T\mu^{n} \widehat{\varphi}^{2},
\\|\widehat{a}_{t}|\leq C T\widehat{\varphi}^{2},
~~|\widehat{a}_{tt}|\leq C T^{2}\widehat{\varphi}^{3}.
\end{array}
\end{array}
\end{equation*}
Observe that $\widehat{\varphi}\leq \frac{T^{2}}{4}\widehat{\varphi}^{2}\leq
\frac{T^{4}}{16}\widehat{\varphi}^{3}\leq \frac{T^{6}}{64}\widehat{\varphi}^{4}\leq
\frac{T^{8}}{256}\widehat{\varphi}^{5}\leq \frac{T^{10}}{1024}\widehat{\varphi}^{6}.$
\par
For the term $|\widehat{u}|^{2}\{\cdot\}$ in (\ref{1.2}), if we choose
$\lambda\geq \mu C(\widehat{\psi})(T+T^{2})$ with $C(\widehat{\psi})$ large enough, then
it holds that
\begin{eqnarray*}
\begin{array}{l}
\begin{array}{llll}
B_{0xxxx}+(B_{0}C_{2})_{xx}-(B_{0}B_{1})_{x}+2B_{0}C_{0}+B_{0}(D_{0}+\overline{D_{0}})-(C_{0}C_{1})_{x}
\\~~~~~~~~~~~+(B_{2}C_{0})_{xx}-(C_{0}C_{3})_{xxx}
\\=16\lambda^{7}\mu^{8}\widehat{\varphi}^{7}\widehat{\psi}_{x}^{8}+\widehat{R}_{0},
\end{array}
\end{array}
\end{eqnarray*}
where $|\widehat{R}_{0}|\leq C(\widehat{\psi})\lambda^{7}\mu^{7}\widehat{\varphi}^{7}.$
\par Namely
\begin{eqnarray}\label{6}
|\widehat{u}|^{2}\{\cdot\}=16\lambda^{7}\mu^{8}\widehat{\varphi}^{7}\widehat{\psi}_{x}^{8}|\widehat{u}|^{2}+\widehat{R}_{0}|\widehat{u}|^{2}.
\end{eqnarray}
 Using the
same method, we can obtain that
\begin{eqnarray}\label{8}
\begin{array}{l}
\begin{array}{llll}
|\widehat{u}_{x}|^{2}\{\cdot\}=80\lambda^{5}\mu^{6}\widehat{\varphi}^{5}\widehat{\psi}_{x}^{6}|\widehat{u}_{x}|^{2}+\widehat{R}_{1}|\widehat{u}_{x}|^{2},
\\|\widehat{u}_{xx}|^{2}\{\cdot\}=16\lambda^{3}\mu^{4}\widehat{\varphi}^{3}\widehat{\psi}_{x}^{4}|\widehat{u}_{xx}|^{2}+\widehat{R}_{2}|\widehat{u}_{xx}|^{2},
\\|\widehat{u}_{xxx}|^{2}\{\cdot\}=16\lambda\mu^{2}\widehat{\varphi}\widehat{\psi}_{x}^{2}|\widehat{u}_{xxx}|^{2}+\widehat{R}_{3}|\widehat{u}_{xxx}|^{2},
\end{array}
\end{array}
\end{eqnarray}
where
\begin{eqnarray*}
\begin{array}{l}
\begin{array}{llll}
|\widehat{R}_{1}|\leq C\lambda^{5}\mu^{5}\widehat{\varphi}^{5},
\\|\widehat{R}_{2}|\leq C\lambda^{3}\mu^{3}\widehat{\varphi}^{3},
\\|\widehat{R}_{3}|\leq C\lambda\mu\widehat{\varphi}.
\end{array}
\end{array}
\end{eqnarray*}
\par
Now, we estimate the term
$E\displaystyle\int_{Q}(\{\cdot\}_{x}+\{\cdot\}_{xx}+\{\cdot\}_{xxx}+\{\cdot\}_{xxxx})dxdt$
in (\ref{1.2}).
\par
Indeed, noting that
$y(0,t)=y(1,t)=y_{x}(0,t)=y_{x}(1,t)=0,$ we have
\begin{eqnarray*}\widehat{u}(0,t)=\widehat{u}(1,t)=\widehat{u}_{x}(0,t)=\widehat{u}_{x}(1,t)=0~~~\forall~t\in(0,T).\end{eqnarray*}
Thus
\begin{eqnarray*}
\begin{array}{l}
\begin{array}{llll}
~~~\displaystyle E\int_{Q}(\{\cdot\}_{x}+\{\cdot\}_{xx}+\{\cdot\}_{xxx}+\{\cdot\}_{xxxx})dxdt
\\=\displaystyle E\int_{Q}\{|\widehat{u}_{xx}|^{2}(-20\lambda^{3}\mu^{3}\widehat{\varphi}^{3}\widehat{\psi}^{3}_{x}+\widehat{r}_{1})+|\widehat{u}_{xxx}|^{2}(-4\lambda\mu\widehat{\varphi}\widehat{\psi}_{x})
\\
~~~+\widehat{u}_{xxx}\overline{\widehat{u}}_{xx}(-6\lambda\mu^{2}\widehat{\varphi}\widehat{\psi}^{2}_{x}+\widehat{r}_{2})+\widehat{u}_{xx}\overline{\widehat{u}}_{xxx}(-6\lambda\mu^{2}\widehat{\varphi}\widehat{\psi}^{2}_{x}+\widehat{r}_{2})\}_{x}dxdt
\\\triangleq \widehat{V}(1)-\widehat{V}(0),

\end{array}
\end{array}
\end{eqnarray*}
where
\begin{eqnarray*}
\begin{array}{l}
\begin{array}{llll}
|\widehat{r}_{1}|\leq C\lambda^{2}\mu^{3}\widehat{\varphi}^{2},
\\|\widehat{r}_{2}|\leq C\lambda\mu\widehat{\varphi}.
\end{array}
\end{array}
\end{eqnarray*}
It holds that for any
$\varepsilon>0,$ if we choose $\lambda\geq \mu
C(\varepsilon,\widehat{\psi})(T+T^{2})$  with $C(\varepsilon,\widehat{\psi})$ large
enough, then
\begin{eqnarray*}
\begin{array}{l}
\begin{array}{llll}
|\widehat{u}_{xxx}\overline{\widehat{u}}_{xx}(-6\lambda\mu^{2}\widehat{\varphi}\widehat{\psi}^{2}_{x}+\widehat{r}_{2})(1,t)+\widehat{u}_{xx}\overline{\widehat{u}}_{xxx}(-6\lambda\mu^{2}\widehat{\varphi}\widehat{\psi}^{2}_{x}+\widehat{r}_{2})(1,t)|
\\\leq \varepsilon \lambda^{3}\mu^{3}\widehat{\varphi}^{3}(1,t)|\widehat{u}_{xxx}(1,t)|^{2}+\varepsilon \lambda\mu\widehat{\varphi}(1,t)|\widehat{u}_{xx}(1,t)|^{2},
\\|\widehat{u}_{xxx}\overline{\widehat{u}}_{xx}(-6\lambda\mu^{2}\widehat{\varphi}\widehat{\psi}^{2}_{x}+\widehat{r}_{2})(0,t)+\widehat{u}_{xx}\overline{\widehat{u}}_{xxx}(-6\lambda\mu^{2}\widehat{\varphi}\widehat{\psi}^{2}_{x}+\widehat{r}_{2})(0,t)|
\\\leq \varepsilon \lambda^{3}\mu^{3}\widehat{\varphi}^{3}(0,t)|\widehat{u}_{xxx}(0,t)|^{2}+\varepsilon \lambda\mu\widehat{\varphi}(0,t)|\widehat{u}_{xx}(0,t)|^{2}.
\end{array}
\end{array}
\end{eqnarray*}
Note that $\widehat{\psi}_{x}(1)<0,~\widehat{\psi}_{x}(0)>0,$ if we choose $\varepsilon$
small sufficiently and $\lambda\geq \mu
C(\varepsilon,\widehat{\psi})(T+T^{2}),$ then there exist positive constants
$N_{1},N_{2},K_{1},K_{2}$ such that
\begin{eqnarray*}
\begin{array}{l}
\begin{array}{llll}
\widehat{V}(1)=\displaystyle E\int_{0}^{T}[|\widehat{u}_{xx}|^{2}(-20\lambda^{3}\mu^{3}\widehat{\varphi}^{3}\widehat{\psi}^{3}_{x}+\widehat{r}_{1})+|\widehat{u}_{xxx}|^{2}(-4\lambda\mu\widehat{\varphi}\widehat{\psi}_{x})
\\
~~~~~~+\widehat{u}_{xxx}\overline{\widehat{u}}_{xx}(-6\lambda\mu^{2}\widehat{\varphi}\widehat{\psi}^{2}_{x}+\widehat{r}_{2})+\widehat{u}_{xx}\overline{\widehat{u}}_{xxx}(-6\lambda\mu^{2}\widehat{\varphi}\widehat{\psi}^{2}_{x}+\widehat{r}_{2})](1,t)dt
\\~~~~~~\geq E\displaystyle
\int_{0}^{T}(-N_{1}\lambda^{3}\mu^{3}\widehat{\varphi}^{3}(1,t)\widehat{\psi}^{3}_{x}(1)|\widehat{u}_{xx}(1 ,t)|^{2}-K_{1}\lambda\mu\widehat{\varphi}(1,t)\widehat{\psi}_{x}(1)|\widehat{u}_{xxx}(1,t)|^{2})dt\\
~~~~~~\geq0,\\
\widehat{V}(0)=\displaystyle E\int_{0}^{T}[|\widehat{u}_{xx}|^{2}(-20\lambda^{3}\mu^{3}\widehat{\varphi}^{3}\widehat{\psi}^{3}_{x}+\widehat{r}_{1})+|\widehat{u}_{xxx}|^{2}(-4\lambda\mu\widehat{\varphi}\widehat{\psi}
_{x})
\\
~~~~~~~~~~~~~~~~+\widehat{u}_{xxx}\overline{\widehat{u}}_{xx}(-6\lambda\mu^{2}\widehat{\varphi}\widehat{\psi}^{2}_{x}+\widehat{r}_{2})+\widehat{u}_{xx}\overline{\widehat{u}}_{xxx}(-6\lambda\mu^{2}\widehat{\varphi}\widehat{\psi}^{2}_{x}+\widehat{r}_{2})](0,t)dt
\\~~~~~~\leq E\displaystyle
\int_{0}^{T}(-N_{2}\lambda^{3}\mu^{3}\widehat{\varphi}^{3}(0,t)\widehat{\psi}
^{3}_{x}(0)|\widehat{u}_{xx}(0,t)|^{2}-K_{2}\lambda\mu\widehat{\varphi}(0,t)\widehat{\psi}_{x}(0)|\widehat{u}_{xxx}(0,t)|^{2})dt
\\~~~~~~\leq 0.
\end{array}
\end{array}
\end{eqnarray*}
Thus, \begin{eqnarray}\label{9} \widehat{V}(1)-\widehat{V}(0)\geq 0.
\end{eqnarray}
\par
Noting that
$\lim\limits_{t\rightarrow0^{+}}\widehat{a}(\cdot,t)=\lim\limits_{t\rightarrow
T^{-}}\widehat{a}(\cdot,t)=-\infty$, we have
\begin{eqnarray*}\widehat{u}(x,0)=\widehat{u}(x,T)=\widehat{u}_{x}(x,0)=\widehat{u}_{x}(x,T)=0~~~\forall~x\in I.\end{eqnarray*}
It is obvious that
\begin{eqnarray*}
\begin{array}{l}
\begin{array}{llll}
\displaystyle E\int_{0}^{T}d\widehat{M}=0.
\end{array}
\end{array}
\end{eqnarray*}
It is a straightforward calculation to show that
\begin{eqnarray*}
\begin{array}{l}
\begin{array}{llll}
|(\widehat{u}\overline{\widehat{u}}_{x}-\widehat{u}_{x}\overline{\widehat{u}})(-\frac{i}{2})(C_{1t}-B_{2xt}+C_{3xxt})+C_{1}D_{0}\widehat{u}\overline{\widehat{u}}_{x}+C_{1}\overline{D_{0}}\widehat{u}_{x}\overline{\widehat{u}}|\leq C\lambda^{4}\mu^{3}\widehat{\varphi}^{5}|\widehat{u}||\widehat{u}_{x}|,
\\|(\widehat{u}_{x}\overline{\widehat{u}}_{xx}-\widehat{u}_{xx}\overline{\widehat{u}}_{x})(\frac{i}{2}C_{3t})|\leq C\lambda\mu\widehat{\varphi}^{2}|\widehat{u}_{x}||\widehat{u}_{xx}|,
\\|B_{2}D_{0}\widehat{u}\overline{\widehat{u}}_{xx}+B_{2}\overline{D_{0}}\widehat{u}_{xx}\overline{\widehat{u}}|\leq C\lambda^{2}\mu^{2}\widehat{\varphi}^{3}|\widehat{u}||\widehat{u}_{xx}|,
\\|C_{3}D_{0}\widehat{u}\overline{\widehat{u}}_{xxx}+C_{3}\overline{D_{0}}\widehat{u}_{xxx}\overline{\widehat{u}}|\leq C\lambda^{2}\mu\widehat{\varphi}^{3}|\widehat{u}||\widehat{u}_{xxx}|,
\end{array}
\end{array}
\end{eqnarray*}
thus,
\begin{eqnarray*}
\begin{array}{l}
\begin{array}{llll}
|\widehat{u}\overline{\widehat{u}}_{x}[(-\frac{i}{2})(C_{1t}-B_{2xt}+C_{3xxt})+C_{1}D_{0}]+\widehat{u}_{x}\overline{\widehat{u}}[-\frac{i}{2}(C_{1t}-B_{2xt}+C_{3xxt})+C_{1}\overline{D_{0}}]
\\+(\widehat{u}_{x}\overline{\widehat{u}}_{xx}-\widehat{u}_{xx}\overline{\widehat{u}}_{x})(\frac{i}{2}C_{3t})+B_{2}D_{0}\widehat{u}\overline{\widehat{u}}_{xx}+B_{2}\overline{D_{0}}\widehat{u}_{xx}\overline{\widehat{u}}
+C_{3}D_{0}\widehat{u}\overline{\widehat{u}}_{xxx}+C_{3}\overline{D_{0}}\widehat{u}_{xxx}\overline{\widehat{u}}|
\\\leq C(\lambda^{7}\mu^{7}\widehat{\varphi}^{7}|\widehat{u}|^{2}
+\lambda^{5}\mu^{5}\widehat{\varphi}^{5}|\widehat{u}_{x}|^{2}+\lambda^{3}\mu^{3}\widehat{\varphi}^{3}|\widehat{u}_{xx}|^{2}+\lambda\mu\widehat{\varphi} |\widehat{u}_{xxx}|^{2}).
\end{array}
\end{array}
\end{eqnarray*}
Moreover, we can deduce that
\begin{eqnarray}\label{11}
\begin{array}{l}
\begin{array}{llll}
|\displaystyle E\int_{Q}[(d\widehat{u}d\overline{\widehat{u}}_{x}-d\widehat{u}_{x}d\overline{\widehat{u}})(-\frac{i}{2})(C_{1}-B_{2x}+C_{3xx})+(d\widehat{u}_{x}d\overline{\widehat{u}}_{xx}-d\widehat{u}_{xx}d\overline{\widehat{u}}_{x})(\frac{i}{2}C_{3})]dx|
\\\leq C\displaystyle E\int_{Q}(\lambda^{3}\mu^{3}\widehat{\varphi}^{3}\widehat{\theta}^{2}|g||g_{x}|
+\lambda^{2}\mu^{2}\widehat{\varphi}^{2}\widehat{\theta}^{2}|g||g_{xx}|+\lambda\mu\widehat{\varphi}\widehat{\theta}^{2}|g_{x}||g_{xx}|)dxdt
\\\leq C\displaystyle E\int_{Q}(\lambda^{4}\mu^{4}\widehat{\varphi}^{4}\widehat{\theta}^{2}|g|^{2}
+\lambda^{2}\mu^{2}\widehat{\varphi}^{2}\widehat{\theta}^{2}|g_{x}|^{2}+\widehat{\theta}^{2}|g_{xx}|^{2})dxdt.
\end{array}
\end{array}
\end{eqnarray}
From (\ref{1.2})-(\ref{11}), we can obtain that
\begin{eqnarray*}
\begin{array}{l}
\begin{array}{llll}
\displaystyle E\int_{Q}(\lambda^{7}\mu^{8}\widehat{\varphi}^{7}\widehat{\psi}_{x}^{8}|\widehat{u}|^{2}+\lambda^{5}\mu^{6}\widehat{\varphi}^{5}\widehat{\psi}_{x}^{6}|\widehat{u}_{x}|^{2}
+\lambda^{3}\mu^{4}\widehat{\varphi}^{3}\widehat{\psi}_{x}^{4}|\widehat{u}_{xx}|^{2}+\lambda\mu^{2}\widehat{\varphi}\widehat{\psi}_{x}^{2}|\widehat{u}_{xxx}|^{2})dxdt
\\\leq C\displaystyle E\int_{Q}(\lambda^{4}\mu^{4}\widehat{\varphi}^{4}\widehat{\theta}^{2}|g|^{2}
+\lambda^{2}\mu^{2}\widehat{\varphi}^{2}\widehat{\theta}^{2}|g_{x}|^{2}+\widehat{\theta}^{2}|g_{xx}|^{2}+\widehat{\theta}
^{2}|f|^{2}
\\~~~~~~~~+\lambda^{7}\mu^{7}\widehat{\varphi}^{7}|\widehat{u}|^{2}
+\lambda^{5}\mu^{5}\widehat{\varphi}^{5}|\widehat{u}_{x}|^{2}+\lambda^{3}\mu^{3}\widehat{\varphi}^{3}|\widehat{u}_{xx}|^{2}+\lambda\mu\widehat{\varphi} |\widehat{u}_{xxx}|^{2})dxdt.
\end{array}
\end{array}
\end{eqnarray*}
Recall that $|\widehat{\psi}_{x}|>0$ in $\overline{I}\setminus I_{0},$ it
follows that
\begin{eqnarray*}
\begin{array}{l}
\begin{array}{llll}
\displaystyle E\int_{Q\setminus
Q^{I_{0}}}(\lambda^{7}\mu^{8}\widehat{\varphi}^{7}|\widehat{u}|^{2}+\lambda^{5}\mu^{6}\widehat{\varphi}^{5}|\widehat{u}_{x}|^{2}
+\lambda^{3}\mu^{4}\widehat{\varphi}^{3}|\widehat{u}_{xx}|^{2}+\lambda\mu^{2}\widehat{\varphi}|\widehat{u}_{xxx}|^{2})dxdt
\\\leq C(\widehat{\psi})\displaystyle E\int_{Q}(\lambda^{4}\mu^{4}\widehat{\varphi}^{4}\widehat{\theta}^{2}|g|^{2}
+\lambda^{2}\mu^{2}\widehat{\varphi}^{2}\widehat{\theta}^{2}|g_{x}|^{2}+\widehat{\theta}^{2}|g_{xx}|^{2}+\widehat{\theta}^{2}|f|^{2}
\\~~~~~~~~+\lambda^{7}\mu^{7}\widehat{\varphi}^{7}|\widehat{u}|^{2}
+\lambda^{5}\mu^{5}\widehat{\varphi}^{5}|\widehat{u}_{x}|^{2}+\lambda^{3}\mu^{3}\widehat{\varphi}^{3}|\widehat{u}_{xx}|^{2}+\lambda\mu\widehat{\varphi} |\widehat{u}_{xxx}|^{2})dxdt,
\end{array}
\end{array}
\end{eqnarray*}
from which if we choose $\mu_{0}= C(\widehat{\psi})+1,$ then it holds that
\begin{eqnarray*}
\begin{array}{l}
\begin{array}{llll}
\displaystyle E\int_{Q\setminus
Q^{I_{0}}}(\lambda^{7}\mu^{7}\widehat{\varphi}^{7}|\widehat{u}|^{2}
+\lambda^{5}\mu^{5}\widehat{\varphi}^{5}|\widehat{u}_{x}|^{2}+\lambda^{3}\mu^{3}\widehat{\varphi}^{3}|\widehat{u}_{xx}|^{2}+\lambda\mu\widehat{\varphi} |\widehat{u}_{xxx}|^{2})dxdt
\\\leq C_{1}(\widehat{\psi})E\displaystyle [\int_{Q}(\lambda^{4}\mu^{4}\widehat{\varphi}^{4}\theta
^{2}|g|^{2}
+\lambda^{2}\mu^{2}\widehat{\varphi}^{2}\widehat{\theta}^{2}|g_{x}|^{2}+\widehat{\theta}^{2}|g_{xx}|^{2}+\widehat{\theta}^{2}|f|^{2})dxdt
\\~~~~~~~~+\displaystyle\int_{Q^{I_{0}}}(\lambda^{7}\mu^{7}\widehat{\varphi}^{7}|\widehat{u}|^{2}
+\lambda^{5}\mu^{5}\widehat{\varphi}^{5}|\widehat{u}_{x}|^{2}+\lambda^{3}\mu^{3}\widehat{\varphi}^{3}|\widehat{u}_{xx}|^{2}+\lambda\mu\widehat{\varphi} |\widehat{u}_{xxx}|^{2})dxdt].
\end{array}
\end{array}
\end{eqnarray*}
Then
\begin{eqnarray*}
\begin{array}{l}
\begin{array}{llll}
\displaystyle E\int_{Q\setminus
Q^{I_{0}}}(\lambda^{7}\mu^{7}\widehat{\varphi}^{7}|\widehat{u}|^{2}
+\lambda^{5}\mu^{5}\widehat{\varphi}^{5}|\widehat{u}_{x}|^{2}+\lambda^{3}\mu^{3}\widehat{\varphi}^{3}|\widehat{u}_{xx}|^{2}+\lambda\mu\widehat{\varphi} |\widehat{u}_{xxx}|^{2})dxdt\\
+\displaystyle E\int_{Q^{I_{0}}}(\lambda^{7}\mu^{7}\widehat{\varphi}^{7}|\widehat{u}|^{2}
+\lambda^{5}\mu^{5}\widehat{\varphi}^{5}|\widehat{u}_{x}|^{2}+\lambda^{3}\mu^{3}\widehat{\varphi}^{3}|\widehat{u}_{xx}|^{2}+\lambda\mu\widehat{\varphi} |\widehat{u}_{xxx}|^{2})dxdt
\\\leq C_{1}(\widehat{\psi}) E\displaystyle \int_{Q}(\lambda^{4}\mu^{4}\widehat{\varphi}^{4}\widehat{\theta}^{2}|g|^{2}
+\lambda^{2}\mu^{2}\widehat{\varphi}^{2}\widehat{\theta}^{2}|g_{x}|^{2}+\widehat{\theta}^{2}|g_{xx}|^{2}+\widehat{\theta}^{2}|f|^{2})dxdt\\
+\displaystyle E\int_{Q^{I_{0}}}(\lambda^{7}\mu^{7}\widehat{\varphi}^{7}|\widehat{u}|^{2}
+\lambda^{5}\mu^{5}\widehat{\varphi}^{5}|\widehat{u}_{x}|^{2}+\lambda^{3}\mu^{3}\widehat{\varphi}^{3}|\widehat{u}_{xx}|^{2}+\lambda\mu\widehat{\varphi} |\widehat{u}_{xxx}|^{2})dxdt,
\end{array}
\end{array}
\end{eqnarray*}
and thus
\begin{eqnarray*}
\begin{array}{l}
\begin{array}{llll}
\displaystyle E\int_{Q}(\lambda^{7}\mu^{7}\widehat{\varphi}^{7}|\widehat{u}|^{2}
+\lambda^{5}\mu^{5}\widehat{\varphi}^{5}|\widehat{u}_{x}|^{2}+\lambda^{3}\mu^{3}\widehat{\varphi}^{3}|\widehat{u}_{xx}|^{2}+\lambda\mu\widehat{\varphi} |\widehat{u}_{xxx}|^{2})dxdt
\\\leq C E[\displaystyle \int_{Q}(\lambda^{4}\mu^{4}\widehat{\varphi}^{4}\theta
^{2}|g|^{2}
+\lambda^{2}\mu^{2}\widehat{\varphi}^{2}\widehat{\theta}^{2}|g_{x}|^{2}+\widehat{\theta}^{2}|g_{xx}|^{2}+\widehat{\theta}^{2}|f|^{2})dxdt\\
+\displaystyle \int_{Q^{I_{0}}}(\lambda^{7}\mu^{7}\widehat{\varphi}^{7}|\widehat{u}|^{2}
+\lambda^{5}\mu^{5}\widehat{\varphi}^{5}|\widehat{u}_{x}|^{2}+\lambda^{3}\mu^{3}\widehat{\varphi}^{3}|\widehat{u}_{xx}|^{2}+\lambda\mu\widehat{\varphi} |\widehat{u}_{xxx}|^{2})]dxdt,
\end{array}
\end{array}
\end{eqnarray*}
from which it holds that
\begin{eqnarray*}
\begin{array}{l}
\begin{array}{llll}
\displaystyle E\int_{Q}(\lambda^{7}\widehat{\varphi}^{7}|\widehat{u}|^{2}+\lambda^{5}\widehat{\varphi}^{5}|\widehat{u}_{x}|^{2}
+\lambda^{3}\widehat{\varphi}^{3}|\widehat{u}_{xx}|^{2}+\lambda\widehat{\varphi} |\widehat{u}_{xxx}|^{2})dxdt
\\\leq C(\mu)E[\displaystyle \int_{Q}(\lambda^{4}\widehat{\varphi}^{4}\widehat{\theta}^{2}|g|^{2}
+\lambda^{2}\widehat{\varphi}^{2}\widehat{\theta}^{2}|g_{x}|^{2}+\widehat{\theta}^{2}|g_{xx}|^{2}+\widehat{\theta}^{2}f^{2})dxdt
\\~~~~~~+\displaystyle \int_{Q^{I_{0}}}(\lambda^{7}\widehat{\varphi}^{7}|\widehat{u}|^{2}
+\lambda^{5}\widehat{\varphi}^{5}|\widehat{u}_{x}|^{2}+\lambda^{3}\widehat{\varphi}^{3}|\widehat{u}_{xx}|^{2}+\lambda\widehat{\varphi} |\widehat{u}_{xxx}|^{2})]dxdt.
\end{array}
\end{array}
\end{eqnarray*}
Returning $\widehat{u}$ to $\widehat{\theta} y,$ we can obtain (\ref{12}).

\textbf{Step 2.} We shall eliminate the terms
$\displaystyle E\int_{Q^{I_{0}}} \lambda^{3}\widehat{\varphi}^{3}
\widehat{\theta}^{2}|y_{xx}|^{2}dxdt$ and $E\displaystyle \int_{Q^{I_{0}}}
\lambda^{5}\widehat{\varphi}^{5} \widehat{\theta}^{2}|y_{x}|^{2}dxdt$ in the right side of
(\ref{12}). Namely, we have (\ref{7}).
\par
Indeed, by the interpolation inequality, we obtain that for any
$\varepsilon>0,$
$$\int_{I_{0}}|(\widehat{\theta} y)_{x}|^{2}dx\leq \varepsilon \int_{I_{0}}|(\widehat{\theta} y)_{xx}|^{2}dx+\frac{C}{\varepsilon}\int_{I_{0}}|\widehat{\theta}y|^{2}dx,$$
where $C$ depends only on $I_{0}.$ Take $\varepsilon$ as
$\frac{\varepsilon_{1}}{2}(\frac{\lambda}{t(T-t)})^{-2}$ in above inequality,
where $\varepsilon_{1}$ will be fixed later. It holds that
\begin{eqnarray*}
\begin{array}{l}
\begin{array}{llll}
\displaystyle \int_{I_{0}}\widehat{\theta}^{2} |y_{x}|^{2}dx &\leq&
\displaystyle \varepsilon_{1}(\frac{\lambda}{t(T-t)})^{-2}
\int_{I_{0}}|(\widehat{\theta}
y)_{xx}|^{2}dx+\frac{C}{\varepsilon_{1}(\frac{\lambda}{t(T-t)})^{-2}}\int_{I_{0}}|\widehat{\theta}y|^{2}dx
\\&&\displaystyle+2\int_{I_{0}}\widehat{\theta}_{x}^{2}|y|^{2}dx.
\end{array}
\end{array}
\end{eqnarray*}
Choosing appropriate $\varepsilon_{1},$ we deduce that
\begin{eqnarray*}\int_{I_{0}}\widehat{\theta}^{2} |y_{x}|^{2}dx\leq \varepsilon \int_{I_{0}}\lambda^{-2}t^{2}(T-t)^{2}\widehat{\theta}^{2} |y_{xx}|^{2}dx
+C\int_{I_{0}}\lambda^{2}t^{-2}(T-t)^{-2}\widehat{\theta}^{2} |y|^{2}dx.
\end{eqnarray*}
\par
Noting that there exist two positive constants $N_{3}$ and $K_{3}$
such that
\begin{eqnarray*}\frac{N_{3}}{t(T-t)}\leq \widehat{\varphi} \leq
\frac{K_{3}}{t(T-t)},
\end{eqnarray*}
 thus, we can obtain that
\begin{eqnarray}\label{13}
\int_{Q^{I_{0}}}\lambda^{5}\widehat{\varphi}^{5}\widehat{\theta}^{2}|y_{x}|^{2}dxdt \leq
\varepsilon
\int_{Q^{I_{0}}}\lambda^{3}\widehat{\varphi}^{3}\widehat{\theta}
^{2}|y_{xx}|^{2}dxdt+C\int_{Q^{I_{0}}}\lambda^{7}\widehat{\varphi}^{7}\widehat{\theta}^{2}|y|^{2}dxdt,
\end{eqnarray}
by the same way, we have
\begin{eqnarray}\label{14}
\int_{Q^{I_{0}}}\lambda^{3}\widehat{\varphi}^{3}\widehat{\theta}^{2}|y_{xx}|^{2}dxdt \leq
\varepsilon
\int_{Q^{I_{0}}}\lambda\widehat{\varphi}\widehat{\theta}^{2}|y_{xxx}|^{2}dxdt+C\int_{Q^{I_{0}}}\lambda^{7}\widehat{\varphi}^{7}\widehat{\theta}^{2}|y|^{2}dxdt.
\end{eqnarray}
According to (\ref{13})-(\ref{14}) and (\ref{12}), we can obtain
(\ref{7}).

\subsection{Proof of Corollary \ref{C4}}
\par
According to (\ref{7}), we have
\begin{eqnarray*}
\begin{array}{l}
\begin{array}{llll}
~~~\displaystyle E\int_{Q}(\lambda\widehat{\varphi}
\widehat{\theta}^{2}|y_{xxx}|^{2}+\lambda^{3}\widehat{\varphi}^{3}\widehat{\theta}^{2}|y_{xx}|^{2}+\lambda^{5}\widehat{\varphi}^{5}\widehat{\theta}^{2}|y_{x}|^{2}+\lambda^{7}\widehat{\varphi}^{7}\widehat{\theta}^{2}|y|^{2})dxdt
\\\leq\displaystyle
C[E\int_{Q^{I_{0}}}(\lambda\widehat{\varphi}
\widehat{\theta}^{2}|y_{xxx}|^{2}+\lambda^{7}\widehat{\varphi}^{7}\widehat{\theta}^{2}|y|^{2}
)dxdt
\\+\displaystyle E\int_{Q}(\lambda^{4}\mu^{4}\widehat{\varphi}^{4}\widehat{\theta}^{2}|by+g|^{2}
+\lambda^{2}\mu^{2}\widehat{\varphi}^{2}\widehat{\theta}^{2}|(by+g)_{x}|^{2}+\widehat{\theta}^{2}|(by+g)_{xx}|^{2}+\widehat{\theta}^{2}|ay+f|^{2})dxdt].
\end{array}
\end{array}
\end{eqnarray*}
If we take $\lambda\geq C(a,b, T),$ where $C(a,b, T)$ is large enough, it follows that
\begin{eqnarray*}
\begin{array}{l}
\begin{array}{llll}
~~~\displaystyle E\int_{Q}(\lambda\widehat{\varphi}
\widehat{\theta}^{2}|y_{xxx}|^{2}+\lambda^{3}\widehat{\varphi}^{3}\widehat{\theta}^{2}|y_{xx}|^{2}+\lambda^{5}\widehat{\varphi}^{5}\widehat{\theta}^{2}|y_{x}|^{2}+\lambda^{7}\widehat{\varphi}^{7}\widehat{\theta}
^{2}|y|^{2})dxdt
\\\leq\displaystyle
C\Big[E\int_{Q^{I_{0}}}(\lambda\widehat{\varphi}
\widehat{\theta}^{2}|y_{xxx}|^{2}+\lambda^{7}\widehat{\varphi}^{7}\widehat{\theta}^{2}|y|^{2})dxdt
\\\displaystyle +E\int_{Q}(\lambda^{4}\widehat{\varphi}^{4}\widehat{\theta}^{2}|g|^{2}
+\lambda^{2}\widehat{\varphi}^{2}\widehat{\theta}^{2}|g_{x}|^{2}+\widehat{\theta}^{2}|g_{xx}|^{2}+\widehat{\theta}^{2}|f|^{2})dxdt\Big].
\end{array}
\end{array}
\end{eqnarray*}
By means of the definitions of $\widehat{\varphi}$ and $\widehat{\theta},$ it holds that
\begin{eqnarray*}
\begin{array}{l}
\begin{array}{llll}
\displaystyle E\int_{Q}(\lambda\widehat{\varphi}
\widehat{\theta}^{2}|y_{xxx}|^{2}+\lambda^{3}\widehat{\varphi}^{3}\widehat{\theta}^{2}|y_{xx}|^{2}+\lambda^{5}\widehat{\varphi}^{5}\widehat{\theta}^{2}|y_{x}|^{2}+\lambda^{7}\widehat{\varphi}^{7}\widehat{\theta}
^{2}|y|^{2})dxdt
\\\geq C\min\limits_{x\in \overline{I}}(\widehat{\varphi}(x,\frac{T}{2})\widehat{\theta}^{2}(x,\frac{T}{4}))\displaystyle E\int_{\frac{T}{4}}^{\frac{3T}{4}}\int_{I}(|y_{xxx}|^{2}+|y_{xx}|^{2}+|y_{x}|^{2}+|y|^{2})dxdt,
\\
\displaystyle
E\int_{Q^{I_{0}}}(\lambda\widehat{\varphi}\widehat{\theta}^{2}|y_{xxx}|^{2}+\lambda^{7}\widehat{\varphi}^{7}\widehat{\theta}^{2}|y|^{2})dxdt
\\\leq \displaystyle C\max_{(x,t)\in \overline{Q}}(\widehat{\varphi}^{7}(x,t)\widehat{\theta}^{2}(x,t))E\int_{Q^{I_{0}}}(|y_{xxx}|^{2}+|y|^{2})dxdt
\end{array}
\end{array}
\end{eqnarray*}
and
\begin{eqnarray*}
\begin{array}{l}
\begin{array}{llll}
\displaystyle E\Big[\int_{Q}(\lambda^{4}\widehat{\varphi}^{4}\widehat{\theta}^{2}|g|^{2}
+\lambda^{2}\widehat{\varphi}^{2}\widehat{\theta}^{2}|g_{x}|^{2}+\widehat{\theta}^{2}|g_{xx}|^{2}+\widehat{\theta}^{2}|f|^{2})dxdt\Big]
\\\leq \displaystyle C\max_{(x,t)\in \overline{Q}}(\widehat{\varphi}^{4}(x,t)\widehat{\theta}^{2}(x,t))E\int_{Q}(|g|^{2}
+|g_{x}|^{2}+|g_{xx}|^{2}+|f|^{2})dxdt.
\end{array}
\end{array}
\end{eqnarray*}
In view of  the above equalities and (\ref{7}), there holds
\begin{eqnarray}\label{1.36}
\begin{array}{l}
\begin{array}{llll}
\displaystyle E\int_{\frac{T}{4}}^{\frac{3T}{4}}\int_{I}(|y_{xxx}|^{2}+|y_{xx}|^{2}+|y_{x}|^{2}+|y|^{2})dxdt
\\\leq
\displaystyle C\frac{\max_{(x,t)\in \overline{Q}}(\widehat{\varphi}^{7}(x,t)\widehat{\theta}^{2}(x,t))}{\min\limits_{x\in \overline{I}}(\widehat{\varphi}(x,\frac{T}{2})\widehat{\theta}^{2}(x,\frac{T}{4}))}
\Big[E\int_{Q^{I_{0}}}(|y_{xxx}|^{2}+|y|^{2})dxdt
\\\displaystyle ~~~~~~~~~~~~~~~~+E\int_{Q}(|g|^{2}
+|g_{x}|^{2}+|g_{xx}|^{2}+|f|^{2})\Big].
\end{array}
\end{array}
\end{eqnarray}
It follows from (\ref{1.39}) with $s=3$ and (\ref{1.36}) that
\begin{eqnarray*}
\begin{array}{l}
\begin{array}{llll}
\displaystyle \frac{T}{2}E\|y_{0}\|_{X_{3}}^{2}
\\\displaystyle\leq C [E\int_{\frac{T}{4}}^{\frac{3T}{4}}\|y(t)\|_{X_{3}}^{2}dt+E\int_{0}^{T}(\|f(t)\|_{X_{3}}^{2}+\|g(t)\|_{X_{3}}^{2})dt]
\\\displaystyle\leq C\Big[E\int_{Q^{I_{0}}}(|y_{xxx}|^{2}+|y|^{2})dxdt+E\int_{0}^{T}(\|f(t)\|_{X_{3}}^{2}+\|g(t)\|_{X_{3}}^{2})dt\Big],
\end{array}
\end{array}
\end{eqnarray*}
namely, it holds that
\begin{eqnarray*}
\begin{array}{l}
\begin{array}{llll}
\displaystyle \|y_{0}\|_{X_{3}}^{2}
\leq C\Big[E\int_{Q^{I_{0}}}(|y_{xxx}|^{2}+|y|^{2})dxdt++E\int_{0}^{T}(\|f(t)\|_{X_{3}}^{2}+\|g(t)\|_{X_{3}}^{2})dt\Big],
\end{array}
\end{array}
\end{eqnarray*}
this implies (\ref{1.40}).
\subsection{Proof of Corollary \ref{C1}}
Taking $f=g=0$ in (\ref{1.40}) and considering $y\equiv 0$ in $Q^{I_{0}}$ $P$-a.s., we have
\begin{eqnarray*}
\|y_{0}\|_{X_{3}}\leq 0,
\end{eqnarray*}
thus, $y_{0}\equiv 0$ in $I$ $P$-a.s., this implies $y\equiv 0$ in $Q$ $P$-a.s.

\section{Proof of Theorem \ref{T3}, Corollary \ref{C5} and  Corollary \ref{C2}}
\subsection{Proof of Theorem \ref{T3}}
Proof of Theorem \ref{T3} is similar to Proof of Theorem \ref{T1}, we give a sketch of it.
\par
Indeed,
applying Corollary \ref{C3} with $l=\widetilde{l},$ then we have $\theta=\widetilde{\theta},u=\widetilde{u}=\widetilde{\theta}y$
and
\begin{eqnarray}\label{1.46}
\begin{array}{l}
\begin{array}{llll}
\displaystyle E\int_{Q}\Big\{(|\widetilde{u}|^{2}\{\cdot\}+|\widetilde{u}_{x}|^{2}\{\cdot\}+|\widetilde{u}_{xx}|^{2}\{\cdot\}+|\widetilde{u}_{xxx}|^{2}\{\cdot\})dt \\~~~~~~~~~~~~~~~~~~~+(\{\cdot\}_{x}+\{\cdot\}_{xx}+\{\cdot\}_{xxx}+\{\cdot\}_{xxxx})dt+d\widetilde{M}\Big\}dx
\\\leq-\displaystyle E\int_{Q}\Big\{\widetilde{u}\overline{\widetilde{u}}_{x}[(-\frac{i}{2})(C_{1t}-B_{2xt}+C_{3xxt})+C_{1}D_{0}]dt
\\~~~~~~~~~~~~+\widetilde{u}_{x}\overline{\widetilde{u}}[\frac{i}{2}(C_{1t}-B_{2xt}+C_{3xxt})+C_{1}\overline{D_{0}}]dt+(\widetilde{u}_{x}\overline{\widetilde{u}}_{xx}-\widetilde{u}_{xx}\overline{\widetilde{u}}_{x})(\frac{i}{2}C_{3t})dt
\\~~~~~~~~~~~~+B_{2}D_{0}\widetilde{u}\overline{\widetilde{u}}_{xx}dt+B_{2}\overline{D_{0}}\widetilde{u}_{xx}\overline{\widetilde{u}}dt
+C_{3}D_{0}\widetilde{u}\overline{\widetilde{u}}_{xxx}dt+C_{3}\overline{D_{0}}\widetilde{u}_{xxx}\overline{\widetilde{u}}dt
\\~~~~~~~~~~~~+(d\widetilde{u}d\overline{\widetilde{u}}_{x}-d\widetilde{u}_{x}d\overline{\widetilde{u}})(-\frac{i}{2})(C_{1}-B_{2x}+C_{3xx})+(d\widetilde{u}_{x}d\overline{\widetilde{u}}_{xx}-d\widetilde{u}_{xx}d\overline{\widetilde{u}}_{x})(\frac{i}{2}C_{3})\Big\}dx
\\+\displaystyle E\int_{Q}\widetilde{\theta}^{2}|f|^{2}dxdt.
\end{array}
\end{array}
\end{eqnarray}
where $\{\cdot\}_{x},\{\cdot\}_{xx},\{\cdot\}_{xxx},\{\cdot\}_{xxxx}$ and $\widetilde{M}$ are the same as in Theorem \ref{T2} with $u=\widetilde{u}.$
\par
The same argument in Proof of Theorem \ref{T1} shows that
\begin{eqnarray}\label{1.26}
\begin{array}{l}
\begin{array}{llll}
|\widetilde{u}|^{2}\{\cdot\}=16\lambda^{7}\mu^{8}\widetilde{\varphi}^{7}\widetilde{\psi}_{x}^{8}|\widetilde{u}|^{2}+\widetilde{R}_{0}|\widetilde{u}|^{2},
\\
|\widetilde{u}_{x}|^{2}\{\cdot\}=80\lambda^{5}\mu^{6}\widetilde{\varphi}^{5}\widetilde{\psi}_{x}^{6}|\widetilde{u}_{x}|^{2}+\widetilde{R}_{1}|\widetilde{u}_{x}|^{2},
\\|\widetilde{u}_{xx}|^{2}\{\cdot\}=16\lambda^{3}\mu^{4}\widetilde{\varphi}^{3}\widetilde{\psi}_{x}^{4}|\widetilde{u}_{xx}|^{2}+\widetilde{R}_{2}|\widetilde{u}_{xx}|^{2},
\\|\widetilde{u}_{xxx}|^{2}\{\cdot\}=16\lambda\mu^{2}\widetilde{\varphi}\widetilde{\psi}_{x}^{2}|\widetilde{u}_{xxx}|^{2}+\widetilde{R}_{3}|\widetilde{u}_{xxx}|^{2},
\end{array}
\end{array}
\end{eqnarray}
where
\begin{eqnarray*}
\begin{array}{l}
\begin{array}{llll}
|\widetilde{R}_{0}|\leq C\lambda^{7}\mu^{7}\widetilde{\varphi}^{7},
\\|\widetilde{R}_{1}|\leq C\lambda^{5}\mu^{5}\widetilde{\varphi}^{5},
\\|\widetilde{R}_{2}|\leq C\lambda^{3}\mu^{3}\widetilde{\varphi}^{3},
\\|\widetilde{R}_{3}|\leq C\lambda\mu\widetilde{\varphi}.
\end{array}
\end{array}
\end{eqnarray*}
\par
Noting that
\begin{eqnarray*}
\begin{array}{l}
\begin{array}{llll}
~~~\displaystyle E\int_{Q}(\{\cdot\}_{x}+\{\cdot\}_{xx}+\{\cdot\}_{xxx}+\{\cdot\}_{xxxx})dxdt
\\=\displaystyle E\int_{Q}\{|\widetilde{u}_{xx}|^{2}(-20\lambda^{3}\mu^{3}\widetilde{\varphi}^{3}\widetilde{\psi}^{3}_{x}+\widetilde{r}_{1})+|\widetilde{u}_{xxx}|^{2}(-4\lambda\mu\widetilde{\varphi}\widetilde{\psi}_{x})
\\
~~~+\widetilde{u}_{xxx}\overline{\widetilde{u}}_{xx}(-6\lambda\mu^{2}\widetilde{\varphi}\widetilde{\psi}^{2}_{x}+\widetilde{r}_{2})+\widetilde{u}_{xx}\overline{\widetilde{u}}_{xxx}(-6\lambda\mu^{2}\widetilde{\varphi}\widetilde{\psi}^{2}_{x}+\widetilde{r}_{2})\}_{x}dxdt
\\\triangleq \widetilde{V}(1)-\widetilde{V}(0),

\end{array}
\end{array}
\end{eqnarray*}
where
\begin{eqnarray*}
\begin{array}{l}
\begin{array}{llll}
|\widetilde{r}_{1}|\leq C\lambda^{2}\mu^{3}\widetilde{\varphi}^{2},
\\|\widetilde{r}_{2}|\leq C\lambda\mu\widetilde{\varphi}.
\end{array}
\end{array}
\end{eqnarray*}
By the same method in Proof of Theorem \ref{T1}, we can know
\begin{eqnarray}\label{1.23}
\begin{array}{l}
\begin{array}{llll}
\widetilde{V}(1)\geq0.
\end{array}
\end{array}
\end{eqnarray}
\par

It is a straightforward calculation to show that
\begin{eqnarray}
\begin{array}{l}
\begin{array}{llll}
\displaystyle E\int_{0}^{T}d\widetilde{M}=0,
\end{array}
\end{array}
\end{eqnarray}
\begin{eqnarray}\label{1.25}
\begin{array}{l}
\begin{array}{llll}
|\widetilde{u}\overline{\widetilde{u}}_{x}[(-\frac{i}{2})(C_{1t}-B_{2xt}+C_{3xxt})+C_{1}D_{0}]+\widetilde{u}_{x}\overline{\widetilde{u}}[-\frac{i}{2}(C_{1t}-B_{2xt}+C_{3xxt})+C_{1}\overline{D_{0}}]
\\+(\widetilde{u}_{x}\overline{\widetilde{u}}_{xx}-\widetilde{u}_{xx}\overline{\widetilde{u}}_{x})(\frac{i}{2}C_{3t})+B_{2}D_{0}\widetilde{u}\overline{\widetilde{u}}_{xx}+B_{2}\overline{D_{0}}\widetilde{u}_{xx}\overline{\widetilde{u}}
+C_{3}D_{0}\widetilde{u}\overline{\widetilde{u}}_{xxx}+C_{3}\overline{D_{0}}\widetilde{u}_{xxx}\overline{\widetilde{u}}|
\\\leq C(\lambda^{7}\mu^{7}\widetilde{\varphi}^{7}|\widetilde{u}|^{2}
+\lambda^{5}\mu^{5}\widetilde{\varphi}^{5}|\widetilde{u}_{x}|^{2}+\lambda^{3}\mu^{3}\widetilde{\varphi}^{3}|\widetilde{u}_{xx}|^{2}+\lambda\mu\widetilde{\varphi} |\widetilde{u}_{xxx}|^{2})
\end{array}
\end{array}
\end{eqnarray}
and
\begin{eqnarray}\label{1.24}
\begin{array}{l}
\begin{array}{llll}
\Big|\displaystyle E\int_{Q}[(d\widetilde{u}d\overline{\widetilde{u}}_{x}-d\widetilde{u}_{x}d\overline{\widetilde{u}})(-\frac{i}{2})(C_{1}-B_{2x}+C_{3xx})
+(d\widetilde{u}_{x}d\overline{\widetilde{u}}_{xx}-d\widetilde{u}_{xx}d\overline{\widetilde{u}}_{x})(\frac{i}{2}C_{3})]dx\Big|
\\\leq C\displaystyle E\int_{Q}(\lambda^{3}\mu^{3}\widetilde{\varphi}^{3}\widetilde{\theta}^{2}|g||g_{x}|
+\lambda^{2}\mu^{2}\widetilde{\varphi}^{2}\widetilde{\theta}^{2}|g||g_{xx}|+\lambda\mu\widetilde{\varphi}\widetilde{\theta}^{2}|g_{x}||g_{xx}|)dxdt
\\\leq C\displaystyle E\int_{Q}(\lambda^{4}\mu^{4}\widetilde{\varphi}^{4}\widetilde{\theta}^{2}|g|^{2}
+\lambda^{2}\mu^{2}\widetilde{\varphi}^{2}\widetilde{\theta}^{2}|g_{x}|^{2}+\widetilde{\theta}^{2}|g_{xx}|^{2})dxdt.
\end{array}
\end{array}
\end{eqnarray}
From (\ref{1.46})-(\ref{1.24}), we can obtain
that
\begin{eqnarray*}
\begin{array}{l}
\begin{array}{llll}
\displaystyle E\int_{Q}(\lambda^{7}\mu^{8}\widetilde{\varphi}^{7}\widetilde{\psi}_{x}^{8}|\widetilde{u}|^{2}+\lambda^{5}\mu^{6}\widetilde{\varphi}^{5}\widetilde{\psi}_{x}^{6}|\widetilde{u}_{x}|^{2}
+\lambda^{3}\mu^{4}\widetilde{\varphi}^{3}\widetilde{\psi}_{x}^{4}|\widetilde{u}_{xx}|^{2}+\lambda\mu^{2}\widetilde{\varphi}\widetilde{\psi}_{x}^{2}|\widetilde{u}_{xxx}|^{2})dxdt
\\\leq \widetilde{V}(0)+C\displaystyle E\int_{Q}(\lambda^{4}\mu^{4}\widetilde{\varphi}^{4}\widetilde{\theta}^{2}|g|^{2}
+\lambda^{2}\mu^{2}\widetilde{\varphi}^{2}\widetilde{\theta}^{2}|g_{x}|^{2}+\widetilde{\theta}^{2}|g_{xx}|^{2}+\widetilde{\theta}^{2}|f|^{2}
\\~~~~~~~~+\lambda^{7}\mu^{7}\widetilde{\varphi}^{7}|\widetilde{u}|^{2}
+\lambda^{5}\mu^{5}\widetilde{\varphi}^{5}|\widetilde{u}_{x}|^{2}+\lambda^{3}\mu^{3}\widetilde{\varphi}^{3}|\widetilde{u}_{xx}|^{2}+\lambda\mu\widetilde{\varphi} |\widetilde{u}_{xxx}|^{2})dxdt.
\end{array}
\end{array}
\end{eqnarray*}
Recall that $|\widetilde{\psi}_{x}|>0$ in $\overline{I},$ it
follows that
\begin{eqnarray*}
\begin{array}{l}
\begin{array}{llll}
\displaystyle E\int_{Q}(\lambda^{7}\mu^{8}\widetilde{\varphi}^{7}|\widetilde{u}|^{2}+\lambda^{5}\mu^{6}\widetilde{\varphi}^{5}|\widetilde{u}_{x}|^{2}
+\lambda^{3}\mu^{4}\widetilde{\varphi}^{3}|\widetilde{u}_{xx}|^{2}+\lambda\mu^{2}\widetilde{\varphi} |\widetilde{u}_{xxx}|^{2})dxdt
\\\leq C(\widetilde{\psi})\displaystyle \Big[\widetilde{V}(0)+E\int_{Q}(\lambda^{4}\mu^{4}\widetilde{\varphi}^{4}\widetilde{\theta}^{2}|g|^{2}
+\lambda^{2}\mu^{2}\widetilde{\varphi}^{2}\widetilde{\theta}^{2}|g_{x}|^{2}+\widetilde{\theta}^{2}|g_{xx}|^{2}+\widetilde{\theta}^{2}|f|^{2}
\\~~~~~~~~+\lambda^{7}\mu^{7}\widetilde{\varphi}^{7}|\widetilde{u}|^{2}
+\lambda^{5}\mu^{5}\widetilde{\varphi}^{5}|\widetilde{u}_{x}|^{2}+\lambda^{3}\mu^{3}\widetilde{\varphi}^{3}|\widetilde{u}_{xx}|^{2}+\lambda\mu\widetilde{\varphi} |\widetilde{u}_{xxx}|^{2})dxdt\Big],
\end{array}
\end{array}
\end{eqnarray*}
from which if we choose $\mu_{0}= C(\widetilde{\psi})+1,$ then it holds that
\begin{eqnarray*}
\begin{array}{l}
\begin{array}{llll}
\displaystyle E\int_{Q}(\lambda^{7}\mu^{7}\widetilde{\varphi}^{7}|\widetilde{u}|^{2}
+\lambda^{5}\mu^{5}\widetilde{\varphi}^{5}|\widetilde{u}_{x}|^{2}+\lambda^{3}\mu^{3}\widetilde{\varphi}^{3}|\widetilde{u}_{xx}|^{2}+\lambda\mu\widetilde{\varphi} |\widetilde{u}_{xxx}|^{2})dxdt
\\\leq C_{1}(\widetilde{\psi})\Big[\widetilde{V}(0)+E\displaystyle\int_{Q}(\lambda^{4}\mu^{4}\widetilde{\varphi}^{4}\widetilde{\theta}^{2}|g|^{2}
+\lambda^{2}\mu^{2}\widetilde{\varphi}^{2}\widetilde{\theta}^{2}|g_{x}|^{2}+\widetilde{\theta}^{2}|g_{xx}|^{2}+\widetilde{\theta}^{2}|f|^{2})dxdt\Big].
\end{array}
\end{array}
\end{eqnarray*}
Then
\begin{eqnarray*}
\begin{array}{l}
\begin{array}{llll}
\displaystyle E\int_{Q}(\lambda^{7}\mu^{7}\widetilde{\varphi}^{7}|\widetilde{u}|^{2}
+\lambda^{5}\mu^{5}\widetilde{\varphi}^{5}|\widetilde{u}_{x}|^{2}+\lambda^{3}\mu^{3}\widetilde{\varphi}^{3}|\widetilde{u}_{xx}|^{2}+\lambda\mu\widetilde{\varphi} |\widetilde{u}_{xxx}|^{2})dxdt
\\\leq C\Big[E\displaystyle
\int_{0}^{T}(\lambda^{3}\mu^{3}\widetilde{\varphi}^{3}(0,t )|\widetilde{u}_{xx}(0,t)|^{2}+\lambda\mu\widetilde{\varphi}(0,t)|\widetilde{u}_{xxx}(0,t)|^{2})dt
\\+E\displaystyle\int_{Q}(\lambda^{4}\mu^{4}\widetilde{\varphi}^{4}\widetilde{\theta}^{2}|g|^{2}
+\lambda^{2}\mu^{2}\widetilde{\varphi}^{2}\widetilde{\theta}^{2}|g_{x}|^{2}+\widetilde{\theta}^{2}|g_{xx}|^{2}+\widetilde{\theta}^{2}|f|^{2})dxdt\Big],
\end{array}
\end{array}
\end{eqnarray*}
from which it holds that
\begin{eqnarray*}
\begin{array}{l}
\begin{array}{llll}
\displaystyle E\int_{Q}(\lambda^{7}\widetilde{\varphi}^{7}|\widetilde{u}|^{2}
+\lambda^{5}\widetilde{\varphi}^{5}|\widetilde{u}_{x}|^{2}+\lambda^{3}\widetilde{\varphi}^{3}|\widetilde{u}_{xx}|^{2}+\lambda\widetilde{\varphi} |\widetilde{u}_{xxx}|^{2})dxdt
\\\leq C(\mu)\Big[E\displaystyle
\int_{0}^{T}(\lambda^{3}\widetilde{\varphi}^{3}(0,t)|\widetilde{u}_{xx}(0,t)|^{2}+\lambda\widetilde{\varphi}(0,t)|\widetilde{u}_{xxx}(0,t)|^{2})dt
\\+E\displaystyle\int_{Q}(\lambda^{4}\widetilde{\varphi}^{4}\widetilde{\theta}^{2}|g|^{2}
+\lambda^{2}\widetilde{\varphi}^{2}\widetilde{\theta}^{2}|g_{x}|^{2}+\widetilde{\theta}^{2}|g_{xx}|^{2}+\widetilde{\theta}^{2}|f|^{2})dxdt\Big].
\end{array}
\end{array}
\end{eqnarray*}
Returning $\widetilde{u}$ to $\widetilde{\theta} y,$ we can obtain (\ref{1.21}).

\subsection{Proof of Corollary \ref{C5}}
The proof of (\ref{1.43}) is similar to (\ref{1.40}). For completeness we give a sketch of it.
\par
According to (\ref{1.21}), we have
\begin{eqnarray*}
\begin{array}{l}
\begin{array}{llll}
~~~\displaystyle E\int_{Q}(\lambda\widetilde{\varphi}
\widetilde{\theta}^{2}y^{2}_{xxx}+\lambda^{3}\widetilde{\varphi}^{3}\widetilde{\theta}^{2}y^{2}_{xx}+\lambda^{5}\widetilde{\varphi}^{5}\widetilde{\theta}^{2}y^{2}_{x}+\lambda^{7}\widetilde{\varphi}^{7}\widetilde{\theta}^{2}y^{2})dxdt
\\\leq\displaystyle
C[E\displaystyle
\int_{0}^{T}(\lambda^{3}\widetilde{\varphi}^{3}(0,t)\widetilde{\theta}^{2}(0,t)y^{2}_{xx}(0,t)+\lambda\widetilde{\varphi}(0,t)\widetilde{\theta}^{2}(0,t)y^{2}_{xxx}(0,t))dt
\\+\displaystyle E\int_{Q}(\lambda^{4}\widetilde{\varphi}^{4}\widetilde{\theta}^{2}|by+g|^{2}
+\lambda^{2}\widetilde{\varphi}^{2}\widetilde{\theta}^{2}|(by+g)_{x}|^{2}+\widetilde{\theta}^{2}|(by+g)_{xx}|^{2}+\widetilde{\theta}^{2}|ay+f|^{2})dxdt].
\end{array}
\end{array}
\end{eqnarray*}
If we take $\lambda\geq C(a,b,T),$ where $C(a,b,T)$ is large enough, it follows that
\begin{eqnarray*}
\begin{array}{l}
\begin{array}{llll}
~~~\displaystyle E\int_{Q}(\lambda\widetilde{\varphi}
\widetilde{\theta}^{2}y^{2}_{xxx}+\lambda^{3}\widetilde{\varphi}^{3}\widetilde{\theta}^{2}y^{2}_{xx}+\lambda^{5}\widetilde{\varphi}^{5}\widetilde{\theta}^{2}y^{2}_{x}+\lambda^{7}\widetilde{\varphi}^{7}\widetilde{\theta}^{2}y^{2})dxdt
\\\leq\displaystyle
C\Big[E\displaystyle
\int_{0}^{T}(\lambda^{3}\widetilde{\varphi}^{3}(0,t )\widetilde{\theta}^{2}(0,t)y^{2}_{xx}(0,t)+\lambda\widetilde{\varphi}(0,t)\widetilde{\theta}^{2}(0,t)y^{2}_{xxx}(0,t))dt
\\+\displaystyle E\int_{Q}(\lambda^{4}\widehat{\varphi}^{4}\widetilde{\theta}^{2}|g|^{2}
+\lambda^{2}\widetilde{\varphi}^{2}\widetilde{\theta}^{2}|g_{x}|^{2}+\widetilde{\theta}^{2}|g_{xx}|^{2}+\widetilde{\theta}^{2}|f|^{2})dxdt\Big].
\end{array}
\end{array}
\end{eqnarray*}
By means of the definitions of $\widetilde{\varphi}$ and $\widetilde{\theta},$ it holds that
\begin{eqnarray*}
\begin{array}{l}
\begin{array}{llll}
\displaystyle E\int_{Q}(\lambda\widetilde{\varphi}
\widetilde{\theta}^{2}y^{2}_{xxx}+\lambda^{3}\widetilde{\varphi}^{3}\widetilde{\theta}^{2}y^{2}_{xx}+\lambda^{5}\widetilde{\varphi}^{5}\widetilde{\theta}^{2}y^{2}_{x}+\lambda^{7}\widetilde{\varphi}^{7}\theta
^{2}y^{2})dxdt
\\\geq C\min\limits_{x\in \overline{I}}(\widetilde{\varphi}(x,\frac{T}{2})\widetilde{\theta}^{2}(x,\frac{T}{4}))\displaystyle E\int_{\frac{T}{4}}^{\frac{3T}{4}}\int_{I}(y^{2}_{xxx}+y^{2}_{xx}+y^{2}_{x}+y^{2})dxdt,
\\
\displaystyle E\int_{0}^{T}(\lambda\widetilde{\varphi}(0,t)
\widetilde{\theta}^{2}(0,t)y^{2}_{xxx}(0,t)+\lambda^{3}\widetilde{\varphi}^{3}(0,t)\widetilde{\theta}^{2}(0,t)y_{xx}^{2}(0,t))dt
\\\leq \displaystyle C\max_{(x,t)\in \overline{Q}}(\widetilde{\varphi}^{3}(x,t)\widetilde{\theta}^{2}(x,t))E\int_{0}^{T}(y^{2}_{xxx}(0,t)+y_{xx}^{2}(0,t))dt,
\end{array}
\end{array}
\end{eqnarray*}
and
\begin{eqnarray*}
\begin{array}{l}
\begin{array}{llll}
\displaystyle E\Big[\int_{Q}(\lambda^{4}\widetilde{\varphi}^{4}\widetilde{\theta}^{2}|g|^{2}
+\lambda^{2}\widetilde{\varphi}^{2}\widetilde{\theta}^{2}|g_{x}|^{2}+\widetilde{\theta}^{2}|g_{xx}|^{2}+\widetilde{\theta}^{2}|f|^{2})dxdt\Big]
\\\leq \displaystyle C\max_{(x,t)\in \overline{Q}}(\widetilde{\varphi}^{4}(x,t)\widetilde{\theta}^{2}(x,t))E\int_{Q}(|g|^{2}
+|g_{x}|^{2}+|g_{xx}|^{2}+|f|^{2})dxdt.
\end{array}
\end{array}
\end{eqnarray*}
In view of  the above equalities and (\ref{1.21}), there holds that
\begin{eqnarray}\label{1.37}
\begin{array}{l}
\begin{array}{llll}
 E\int_{\frac{T}{4}}^{\frac{3T}{4}}\int_{I}(y^{2}_{xxx}+y^{2}_{xx}+y^{2}_{x}+y^{2})dxdt
\\\leq
\displaystyle C\frac{\max_{(x,t)\in \overline{Q}}(\widetilde{\varphi}^{4}(x,t)\widetilde{\theta}^{2}(x,t))}{\min_{x\in \overline{I}}(\widetilde{\varphi}(x,\frac{T}{2})\widetilde{\theta}^{2}(x,\frac{T}{4}))}
\Big[E\int_{0}^{T}(y^{2}_{xxx}(0,t)+y_{xx}^{2}(0,t))dt
+E\int_{Q}(|g|^{2}
+|g_{x}|^{2}+|g_{xx}|^{2}+|f|^{2})\Big].
\end{array}
\end{array}
\end{eqnarray}
It follows from (\ref{1.39}) with $s=3$ and (\ref{1.37}) that
\begin{eqnarray*}
\begin{array}{l}
\begin{array}{llll}
\displaystyle \frac{T}{2}E\|y_{0}\|_{X_{3}}^{2}
\\\displaystyle\leq C E\int_{\frac{T}{4}}^{\frac{3T}{4}}\|y(t)\|_{X_{3}}^{2}dt+C E\int_{0}^{T}(\|f(t)\|_{X_{3}}^{2}+\|g(t)\|_{X_{3}}^{2})dt
\\\displaystyle\leq C\Big[E\int_{0}^{T}(y_{xx}^{2}(0,t)+y_{xxx}^{2}(0,t))dt+E\int_{0}^{T}(\|f(t)\|_{X_{3}}^{2}+\|g(t)\|_{X_{3}}^{2})dt\Big],
\end{array}
\end{array}
\end{eqnarray*}
namely, it holds that
\begin{eqnarray*}
\begin{array}{l}
\begin{array}{llll}
\displaystyle \|y_{0}\|_{X_{3}}^{2}
\leq C\Big[E\int_{0}^{T}(y_{xx}^{2}(0,t)+y_{xxx}^{2}(0,t))dt+E\int_{0}^{T}(\|f(t)\|_{X_{3}}^{2}+\|g(t)\|_{X_{3}}^{2})dt\Big],
\end{array}
\end{array}
\end{eqnarray*}
this implies (\ref{1.43}).

\subsection{Proof of Corollary \ref{C2}}
Taking $f=g=0$ in (\ref{1.43}) and considering $y_{xx}(0,t)=y_{xxx}(0,t)\equiv0~~ {\rm{in}}~~ (0,T),~P-{\rm{a.s.}}$, we have
\begin{eqnarray*}
\|y_{0}\|_{X_{3}}\leq 0,
\end{eqnarray*}
thus, $y_{0}\equiv 0$ in $I$ $P$-a.s., this implies $y\equiv 0$ in $Q$ $P$-a.s.

\section{Proof of Theorem \ref{T5}}
\par The main idea in this part comes from \cite{0}.
\par
Since system (\ref{1.3}) is linear, we only need to show that the attainable set  at time $T>0$ with initial datum
$y_{0}=0$ is $L^{2}\mathcal (\Omega,\mathcal
{F}_{T},P; X_{3}^{\prime}),$ that is, for any $y_{1}\in L^{2}\mathcal (\Omega,\mathcal
{F}_{T},P; X_{3}^{\prime}),$ we can find controls
\begin{eqnarray*}
(u_{1},u_{2},g)\in L^{2}_{\mathcal{F}}(\Omega,L^{2}(0,T))\times L^{2}_{\mathcal{F}}(\Omega,L^{2}(0,T))\times L_{\mathcal {F}}^{2}(0,T; X_{3}^{\prime})
\end{eqnarray*}
such that the solution to the system (\ref{1.3}) with $y_{0}=0$ satisfies that $y(T)=y_{1}.$ We achieve this goal by duality argument.
\par
Let us set
\begin{eqnarray*}
\mathcal{W}=\{(z_{xx}(0,t),z_{xxx}(0,t),Z)~|~(z,Z)~{\rm{solves}} ~(\ref{15})~{\rm{with ~some}} ~z_{T}\in L^{2}\mathcal (\Omega,\mathcal
{F}_{T},P; X_{3})\}.
\end{eqnarray*}
Clearly, $\mathcal{W}$ is a linear subspace of $L^{2}_{\mathcal{F}}(\Omega,L^{2}(0,T))\times L^{2}_{\mathcal{F}}(\Omega,L^{2}(0,T))\times L_{\mathcal {F}}^{2}(0,T; X_{3}^{\prime}).$ Let us define a linear functional $\mathcal{L}$ on $\mathcal{W}$ as follows:
\begin{eqnarray*}
\mathcal{L}((z_{xx}(0,t),z_{xxx}(0,t),Z))=E(y_{1},\overline{z_{T}})_{X_{3}^{\prime},X_{3}}-E\int_{0}^{T}(f,\overline{z})_{X_{3}^{\prime},X_{3}}dt.
\end{eqnarray*}
From (\ref{1.29}) and Proposition \ref{P2} iii) with $s=3,h=0$, we see that $\mathcal{L}$ is a bounded linear functional on $\mathcal{W}$. By means of the Hahn-Banach Theorem, $\mathcal{L}$ can be extended to be a bounded linear functional on the space $L^{2}_{\mathcal{F}}(\Omega,L^{2}(0,T))\times L^{2}_{\mathcal{F}}(\Omega,L^{2}(0,T))\times L_{\mathcal {F}}^{2}(0,T;X_{3}^{\prime}).$
For simplicity, we still use $\mathcal{L}$ to denote this extension. Now, by the Riesz Representation Theorem, we know that there is a  random fields
\begin{eqnarray*}
(u_{1},u_{2},g)\in L^{2}_{\mathcal{F}}(\Omega,L^{2}(0,T))\times L^{2}_{\mathcal{F}}(\Omega,L^{2}(0,T))\times L_{\mathcal {F}}^{2}(0,T;X_{3}^{\prime})
\end{eqnarray*}
such that
\begin{eqnarray*}
\begin{array}{l}
\begin{array}{llll}
\displaystyle
E(y_{1},\overline{z}_{T})_{X_{3}^{\prime},X_{3}}+E\int_{0}^{T}i(f,\overline{z})_{X_{3}^{\prime},X_{3}}dt
\\~~~~~~~~~~~~~~~=\displaystyle E\int_{0}^{T}i(u_{1}(t)\overline{z}_{xxx}(0,t)+u_{2}(t)\overline{z}_{xx}(0,t))dt
+\displaystyle E\int_{0}^{T}(g,\overline{Z})_{X_{3}^{\prime},X_{3}}dt.
\end{array}
\end{array}
\end{eqnarray*}
\par
We claim that this  random fields $(u_{1},u_{2},g)$ is the control we need.
\par
In fact, from the definition of the solution to (\ref{1.3}), we have
\begin{eqnarray*}
\begin{array}{l}
\begin{array}{llll}
\displaystyle
\displaystyle
E(y(T),\overline{z}_{T})_{X_{3}^{\prime},X_{3}}
\\=\displaystyle E\int_{0}^{T}i(u_{1}(t)\overline{z}_{xxx}(0,t)+u_{2}(t)\overline{z}_{xx}(0,t))dt
\\~~~~~~~~~~~~~~~~~~~+\displaystyle E\int_{0}^{T}[-i(f,\overline{z})_{X_{3}^{\prime},X_{3}}+(g,\overline{Z})_{X_{3}^{\prime},X_{3}}]dt,
\end{array}
\end{array}
\end{eqnarray*}
thus, we have
\begin{eqnarray}\label{1.33}
E(y_{1},\overline{z}_{T})_{X_{3}^{\prime},X_{3}}=E(y(T),\overline{z}_{T})_{X_{3}^{\prime},X_{3}}.
\end{eqnarray}
Since $z_{T}$ can be arbitrary element in $L^{2}\mathcal (\Omega,\mathcal
{F}_{T},P; X_{3}),$ from the equality (\ref{1.33}), we get $y(T)=y_{1}.$

\noindent \footnotesize {\bf Acknowledgements.} \par I sincerely
 thank Professor Yong Li for many useful suggestions and help.\par
 \par

\baselineskip 9pt \renewcommand{\baselinestretch}{1.08}
\par
\par

\end{document}